\documentclass[11pt]{article}
\usepackage{amsmath}
\usepackage{amsthm}
\usepackage{amssymb}
\usepackage{enumitem}
\usepackage{hyperref}
\usepackage{verbatim}
\usepackage{graphicx}
\usepackage[usenames,dvipsnames,svgnames,table]{xcolor}
\usepackage{mathrsfs}  \usepackage{accents} \usepackage{mathtools}
\usepackage{longtable}
\usepackage{array}
\usepackage[T1]{fontenc}

\def\D{\mathrm{D}}

\usepackage[hmargin=2.4cm,vmargin=2.4cm]{geometry}
\def\a{\alpha}
\def\b{\beta}
\def\ga{\gamma}
\def\Ga{\Gamma}
\def\de{\delta}
\def\De{\Delta}
\def\ep{\epsilon}

\def\la{\lambda}

\def\si{\sigma}

\def\th{\theta}
\newcommand{\Th}{\Theta}

\def\varep{\varepsilon}

\def\BB{{\cal B}}

\def\DD{{\cal D}}

\def\II{{\cal I}}

\def\OO{{\cal O}}

\def\DD{{\cal D}}
\def\QQ{{\cal Q}}
\def\RR{{\cal R}}

\newcommand{\N}[0]{\mathbb{N}}
\newcommand{\F}[0]{\mathbb{F}}
\newcommand{\R}[0]{\mathbb{R}}
\newcommand{\Z}[0]{\mathbb{Z}}

\newcommand{\T}[0]{\mathbb{T}}

\newcommand{\supp}{\mathrm{supp} \,}
\newcommand{\suppt}{\mathrm{supp}_t \,}

\newcommand{\pa}{\partial}
\newcommand{\tsi}{\ost{\tau}^{-1}}
\newcommand{\fr}[2]{\frac{#1}{#2}}

\newcommand{\mat}[2]{\left[ \begin{array}{ #1} #2 \end{array} \right]}
\newcommand{\ALI}[1]{\begin{align*} #1 \end{align*}}

\newcommand{\lsm}[0]{\lesssim}

\newcommand{\pr}[0]{\partial}
\newcommand{\nb}{\nabla}

\newcommand{\co}[1]{\|#1\|_{C^0}}
\newcommand{\cda}[1]{\|#1\|_{\dot{C}^\a}}

\newcommand{\Ddt}[0]{\overline{D}_t}

\newcommand{\Dtdt}[0]{\tilde{D}_t}

\newcommand{\unl}{{\underline L}}
\newcommand{\hxi}[0]{\widehat{\Xi}}          \newcommand{\hn}[0]{\widehat{N}}

\newcommand{\plhxi}[0]{(\log \widehat{\Xi})}
\newcommand{\va}[0]{\vec{a}}

\newcommand{\vcb}[0]{\vec{b}}

\DeclareMathAlphabet{\mathpzc}{OT1}{pzc}{m}{it}
\newcommand{\hh}[0]{\mathpzc{h}}

\newcommand{\hq}[0]{\hat{q}}
\newcommand{\hc}[0]{\widehat{C}}

\newcommand{\tisum}{\overset{\sim}{\sum}}

\newcommand{\ali}[1]{ \begin{align} #1 \end{align} }

\def\XXint#1#2#3{{\setbox0=\hbox{$#1{#2#3}{\int}$}
     \vcenter{\hbox{$#2#3$}}\kern-.5\wd0}}

\newcommand{\calB}{\mathcal B}

\newcommand{\calE}{\mathcal E}

\newcommand{\calH}{\mathcal H}

\newcommand{\calR}{\mathcal R}

\newcommand{\ls}{\lesssim}

\newcommand{\ost}[1]{\accentset{\ast}{#1}}

\newcommand{\mrg}[1]{\mathring{#1}}
\newcommand{\ti}{\tilde}

\newtheorem{thm}{Theorem}\newtheorem{lem}{Lemma}[section]

\newtheorem{prop}{Proposition}[section]
\newtheorem{cor}{Corollary}[section]

\theoremstyle{definition}
\newtheorem{defn}{Definition}[section]

\theoremstyle{remark}
\newtheorem{rem}{Remark}

\title{ A proof of Onsager's Conjecture for the SQG equation }
\author{  Philip Isett\thanks{Department of Mathematics, Caltech. }, Sam Looi\thanks{Department of Mathematics, Caltech. }
}
\date{ }

\begin{document}
\maketitle

\begin{abstract}
    We construct solutions to the SQG equation that fail to conserve the Hamiltonian while having the maximal allowable regularity for this property to hold.  This result solves the generalized Onsager conjecture on the threshold regularity for Hamiltonian conservation for SQG. We also prove nonuniqueness for arbitrary smooth data and an h-principle result. \end{abstract}

\tableofcontents

\section{Introduction}

In this paper we are concerned with the surface quasi-geostrophic equation (SQG equation), which arises as an important model equation in geophysical fluid dynamics that has applications to both oceanic and meteorological flows \cite{sqgPhys,pedlosky2013geophysical}.  The SQG equation for an unknown scalar field $\th$ on a two-dimensional spatial domain has the form 
\ali{
\pr_t \th + \nb \cdot ( \th u ) &= 0, \qquad  u = T[\th] =  |\nb|^{-1} \nb^\perp \th,
}
where $|\nb| = \sqrt{-\De}$.  

SQG is an {\it active scalar equation}, so called since the velocity field advecting the scalar field depends at every time on the values of the scalar field.  The field $\th$ can represent either the temperature or surface buoyancy in a certain regime of stratified flow.  The equation has been studied extensively in the mathematical
literature due to its close analogy with the 3D incompressible Euler equations and the problem of
blowup for initially classical solutions, which remains open as it does for the Euler equations.  A survey of mathematical developments is given in the introduction to \cite{buckShkVicSQG}.  For more recent mathematical works on SQG we refer to \cite{hunter2021global,albritton2022non,garcia2022self,gomez2023quasiperiodic,hofmanova2023class,hmidi2023emergence,hofmanova2024surface,Ai-Avadanei-2024,Ai-Avadanei-2023a,Ai-Avadanei-2023b} and the references therein. 

Fundamental to the study of the SQG equation are the following basic conservation laws:
\begin{enumerate}
    \item For all sufficiently smooth solutions, the {\it Hamiltonian} $\fr{1}{2}\int_{\T^2} (|\nb|^{-1/2} \th(t,x))^2 dx$ remains constant.
    \item For all sufficiently smooth solutions, the $L^p$ norms $\| \th(t) \|_{L^p(\T^2)}$ remain constant for $1 \leq p \leq \infty$, as do the integrals $\int F(\th(t,x)) dx$ for any smooth function $F$.
    \item For all weak solutions to SQG, the mean, impulse, and angular momentum defined respectively by
    \ali{
M = \int \th(t,x) dx, \quad \vec{I} = \int_{\R^2} x \th(t,x) dx, \quad A = \int_{\R^2} |x|^2 \th(t,x) dx
    }
    are conserved quantities.  On the torus $\T^2$, the mean is well-defined and conserved.
\end{enumerate}
(To prove (1), multiply the equation by $|\nb|^{-1}\th$ and integrate by parts.  To prove (2), use $\nb \cdot u = 0$ to check that $F(\th)$ satisfies $\pr_t F(\th) + \nb \cdot( F(\th) u ) = 0$.  See \cite{isett2024conservation} for a proof of (3).)

Note that in contrast to (3), the nonlinear laws (1) and (2) require that the solution is ``sufficiently smooth''.  If one expects that turbulent SQG solutions have a dual energy cascade as in the Batchelor-
Kraichnan predictions of 2D turbulence \cite{constScalingScalar,constantin2002energy,buckShkVicSQG}, then one has motivation to consider weak solutions
that are not smooth.  A basic question for the SQG equation is then: What is the minimal amount of smoothness required for the conservation laws to hold?  This question is exactly the concern of the (generalized) {\it Onsager conjectures} for the SQG equation.  A closely related open problem is to find the minimal regularity required to imply uniqueness of solutions.

Using H\"{o}lder spaces to measure regularity, the Onsager conjectures can be stated as follows
\begin{enumerate}[label=\roman*)]
    \item If $\th \in C^0$, then conservation of the Hamiltonian holds.  However, for any $\a < 1/2$, there exist solutions with $|\nb|^{-1/2} \th \in L_t^\infty C^\a$ that do not conserve the Hamiltonian.
    \item If $\a > 1/3$ then the integral $\int F(\th(t,x)) dx$ is conserved for any smooth function $F$.  If $\a < 1/3$, there exist solutions in $\th \in L_t^\infty C^\a$ that violate this conservation law.
\end{enumerate}
The contribution of this paper is to fully answer the first conjecture in the affirmative.

Some remarks about these problems are in order:
\begin{enumerate}
\item These problems generalize the original Onsager conjecture \cite{onsag}, which concerned turbulent dissipation in the incompressible Euler equations and stated that the H\"{o}lder exponent 1/3 should mark the threshold regularity for conservation of energy for solutions to the incompressible Euler equations. See \cite{deLSzeCtsSurv,eyink2024onsager} for discussions of the significance of Onsager’s conjecture in turbulence theory.
\item The threshold exponents are derived from the fact that the conservation law for sufficiently regular solutions has been proven in both of the above cases.  Namely, \cite{isettVicol} proves conservation of the Hamiltonian for solutions with $\th \in L^3(I \times \T^2)$, while \cite{akramov2019renormalization} proves the conservation law (ii) for $\a > 1/3$. The proofs are variants of the kinematic argument of \cite{CET}, which proved energy conservation for the Euler equations above Onsager’s conjectured threshold.  For Hamiltonian conservation in the nonperiodic case, see \cite{isett2024conservation}.
\item Following the seminal work \cite{deLSzeCts}, advances in the method of convex integration have made possible the pursuit of Onsager's conjecture both for the Euler equations and more general fluid equations.
In particular, Onsager's conjecture for the 3D Euler equations has been proven in \cite{isettOnsag} (see also
\cite{buckDeLSVonsag,isett2017endpoint}), while the first progress towards the Onsager conjecture (i) for SQG was made in \cite{buckShkVicSQG}, with an alternative approach given in \cite{isett2021direct}. See \cite{deLOnsagThm, buckmaster2020convex} for surveys and \cite{klainerman2017nash} for a discussion of generalized Onsager conjectures.
\item To make sense of the Onsager problem for the Hamiltonian, it must be noted that the SQG equation is well-defined for $\th$ having negative regularity. 
 Namely, for any smooth vector field $\phi(x)$ on $\T^2$, the quadratic form
\ALI{
Q(\phi, \th) &= \int \phi T[\th] \cdot \nb \th dx dt ,
}
 initially defined for Schwartz $\th$ with compact frequency support away from the origin, has a unique bounded extension to $\th \in \dot{H}^{-1/2}$. This fact, which relies on the anti-symmetry of the operator $T$, allows the SQG nonlinearity to be well-defined in $\DD'$ for $\th$ of class $\th \in L_t^2 \dot{H}^{-1/2}$.  In fact, one has the following estimate, which is optimal:
 \ali{
|Q(\phi, \th)| &\lesssim \| \mrg{\nb}^2_{j,\ell} \phi \|_{L^\infty} \| \th \|_{\dot{H}^{-1/2}}^2 
 }
where $\mrg{\nb}^2_{j\ell} \phi = \nb_j \nb_\ell \phi - \fr{1}{2} \de_{j\ell} \De \phi$ is the trace-free part of the Hessian of $\phi$.  See \cite{isett2024conservation} for a proof of this bound and its optimality, and \cite{marchand2008existence,cheng2021non} for earlier definitions of the nonlinearity with weaker estimates.   
\end{enumerate}

The main theorems of our work are the following, which settle the Onsager conjecture on the threshold for Hamiltonian conservation for SQG.
\begin{thm}\label{thm:basicThm}  For any $\a < 1/2$, there exist weak solutions $\th$ to SQG that do not conserve the Hamiltonian such that $|\nb|^{-1/2} \th \in C_t C^\a$.
\end{thm}
\begin{thm}[Initial value problem] \label{thm:cpctSupp}  For any $\a < 1/2$, $c > 0$ and any smooth $\th_0 : \T^2 \to \R$ with integral zero there is a $T > 0$ and a weak solution to SQG of class $|\nb|^{-1/2} \th \in C^\a$ such that $\th$ realizes the initial data $\th_0$ at time $0$, coincides with the unique classical solution to SQG on $(-T,T)$, and is equal to $0$ on $|t| \geq T + c$.    
\end{thm}

\begin{thm}[h-Principle]\label{thm:hPrinciple} For any $\a < 1/2$ and for any $C_c^\infty((0,T)\times \T^2)$ function $f$ that conserves the mean, i.e. $\int_{\T^2} f(t,x) dx = 0$ for all $t$, there exists a sequence of SQG solutions $\th_n$ of class $|\nb|^{-1/2} \th_n \in C_t C^\a$ with compact support in time, such that 
$|\nb|^{-1/2} \th_n \rightharpoonup |\nb|^{-1/2} f
$in $L_{t,x}^\infty$ weak-*.
\end{thm}

Theorems~\ref{thm:cpctSupp} and \ref{thm:hPrinciple}, which both provide stronger forms of the first theorem, are both proven using a general approximation result (Theorem~\ref{thm:approxThem} below), which estimates the distance from a smooth SQG Reynolds flow in $C_tH^{-1/2}$ to the space of weak solutions to SQG.  The method of Theorem~\ref{thm:cpctSupp} can also be used to show that any two smooth SQG solutions can be glued together if they have the same mean.  Theorem~\ref{thm:hPrinciple}, meanwhile, is inspired by the original h-Principle of Nash \cite{nashC1} on the $C^0$ density of $C^1$ isometric immersions in the space of short maps.  The connection between h-principles and conservation laws was originally noted in \cite{IOnonpd, isettVicol}.  See also \cite{de2022weak} for a recent discussion of h-principle results in fluids.  We refer to Section~\ref{sec:nonuniquenessInitValue} for the proof of Theorem~\ref{thm:cpctSupp} and Section~\ref{sec:hPrinc} for Theorem~\ref{thm:hPrinciple}.

The previous best known result on this problem, due to \cite{buckShkVicSQG}, achieved regularity $|\nb|^{-1/2} \th \in C^{3/10-}$, with an alternative approach given in \cite{isett2021direct}.  Nonuniqueness of SQG steady states was proven in \cite{cheng2021non}. 
 We note also the works \cite{bulut2023non,dai2023non}, which prove nonuniqueness for {\it forced} SQG up to the Onsager threshold $|\nb|^{-1/2} \th \in C^{1/2-}$.  

Our improvement of the exponent relies on the following ideas:
\begin{enumerate}
    \item We build on the recent breakthrough solving the $2D$ Onsager conjecture in \cite{giri20232d}, which introduces a ``Newton iteration,'' which takes an arbitrary Euler-Reynolds flow and perturbs the velocity field so that the error is a sum of one-dimensional pieces with disjoint temporal support plus other error terms of acceptable size.  This idea builds on work of \cite{cheskidov2022sharp,cheskidov2023L2}.    \item The main difficulty in implementing the Newton iteration in the SQG context is to prove good estimates for a trace-free second-order\footnote{We require a second order anti-divergence since we base our approach on that of \cite{isett2021direct}.} anti-divergence tensor for the Newton correction.  That is, a trace-free solution $\rho$ to  $\mbox{div div } \rho = w$, where $w$ is the Newton correction.   The straightforward estimate for the solution to this equation is $\co{\rho} \ls \co{w}$, which turns out to be far from adequate.  Our approach to tackling this difficulty starts with the following idea:
    \begin{itemize}
        \item We define a system of ``transport elliptic'' equations that couples the equation for $w$ with an equation for a first-order anti-divergence $z$, which is then coupled to the equation for a second-order anti-divergence $r$ (that may not be trace-free).
        \item  We use a Littlewood-Paley analysis to prove suitable estimates for $r$, which then are shown to imply suitable estimates for $\rho$.  Much of the novelty of the work goes into this analysis.     \end{itemize}  
    \item The second main difficulty that separates the SQG scheme from 2D Euler is that certain bilinear or quadratic terms that occur naturally in both the Newton iteration and the convex integration steps need to be written in divergence form with an anti-divergence that satisfies good (dimensionally correct) estimates.  Here we build on ideas of \cite{isett2024conservation} and provide a more direct approach to achieving such divergence forms.  The main tool, which we call the ``divergence form principle,'' traces back to an important calculation in \cite{buckShkVicSQG} that was generalized and streamlined in \cite{isett2021direct}.  See Section~\ref{sec:divFormPrinciple}.
    \item  Within the Newton iteration, we use analytical ideas that we believe to be of independent interest.  For example, our methods can be used to give an alternative approach to some main results of \cite{isett2024conservation}, including the conservation of angular momentum for general weak solutions, and our commutator estimates (e.g. Lemma~\ref{lem:zeroOrder}) can be used to give an alternative treatment of the low-frequency terms that arise in the improvement of the borderline regularity estimate sketched in \cite[Section~12]{isett2017endpoint}. The sharp estimates we prove should be useful for obtaining an asymptotic endpoint type result for SQG, similar to that of \cite{isett2017endpoint}, but currently we do not know how to remove the reliance on double-exponential frequency growth in the Newton step.  \end{enumerate}

While the above are the main ideas that are new to this paper, we note that they are not the only ones needed to surpass the exponent $|\nb|^{-1/2} \th \in C^{3/10-}$.  In particular, we rely on some nonperturbative techniques that were already used in \cite{isett2021direct}, including the use of nonlinear phase functions as in \cite{isett}, the microlocal Lemma of \cite{isettVicol} and the bilinear microlocal Lemma of \cite{isett2021direct}.  In \cite{buckShkVicSQG} it was shown that certain perturbative techniques could be used in place of the above methods, but to get the sharp exponent we require techniques that remain effective on a nonperturbative timescale.  We also take advantage of an observation in \cite{bulut2023non}, which is that estimates on pure time derivatives for SQG can be used in place of advective derivative bounds.  While this point is likely not essential to the proof, it allows for a simpler argument where one does not need to commute advective derivatives with nonlocal operators many times.

Finally, we comment that during the writing of this paper we learned that \cite{dai2024SQG} have independently and concurrently obtained another proof of Theorem~\ref{thm:basicThm}.  

$ $

\noindent {\bf Acknowledgments}  The work of PI has been supported by a Sloan Fellowship and the NSF grant DMS 2346799. The work of SL was supported by a Taussky--Todd Fellowship. 

$ $

We now begin the proof with some notation.
\subsection{Notation} \label{sec:notation}
In this paper, the dimension $d = 2$.  We use vectors to indicate multi-indices $\va$ and use $|\va|$ to indicate the order of the multi-index.  For instance, if $\va = (a_1, a_2, a_3)$, $1 \leq a_i \leq d$, then $\nb_{\va} = \nb_{a_1} \nb_{a_2} \nb_{a_3}$ is a partial derivative of order $|\va| = 3$. 
We will use many times the following elementary {\bf counting inequality} with parameters $(x_1, x_2, y)$:
\ali{
(x_1 - y)_+ + (x_2 - y)_+ &\leq (x_1 + x_2 - y)_+ \quad \mbox{ if } x_1, x_2, y \geq 0.
}

We use the symbol $\tisum$ to indicate a sum with combinatorial coefficients that we have omitted to simplify notation.  For example, the product rule implies,
\ali{
\nb_{\va}(fg) = \tisum \nb_{\va_1} f \nb_{\va_2} g
}
where the sum runs over some but not all multi-indices with $|\va_1| + |\va_2| = |\va|$.  Meanwhile, the chain rule and product rule give
\ali{
\nb_{\va}G(F(x)) = \sum_{m=0}^{|\va|} \tisum \pr^{m} G(F(x)) \prod_{j=1}^m \nb_{\va_j} F,
}
where the empty product is $1$ and the sum is over certain multi-indices with $|\va_1| + \cdots + |\va_m| = |\va|$.  (To be more precise the multi-indices should be of the form $\va_{m,j}$, but we omit the $m$ subscript to simplify notation.)

We define Littlewood-Paley projections with the following conventions.  Suppose $\hat{\eta}_{\leq 0}(\xi)$ is $1$ on $|\xi| \leq 1/2$ and $0$ on $1 \leq |\xi|$, $\hat{\eta}_{\leq 0} \in C_c^\infty (\widehat{\R}^d)$.  For $q \in \Z$ we define
\ALI{
\widehat{P_{\leq q} f}(\xi) = \hat{\eta}(\xi/2^q) \hat{f}(\xi).
}
Thus in physical space one has $P_{\leq q} f = \eta_{\leq q} \ast f$ for $\eta_{\leq q}(h) = 2^{dq} \eta_{\leq 0}(2^q h)$.  We define Littlewood-Paley projections $P_q f = P_{\leq q+1} f - P_{\leq q} f$ so that $P_q f$ has frequency support in $\{ 2^{q-1} \leq |\xi| \leq 2^{q+1} \}$.

We use $P_{\approx q}$ to indicate a Fourier multiplier that is a bump function adapted to frequencies of size $\xi \sim 2^q$.  So for example, $P_q = P_q P_{\approx q}$.

We will use the summation convention to sum over repeated indices.  For example, $\nb_i u^i$ is the divergence of a vector field $u$. 

We will make use of two different anti-divergence operators.  The first is the order $-1$ operator $\RR_a^{j\ell}$, which solves
\ALI{
\nb_j \RR_a^{j\ell}[f^a] &= f^\ell, \qquad \de_{j\ell}\RR_a^{j\ell} = 0, \quad \RR^{j\ell}_a = \RR^{\ell j}_a
}
whenever $f^\ell$ is a vector field of mean zero on the torus.  The second operator is the order $-2$ operator $\RR^{j\ell}$, which solves
\ALI{
\nb_j\nb_\ell \RR^{j\ell}[f] &= f, \qquad \de_{j\ell}\RR^{j\ell} = 0, \quad \RR^{j\ell} = \RR^{\ell j}
}
whenever $f$ is a scalar field of mean zero on the torus.  Explicit formulas for these operators can be given in terms of the Helmholtz projection to divergence-free vector fields
\ALI{
\calH^{\ell}_a &\equiv \de_a^\ell - \De^{-1} \nb^\ell \nb_a \\
\RR^{j\ell}_a &= \De^{-1}(\nb^j \calH_a^\ell + \nb^\ell \calH_a^j ) - \De^{-1} \de^{j\ell} \nb_a + 2 \De^{-2} \nb^j \nb^\ell \nb_a \\
\RR^{j\ell} &= - \De^{-1} \de^{j\ell} + 2 \De^{-2}\nb^j \nb^\ell
}
See Section~\ref{sec:Glossary} for a glossary of the various symbols introduced in the proof.

\section{The main lemma}

\begin{defn}  A scalar-valued $\th : \R \times \T^2 \to \R$ and a symmetric traceless tensor field $R^{j\ell} : \R \times \T^2 \to \R^{2\times2}$ solve the {\bf SQG Reynolds equations} if
\ALI{
\pr_t \th + u^\ell \nb_\ell \th &= \nb_j \nb_\ell R^{j\ell} \\
u^\ell &= T^\ell \th = \ep^{\ell a} \nb_a |\nb|^{-1} \th
}
where $|\nb| = \sqrt{-\De}$.  The tensor $R^{j\ell}$ is called the {\bf error} since one has a solution when $R = 0$.
\end{defn}

\begin{defn} \label{defn:freqEnLevels}  Let $(\th, u, R)$ be an SQG-Reynolds flow, $\Xi \geq 1$ and $\D_u \geq \D_R \geq 0$ be non-negative numbers.  Define the advective derivative $D_t := \pr_t + T^\ell \th \nb_\ell$.  We say that $(\th, u, R)$ has {\bf frequency energy levels} below $(\Xi, \D_u, \D_R)$ to order $L$ in $C^0$ if $(\th, u, R)$ are of class $C_t^0 C_x^L$ and the following statements hold
\ALI{
\| \nb_{\va} \th \|_{C^0} , \| \nb_{\va} u \|_{C^0} &\leq \Xi^{|\va|} e_u^{1/2}, \qquad \mbox{ for all } |\va| = 0, \ldots, L \\
\| \nb_{\va} R \|_{C^0} &\leq \Xi^{|\va|} \D_R, \qquad \mbox{ for all } |\va| = 0, \ldots, L \\
\| \nb_{\va} D_t \th \|_{C^0}, \| \nb_{\va} D_t u \|_{C^0} &\leq \Xi^{|\va|} (\Xi e_u^{1/2}) e_u^{1/2}  \qquad \mbox{ for all } |\va| = 0, \ldots, L-1 \\
\| \nb_{\va} D_t R \|_{C^0} &\leq \Xi^{|\va|} (\Xi e_u^{1/2}) \D_R, \qquad \mbox{ for all } |\va| = 0, \ldots, L - 1
}
with $e_u^{1/2} = \Xi^{1/2} \D_u^{1/2}$ and $e_R^{1/2} = \Xi^{1/2} \D_R^{1/2}$. We note that, in contrast to other equations such as Euler, $e_u$ and $e_R$ will be large parameters. \end{defn}

\begin{lem}[Main Lemma] \label{lem:mainLem} For $L\geq 7$, $\eta \in (0, 1)$ there is a constant $\hc = \hc_{L,\eta} > 1$ such that the following holds: Given an SQG-Reynolds flow $(\th, u,R)$ with frequency energy levels below $(\Xi, D_u, D_R)$ to order $L$ and a set $J_0 \subseteq \R$ with $\suppt R \subseteq J_0$,  suppose that $N \geq 1$ satisfies
\ali{
N \geq \hc N^{\fr{6}{L}} N^{4\eta}  \Xi^{4\eta} \left(D_u/D_R\right). \label{eq:NlowerBd}
}
Then there exists an SQG-Reynolds flow $(\ost{\th}, \ost{u}, \ost{R})$ of the form $\ost{\th} = \th + W, \ost{u} = u + T[W]$ with frequency energy levels below 
\ali{
(\ost{\Xi}, \ost{D}_u, \ost{D}_R) = (\hc N \Xi, D_R, \hc N^{-1/2} (D_R/D_u)^{1/2} D_R) \label{eq:newFreqEns}
}
to order $L$ in $C^0$.

Furthermore the new stress $\ost{R}$ and the correction $W$ are supported in the set
\ali{
\suppt \ost{R} \cup \suppt W \subseteq N(J_0) := \{ t + h ~:~ t \in J_0, |h| \leq 5 (\Xi e_u^{1/2})^{-1} \}  \label{eq:supportEnlargement}
}
Additionally, $|\nb|^{-1/2} W$ satisfies the estimate
\ali{
\|\nb_{\va} |\nb|^{-1/2} W \|_{C^0} &\leq \hc (N\Xi)^{|\va|} \D_R^{1/2} , \qquad |\va| = 0, 1. \label{eq:LaWBd}
}
\end{lem}
It will be convenient to introduce the notation $\hn = N^{1/L}$.  We have $\ost{\Xi}=\widehat{C} N \Xi$ and $\ost{e}_u^{1/2}= \ost{\Xi}^{1/2} D_R^{1/2}$.

\begin{rem} \label{rem:barelyNotC0}  We remark that, while the lemma is true in the full range of parameters stated (and vacuously true if there is no $N$ satisfying \eqref{eq:NlowerBd}), we will be applying the lemma in the regime where $\fr{6}{L}$ and $4\eta$ are both small parameters depending on the regularity of the limiting solution.  In this case $N^{\fr{6}{L} + 4 \eta} \Xi^{4\eta}$ in \eqref{eq:NlowerBd} is close to $1$ so that \eqref{eq:NlowerBd} reduces to essentially $N \geq D_u/D_R$.  This last estimate is basically equivalent to the blowup of $C^0$ norms during the iteration, since $\co{\th} \leq \Xi^{1/2} D_u^{1/2}$ while the new scalar has size $\co{\ost{\th}} \lesssim N^{1/2} \Xi^{1/2} D_R^{1/2}$.

On the other hand, the $C^0$ norms blow up only slowly as the regularity of our solutions becomes arbitrarily close to $\th \in C^0$, so there exists only a small amount of room in the inequality $N \geq D_u/D_R$ for \eqref{eq:NlowerBd} to be possible.  Consequently $\fr{6}{L}$ and $4\eta$ must be small as the regularity of solutions approaches $C^0$, and verifying the inequality \eqref{eq:NlowerBd} in Section~\ref{sec:mainLemImpThm} is ultimately a delicate calculation. 
\end{rem}

\begin{rem} \label{rem:constantNormalize}  We remark that the constant $\hc$ in the bound on $\ost{D}_R$ in \eqref{eq:newFreqEns} can be normalized to $1$ as in previous work by applying the lemma with $N$ chosen to be $\hc^2 N$.
\end{rem}

\begin{rem}To prepare for proving Lemma~\ref{lem:mainLem} it is convenient to introduce the {\bf final weighted norm}
\ali{
\ost{H}^*[F] &= \max_{0 \leq r \leq 1} \max_{0 \leq |\va| + r \leq L} \fr{\|\nb_{\va} \ost{D}_t^r F\|}{\ost{\Xi}^{|\va|} \ost{\tau}^{-r}} \label{eq:finalWeighted}
}
where $\ost{D}_t = \pr_t + \ost{u} \cdot \nb$ is the new advective derivative and the new frequency and inverse timescale are $\ost{\Xi} = \hat{C} N \Xi$ and $\ost{\tau}^{-1} = \ost{\Xi} \ost{e}_u^{1/2}$, with $\ost{e}_u^{1/2} = \ost{\Xi}^{1/2} D_R^{1/2}$.  Including the factors of $\hc$ in the definition of the norm is optional, since in any case one obtains an equivalent norm.

In this notation, the estimates we desire for $\ost{R}$ in \eqref{eq:newFreqEns} can be written succinctly as 
\[ \ost{H}^*[\ost{R}] \lesssim N^{-1/2} (D_R/D_u)^{1/2} D_R \] 
with an implicit constant independent of $(\Xi, D_u, D_R)$ and $N$.  Some advantages of introducing such a weighted norm are that they are well-behaved with respect to products and composition with smooth functions, and the advective derivative $\ost{D}_t$ can be replaced by other time derivatives that appear in the proof such as $D_t = \pr_t + u \cdot \nb$ or $\pr_t$ to yield an equivalent norm that can be easier to control directly.  We introduce these equivalent norms, denoted by 
\[ \ost{H}^*[F] \lesssim H^*[F] \lesssim \bar{H}^*[F] \lesssim \tilde{H}^*[F] \lesssim \ost{H}^*[F], \] in Definition~\ref{defn:finalWeight}.

\end{rem}

\subsection{Summary Section}
The purpose of this section is to write down where all the main estimates of the paper are proven.

The main Proposition in the Newton step, which summarizes all that is needed to proceed to the convex integration step, is given by Proposition~\ref{prop:newtonBound}.  After the Newton step is setup, a summary section for the Newton step estimates appears in Section~\ref{sec:newtonSummary}.  A key upshot is that the standard estimates for beginning a convex integration scheme hold once the Newton step is complete, but with the threshold for borrowed derivatives $L$ replaced by $\unl \equiv L - 3$ due to a loss of three derivatives in the Newton step equations.

The new frequency-energy levels for $\ost{\th}$ and $\ost{u}$ are verified in Proposition~\ref{prop:newVelocBound}.  The stress $\ost{R}$ on the other hand has many different components, and each one is estimated either by $N^{-1} D_R$ or the slightly larger quantity $N^{-1/2} (D_u/D_R)^{-1/2} D_R$.  The bounds for the mollification and quadratic errors in the Newton Step are consequences of Proposition~\ref{prop:RmRqnewt} thanks to Proposition~\ref{prop:finalComp}, which guarantees the equivalence of estimates for $\pr_t$ and $\ost{D}_t$ for the final weighted norm.  After $\Ga$ iterations of the Newton step, the acceptable bound for the error $R_{(\Ga)}^{j\ell}$ follows from Proposition~\ref{prop:newtonBound} thanks to the choice of $\Ga$ in line \eqref{eq:DRn} guaranteeing $D_{R,(\Ga)} \leq N^{-1} D_R$, and also Proposition~\ref{prop:finalComp}, which guarantees that estimates for different advective derivatives are equivalent for the final weighted norm.  

The error terms in the convex integration step are defined in line \eqref{eq:errorList}.  The bounds for the transport error $R_T$ and the high frequency interference terms $R_H$ are obtained Section~\ref{sec:highFreqErrors} using again the equivalence of weighted norms for different advective derivatives.  The bounds for the mollification error $R_M$ are obtained in Section~\ref{sec:cvxMollErr}.  The bounds for $R_S$, which contains the stress error and flow error, are obtained in Section~\ref{sec:estimateRS}. 

The bound \eqref{eq:LaWBd} is a consequence of \eqref{eq:sizeOfNewtCorrect} and \eqref{eq:nbminhalfTh} where $\la \sim N \Xi$.  Meanwhile, the bound \eqref{eq:supportEnlargement} is a consequence of \eqref{eq:newtonTimeEnlarge}, the containment \eqref{eq:howBigSupport}, and the construction of $e_n^{1/2}(t)$ in line \eqref{eq:enConstruct}, since the support of the convex integration perturbation and error are bounded by the support of $e_n^{1/2}(t)$.

\section{The Newton Step}

Consider a given SQG-Reynolds flow $(\th, u, R)$ with frequency energy levels below $(\Xi, \D_u, \D_R)$ to order $L$ and time support interval $J_0$ and let $\eta > 0$ be given.  Our goal is to perturb the scalar field in such a way that the error will become smaller.  This goal will be achieved in two steps, the first called the Newton step and the second called the convex integration step.  Our new scalar field $\ost{\th}$ will have the form $\th + w + \Th$, where $w$ is called the {\bf Newton perturbation} and $\Th$ is called the {\bf oscillatory perturbation}, which arises in the convex integration step.

The goal of the Newton perturbation is to perturb the scalar field so that  the original stress $R$ is replaced by a new $\tilde{R}$ that is supported on disjoint intervals, where in each interval $\tilde{R}$ can be canceled out by a one-dimensional convex integration perturbation.  Doing so overcomes the difficulty in the convex integration step that waves oscillating in distinct directions are not allowed to interfere with each other. 

Constructing the Newton perturbation that achieves this localization will be achieved in a number $\Ga = \lceil \eta^{-1} \rceil$ iterative steps indexed by $n\in\{0,\ldots,\Gamma-1\}$.  After the Newton perturbation we will add a high frequency perturbation $\Th$ that will be the sum of waves of the form $\Th = \sum_I \Th_I \approx \sum_I \th_I e^{i \la \xi_I}$ that will cancel out the ``low frequency part'' of what remains of the error, leaving behind an error that is small enough for the whole procedure to be repeated until the error is reduced to zero in the limit.  Each wave has a conjugate wave $\Th_{\bar{I}} = \overline{\Th}_I$, $\xi_{\bar{I}} = - \xi_I$, making $\Th$ real-valued.

We define the sets $F = \{ \pm (1,2), \pm (2,1) \}$ and $\F = \{ (1,2), (2,1) \}$, which will be the directions in which the oscillatory waves of the convex integration stage oscillate.  That is, $\nb \xi_I $ is reasonably ($O(1)$) close to an element of $F$.

During the convex integration step, each wave $\Th_I + \Th_{\bar{I}}$ is individually able to cancel out a ``one-dimensional'' component of the error that takes on the form $- \ga^2 B^{j\ell}(\nb \xi_I)$, where 
\ali{
B^{j\ell}(p) = -i(\nb^j m^\ell(p) + \nb^\ell m^j)(p), \label{eq:Bjldef}
}
where $m^\ell(p)~=~i \ep^{\ell a} p_a |p|^{-1}$ is the multiplier for SQG and where $\ga$ is a slowly varying smooth function that remains to be chosen.  The leading symbol and its remainder are supplied by the Bilinear Microlocal Lemma \cite[Lemma~4.2]{isett2021direct}; the applicability of that lemma in our notation is explained after \eqref{eq:selfInteract}. Thus one of the first tasks that must be done is to decompose the (low frequency part of the) error into a linear combination of terms of this form.  Before we perform this decomposition, we must define what we mean by the low frequency part of the error, which is the part that will be canceled out by the oscillatory perturbation $\Th$.
{
\subsection{Length scales and time scales} \label{sec:lengthAndTimeScales}

The construction involves three principle length scales, which are all small parameters.  The ambient length scale is $\Xi^{-1}$, the mollification length scale is $\ep = \ep_x = N^{-1/L} \Xi^{-1} = \hn^{-1} \Xi^{-1}$, and the new length scale is $(N \Xi)^{-1}$.  Closely related are the cutoff frequency $\hxi = \hn \Xi = N^{1/L} \Xi$ and the new frequency, a number $\la \in 2\pi \Z$ such that $\la \sim N \Xi$.  The chain of inequalities $1 \ll N^{1/L} = \hn \ll N \lesssim \hc N$ imply an ordering of the inverse length scales 
\ali{
\Xi \ll \ep^{-1} = \hxi \ll \la \lesssim \ost{\Xi}. 
}
The important time scales in the problem include the ambient time scale $(\Xi e_u^{1/2})^{-1}$, the slow time scale for the Newton iteration 
\ali{
 \tau \equiv b (\log \hxi)^{-1} (\Xi e_u^{1/2})^{-1} = b (\log \hxi)^{-1} (\Xi^{3/2} \D_u^{1/2})^{-1}, \label{eq:natTimescale}
}
(where $b$ is an $(L,\eta)$-dependent constant to be chosen).  Many estimates will have the feature that the advective derivative costs $\tau^{-1}$.  For bookkeeping purposes, we introduce the auxiliary energy level
\ali{
\widetilde e_u^{1/2}:=(\log\widehat\Xi)e_u^{1/2},
\qquad \Xi\widetilde e_u^{1/2}=b\tau^{-1}\sim_{L,\eta}\tau^{-1}. \label{eq:effectiveEu}
}
The fast time scale for the Newton iteration is $\mu^{-1}$, where
\ali{
\mu = N^{1/2} \Xi e_R^{1/2} =  N^{1/2} \Xi^{3/2} \D_R^{1/2}, \label{eq:muChoice}
}  The timescale for mollification along the flow is $\ep_t$, where
\ali{
\ep_t &= (\Xi e_u^{1/2})^{-1} (D_u/D_R)^{-1/2} N^{-1/2}. \label{eq:ept}
}
The new time scale is $\ost{\tau} = (\ost{\Xi} \ost{e}_u^{1/2})^{-1}$.

The order of the associated inverse time scales is
\ali{
\Xi e_u^{1/2} \ll \tau^{-1} \ll \mu \ll \ep_t^{-1} \ll \ost{\tau}^{-1}.
}
For the last comparison, observe that
\(\epsilon_t^{-1}/\ost\tau^{-1}\lesssim (D_u/D_R)/N\),
which is small due to \eqref{eq:NlowerBd}.

Finally, the characteristic velocities $e_u^{1/2}$ and $\ost{e}_u^{1/2}$ satisfy $e_u^{1/2} \ll \ost{e}_u^{1/2}$ using $N \gg D_u/D_R$.  This comparison is different from Euler, where the characteristic velocities get smaller at smaller scales, and reflects the fact that the velocity field in the limit is not bounded.

\subsection{Regularizing the scalar field and error tensor}

We mollify the velocity field and scalar at the length scale $\ep = (\hn \Xi)^{-1} = \hxi^{-1}$.  Choose an integer $q_\ep$ such that $2^{q_\ep}$ is within a factor of $2$ of $\ep^{-1}$ and set
\ali{
\th_\ep = P_{\leq q_\ep} \th = \eta_{\leq q_\ep} \ast \th, \qquad u_\ep^\ell = T^\ell \th_\ep,
}
where $P_{\leq q_\ep}$ is a Littlewood--Paley projection in the spatial variables.  Its convolution kernel has integral one and all moments of orders $1$ through $L$ equal to zero.

In terms of the coarse scale velocity field we define the coarse scale advective derivative according to $$\Ddt = \pr_t + u_\ep \cdot \nb.$$  

The error tensor $R$ must also be regularized before we attempt to cancel it out.  We define $R_{\ep}$ by mollifying $\eta_{\ep_x} \ast_x \eta_{\ep_x} \ast_x R(t,x)$ only in the spatial variables at the length scale $\ep_x = N^{-1/L} \Xi^{-1},$ and using a mollifying kernel such that $\int h^{\va} \eta(h) dh = 0$ for all multi-indices $1 \leq |\va| \leq L$.

To define the iteration we need the following basic estimates that follow from Minkowski's inequality
\ali{
\co{\nb u_\ep} \leq A_L \Xi e_u^{1/2}, \qquad \co{R_\ep} \leq A_L D_R \label{eq:ALappears}
}
where $A_L = \max \{ \| \eta_{\leq q_\ep} \|_{L^1}, \|\eta_{\ep_x} \ast \eta_{\ep_x} \|_{L^1} \}$.

\subsection{Prelude to convex integration} \label{sec:preludeConvex}
Let $e_0(t)$ be a smooth function of time (an ``envelope function'') to be chosen shortly.  We re-write the SQG-Reynolds equation as
\ali{
\pr_t \th + \nb_\ell[ \th T^\ell[\th]] &= \nb_j \nb_\ell (R_\ep^{j\ell} - e_0(t) M^{j\ell}) + \nb_j\nb_\ell (R^{j\ell} - R_\ep^{j\ell}) \label{eq:step0Prepared} }
where
\ali{
M^{j\ell} &= B^{j\ell}((1,2)) + B^{j\ell}((2,1))
}
Here we use that $M^{j\ell}$ is a constant matrix, which implies $\nb_j \nb_\ell M^{j\ell} = 0$.  The function $e_0(t)$ will be chosen just large enough so that $e_0(t) M^{j\ell}$ dominates the term $R_\ep^{j\ell}$.

We must decompose $R_\ep - e_0 M$ into components that live on a size $\tau$ time interval and that can be canceled out by one-dimensional oscillations.  Towards the latter goal, consider the quadratic equation for unknown coefficients $\ga_f$, $f \in \F = \{ (1,2), (2,1) \}$, where the $p_f \in \R^d$ and $\varep^{j\ell}$ is a symmetric and trace-free two-tensor:
\ali{
\sum_{f \in \F} \ga_f^2 B^{j\ell}(p_f) &= M^{j\ell} - \varep^{j\ell}. \label{eq:normalizedQuadratic}
}
When each $p_f = f$ and $\varep = 0$, this equation has a solution given by $\ga_{(1,2)} = 1$ and $\ga_{(2,1)} = 1$.  Moreover, the tensors $B^{j\ell}((1,2))$ and $B^{j\ell}((2,1))$ are easily checked to be linearly independent, hence form a basis for the two-dimensional space of symmetric, trace-free two-tensors on $\R^2$.  Moreover, $B^{j\ell}(p)$ depends smoothly on $p$, so by the implicit function theorem (or simply inverting the linear equation and taking square roots) there is a constant $c_1>0$ and there are smooth functions $\ga_f(\varep, p)$ of $(\varep, p_{(1,2)}, p_{(2,1)})$, $f \in \F$, defined in an open neighborhood of the set
\ali{
|\varep| \leq c_1, | p_{f} - f | \leq c_1  \label{eq:neighborhood}
}
such that the $\ga_f(\varep, p)$ solve \eqref{eq:normalizedQuadratic} and take values in $| \ga_f(\varep, p) - 1| \leq 1/2$.  (The choice of norms in \eqref{eq:neighborhood} is unimportant since all norms in finite dimensions are equivalent.) 

We now choose $e_0(t)$ to make sure $e_0(t) M^{j\ell}$ dominates $R_\ep^{j\ell}$.  Let $A_L$ be as in \eqref{eq:ALappears}.  Let $M_e$ denote a constant depending on $L$, to be chosen, such that $e_0(t)$ takes on the constant value $M_e D_R$ on the set $\{ t + h ~:~ t \in J_0, |h| \leq (\Xi e_u^{1/2})^{-1} \}$, which contains the support of $R_\ep$ and also the support of $\check{R}$ to be defined later on.

We now decompose the right hand side of \eqref{eq:step0Prepared} into time-localized pieces that are almost well-prepared to be canceled by the convex integration construction.  By a standard construction (see for instance \cite{deLSzeCts}), we can choose a partition in unity in time of the form  
\[ 1 = \sum_{k \in \Z} \chi_k^2(t), \qquad \chi_k(t) = \chi(\tau^{-1}(t - \tau_k)) \]
where each $\chi_k(t)$ has support in $(\tau k - \tau, \tau k + \tau)$ and is equal to $1$ in $[\tau k - \tau/4, \tau_k + \tau/4]$.

Define the Newton step phase functions as the solutions to
\ali{
\begin{split}
\label{eq:transportOfPhase}
(\pr_t + u_\ep^j \nb_j) \check{\xi}_{(k,f)} &= 0 \\
\check{\xi}_{(k,f)}(k\tau, x) &= f \cdot x.
\end{split}
}
(We note that $\nb \check{\xi}$ is well-defined on the torus thanks to the condition $f \in \Z^2$.)  These oscillate in the direction $\nb \xi_{(k,f)}$, which solves the transport equation 
\ALI{
(\pr_t + u_\ep^j) \nb_j \nb_a \check{\xi}_{(k,f)} &= - \nb_a u_\ep^j \nb_j \check{\xi}_{(k,f)} \\
\nb \check{\xi}_{(k,f)}(\tau k, x) &= f
}
Recalling that $\co{\nb u_\ep} \leq A_L \Xi e_u^{1/2}$, Gronwall's inequality yields 
\[ | \nb \check{\xi}_{(k,f)} - f | \leq |f| [\exp(A_L \Xi e_u^{1/2} |t - \tau k|)  - 1] \]
and hence (compare \cite[Proposition~17.2, equations~(536)--(537)]{isett})
\ali{
|\nb \check{\xi}_{(k,f)} - f| &\lesssim_L b,
}
for times of the order $|t - \tau k| \leq \tau \leq b (\Xi e_u^{1/2})^{-1}$, 
where $b$ is the constant in the definition of $\tau$.

Now define the coefficients for stage $0$ of the Newton step
\ali{
\ga_{(k,f,0)} = \chi_k e_0^{1/2}(t) \ga_f\left(\fr{R_\ep^{j\ell}}{M_e D_R}, \nb \check{\xi}_{k} \right) \label{eq:implicitFunct}
}
where $\nb \check{\xi}_k = (\nb \check{\xi}_{(1,2)}, \nb \check{\xi}_{(2,1)})$.  The coefficients are well defined provided $b$ is sufficiently small and $M_e$ is sufficiently large (both depending on $A_L$) so that the arguments of $\ga_f$ lie in the neighborhood \eqref{eq:neighborhood}.  We can fix $M_e$ once and for all so that not only is \eqref{eq:implicitFunct} indeed well-defined, but we have some room so that $\co{R_\ep/(M_e D_R)} \leq A_L / M_e$ is bounded by $c_1/2$, where $c_1$ is as in \eqref{eq:neighborhood}.  Fix also an upper bound for $b$ so that the phase gradients lie in \eqref{eq:neighborhood} with $c_1/2$ replacing $c_1$.

In terms of these coefficients, we finally have the decomposition into time-localized pieces that are almost prepared for convex integration as desired
\ali{
e_0(t) M^{j\ell} - R_\ep^{j\ell} &= \sum_{k \in \Z} \sum_{f \in \F} A_{(k,f,0)}^{j\ell} \label{eq:stressEqn} \\
A_{(k,f,0)}^{j\ell} &= \ga_{(k, f,0)}^2 B^{j\ell}(\nb \check{\xi}_{(k,f)}).
}
Each of the individual terms in the sum \eqref{eq:stressEqn} is now exactly of the form that can be canceled by a one dimensional wave using convex integration.  At this moment the difficulty we face is the interference problem: Waves that oscillate in different directions interact with each other and produce terms that cannot be well-controlled.  

The solution to this problem, building on \cite{giri20232d,cheskidov2023L2}, is to use a {\bf Newton iteration}, which perturbs the velocity field in such a way that the individual terms in \eqref{eq:stressEqn} are replaced by terms with disjoint support in time, still well-prepared for convex integration, modulo other acceptable error terms.  The challenge we face in the present work is how to adapt this method, originally used for 2D Euler, to the setting of SQG.   

\subsection{Setting up the Newton iteration} \label{sec:settingUp}

Once again, our aim is to add a Newton correction $w$ to $\th$ that will replace the term $(R_\ep^{j\ell} - e_0(t) M^{j\ell})$ with a sum of error terms that are ``one-dimensional'' with disjoint supports that can be canceled out by a convex integration argument, modulo other acceptable errors.  This task will be done by an iteration involving $\Ga = \Ga(\eta)$ steps indexed by $n\in\{0,\ldots,\Gamma-1\}$.  

We write down some basic intuition for the Newton iteration of \cite{giri20232d} in the current context.  We want to replace the forcing terms $\nb_j\nb_\ell A_{(k,f,0)}^{j\ell}$ with forcing terms that have disjoint support.  For each $([k],f) \in \Z/2\Z \times \F$ choose a one-periodic function $g_{([k],f,0)}$ such that $\int_{\R/\Z} g_{([k],f,0)}^2(y)dy = 1$ and such that the supports of the $g_{([k],f,0)}$ are disjoint.  Then, at least on time scales of size $\tau$, one suspects that replacing each forcing term $\nb_j\nb_\ell A_{(k,f,0)}^{j\ell}$ with the forcing term $ g_{([k],f,0)}^2(\mu t) \nb_j\nb_\ell A_{(k,f,0)}^{j\ell}$ will only perturb our approximate solution $\th$ since the two forcing terms are essentially the same on average in time; in other words, the two forcing terms differ by something high frequency in time as $\mu \gg \tau^{-1}$.  And if we believe the change in $\th$ will be small, we can hope the error terms generated by nonlinear interactions are acceptably small errors for the iteration when $\mu$ is large.  Moreover, with this choice the new stress terms $g_{([k],f,0)}^2(\mu t) A_{(k,f,0)}^{j\ell}$ have disjoint supports and are well-prepared for convex integration, as the term $g_{([k],f,0)}^2$ can be absorbed into the coefficients $\ga^2$.  While this intuition holds on short time intervals, we can also hope that the corrections can be glued together on size $\tau$ time intervals in such a way that the gluing error is also under control similar to the gluing in the proof of \cite{isettOnsag}.  This last task involves nontrivial harmonic analysis that is a main undertaking of this paper.  

There is an important point, however, that the above intuition neglects.  According to the above, a large choice of $\mu$ should result in a small perturbation to the scalar $\th$ and acceptable error terms, but there is a price to pay from taking a large $\mu$ which is that the transport term in the convex integration becomes larger by a factor of $\mu$,  The best one can do is optimize between the nonlinear errors in the Newton step and the transport error in convex integration, which has resulted in our choice of $\mu$ in \eqref{eq:muChoice}.  It turns out that, with this choice, the gain in the gluing error is miniscule (the new error is not much smaller than $D_{R,1} = D_R/(N\Xi)^\eta$), but this Newton step can be iterated $O(\eta^{-1})$ times until the gluing error is acceptable and convex integration can finally be performed.

Let us now prepare for the full Newton iteration.  Choose cutoff functions $\tilde{\chi}_k(t) = \tilde{\chi}(\tau^{-1}(\tau k - t))$ with support in $|t - \tau_k| \leq 3\tau/2$ such that $\tilde{\chi}_k$ is $1$ on the support of $\chi_k$.  Set $D_{R,0} = D_R$ and define the $n$th error bound
\ali{
D_{R,n} = (N^{-\eta} \Xi^{-\eta})^n D_{R,0}, \label{eq:DRn}
}
Fix $\Ga = \Ga(\eta)$ such that $D_{R,\Ga} \leq D_R/N$.  

With $J_0$ as in the Main Lemma, the $n$th error will be supported in
\ali{
J_{n+1} = \{ t + h ~:~ t \in J_{n} \, |h| \leq 3 \tau \}.
}
Fix $b = b(\eta, L)$ such that 
\ali{
J_{\Ga} &\subseteq \{ t + h ~:~ t \in J_0 \, |h| \leq (\Xi e_u^{1/2})^{-1} \}. \label{eq:howBigSupport}
}
We choose functions $e_n(t)$ supported in a $\tau/2$ neighborhood of $J_n$ such that $e_n^{1/2}(t)$ is smooth and satisfies the cutoff-scale estimates
\ali{
\co{\pr_t^r e_n^{1/2}(t)} \lesssim_r \tau^{-r} D_{R,n}^{1/2}, \label{eq:enConstruct}
}
and $e_n(t) = M_e D_{R,n}$ on a $\tau/4$ neighborhood of $J_n$.  Such an $e_n^{1/2}(t)$ can be constructed by mollifying the indicator function of a $\tau/3$ open neighborhood of $J_n$.

Define the index $J = (k,n)$ taking values in $\Z \times \{0,1, \ldots, \Ga - 1\}$.  The $n$th stress decomposes into terms of the form
\ali{
A_{(f,k,n)}^{j\ell} = A_{(f,J)}^{j\ell}
&:=\chi_k^2(t)e_n(t)\,
\ga_f^2\!\left(\frac{R_{(n)}}{M_eD_{R,n}},\nabla\check\xi_k\right)
B^{j\ell}(\nabla\check\xi_{(k,f)}). \label{eq:AfJform}
}
Except for the case $n = 0$ where $R_0 = R_\ep$, we have not yet defined $R_{(n)}$ but will do so shortly.  

Towards obtaining stress terms with disjoint support, fix a choice of $4 \Ga$ smooth non-negative $1$-periodic functions $g_{([k], f, n)}$ indexed by $([k], f, n) \in \Z/2\Z \times \F \times \{0,\ldots, \Ga - 1\}$ such that their supports are disjoint and normalized so that $\int_{\R/\Z} g_{([k],f,n)}^2(y) dy = 1$.  The existence of such functions is a simple induction on the number of $g$ (see for instance \cite[Lemma~3.3]{giri20232d}).  For $J = (k, n)$ we abbreviate $[J] = ([k], n) \in \Z/2\Z \times \{0, \ldots, \Ga - 1 \}$.

Define $w_J$ to solve the forced linear transport equation
\ali{
\begin{split}
\pr_t w_{J} + T^\ell \th_\ep \nb_\ell w_J + T^\ell w_J \nb_\ell \th_\ep &= \sum_{f \in \F} (1 - g_{[f,J]}^2(\mu t) ) \nb_j \nb_\ell A_{(f,J)}^{j\ell} \\
&=: \nb_j \nb_\ell O_J^{j\ell} \\
w_{J}(k\tau, x)&= \tilde{w}_{J}(k \tau, x). \label{eq:NewtonEquation}
\end{split}
}
The initial data is in terms of a parametrix $\tilde{w}_{(k,n)}$ that we introduce in Section~\ref{sec:cancellation} below to prove cancellation in solving the equation.  Using a contraction mapping argument in Sobolev spaces, we prove existence, uniqueness and smoothness of solutions to this equation for arbitrary smooth data and forcing in the Appendix section~\ref{sec:ExistenceNewton}.  

Let $\th_n$ and $T^\ell \th_n$ refer to the scalar field and velocity field after $n$ Newton iteration steps.  As suggested above, we want to glue together the $w_{(k,n)}$ to define our Newton correction.  Thus,
\ALI{
\th_{n+1} &= \th + \sum_{0 \leq j \leq n} w_{j}, \qquad w_{n} = \sum_{k \in \Z} \tilde{\chi}_k w_{(k, n)}.
}
The Newton correction is $w = \sum_{n=0}^{\Ga - 1} w_n$.

After $n$ Newton steps, we inductively have a system of the form
\ali{
\pr_t \th_n + T^\ell \th_n \nb_\ell \th_n &= \nb_j \nb_\ell R_{(n)}^{j\ell} + \nb_j \nb_\ell S_{(n)}^{j\ell} + \nb_j \nb_\ell P_{(n)}^{j\ell} \label{eq:afternNewton}
}
where
\begin{itemize}
	\item $R_{(n)}$ is the gluing error obtained in the previous stage if $n \geq 1$ (see \eqref{eq:gluingError2} below), while $R_{(0)}^{j\ell} = R_\ep^{j\ell}$. 
	\item $S_{(n)}$ is the error that will be canceled out by one-dimensional oscillations during the convex integration step while $S_{(0)}^{j\ell} = 0$.
	\item $P_{(n)}$ is the error that is small enough to be included in $\ost{R}$ in the next stage of the iteration, where $P_{(0)} = R^{j\ell} - R_\ep^{j\ell}$.
\end{itemize}
More specifically: With such a choice of Newton correction, the errors after the $n+1$'th step solve the following system of equations
\ali{
\th_{n+1} &= \th_n + w_n = \th + \sum_{j=0}^n w_j \\
\nb_j \nb_\ell R_{(n+1)}^{j\ell} &= \sum_k \pr_t \tilde{\chi}_k w_{(k,n)} \label{eq:gluingError2} \\
S_{(n+1)}^{j\ell} &= S_{(n)}^{j\ell} - \sum_{k \in \Z} \sum_{f \in \F} g_{(f, [k], n)}^2(\mu t) A_{(f,k,n)}^{j\ell} \\
\nb_j \nb_\ell P_{(n+1)} &= \nb_j \nb_\ell P_{(n)} + T^\ell(\th - \th_\ep) \nb_\ell w_n + T^\ell w_n \nb_\ell(\th - \th_\ep)  \label{eq:notDivForm10}\\
&+ T^\ell w_n \nb_\ell w_n + \sum_{j = 0}^{n-1} (T^\ell w_n \nb_\ell w_j + T^\ell w_j \nb_\ell w_n)\label{eq:notDivForm20}
}
Notice that the terms in \eqref{eq:gluingError2}, \eqref{eq:notDivForm10} and \eqref{eq:notDivForm20} are not in the form of a second-order divergence of a trace-free tensor field, in contrast to the analogous terms for Euler, where the latter two are readily of the correct form.  Handling this new issue and getting good estimates for the solutions to \eqref{eq:gluingError2}, \eqref{eq:notDivForm10}-\eqref{eq:notDivForm20} is another of the main contributions of this paper.

\subsection{Obtaining cancellation in the Newton iteration} \label{sec:cancellation}
It is necessary to prove a great deal of cancellation in our estimates for the Newton corrections $w_J$.  First, we must exploit the fact that the right hand side is high frequency in time when $\mu$ is large if we hope to have adequate estimates for $P_{(n+1)}^{j\ell}$ to be acceptable.  To do so we must use bounds on the advective derivatives of the $A_J$ terms to show that their characteristic time scale (which will be $\tau$) is larger than the fast time scale $\mu^{-1}$.  But we also need another cancellation when we solve the divergence equation to determine $R_{(n+1)}^{j\ell}$.  Indeed we should gain a factor of $\Xi^{-2}$ when solving this second order equation to be consistent with dimensional analysis, which is a much better bound than the trivial one obtained from $\co{\mbox{div}^{-1} \mbox{div}^{-1} f} \lesssim \co{f}$, and which is subtle because the terms $w_J$ are not high frequency in space.

To obtain both of these cancellations at once, we combine two techniques.  First, we use the technique of \cite{giri20232d} to write each $w_J$ as $w_J = \tilde{w}_J + \bar{w}_J$ as a parametrix plus a remainder, where the parametrix is explicit and gains the smallness of $\mu^{-1}$ while the remainder solves the transport equation with a forcing term that gains a $\mu^{-1}$.  (This can be compared with the use of nonstationary phase in convex integration; see Section~\ref{sec:nonstatPhase}.)  Second, we combine this with the method of \cite{isettOnsag} to find a second order anti-divergence $r_J^{j\ell}$ (not symmetric or trace free) that is defined through a ``transport elliptic equation'' and is shown by dynamical methods to satisfy good bounds.  We then transfer these bounds for $r_J$ with a logarithmic loss to $\rho_J^{j\ell} = \RR^{j\ell} w_J = \RR^{j\ell} \nb_a\nb_b r_J^{ab}$.  The explicit details are as follows.

For $J = (k,n)$, consider a solution to the equation 
\ali{
(\pr_t + T^\ell \th_\ep \nb_\ell) z_J^i &= \nb_a T^i \th_\ep z_J^a  - T^i w_J \th_\ep + \nb_a O_J^{ia} \label{eq:plainForm} \\
z_{J}^i(\tau k, x) &= \tilde{z}_{J}^i(\tau k) 
}
with smooth initial data to be specified below such that $\nb_i \tilde{z}_{(k,n)}^{i}(\tau k) = \tilde{w}_{(k,n)}(\tau k)$.  The existence of a smooth solution to \eqref{eq:plainForm} follows from standard existence theory for transport equations by the method of characteristics and variation of parameters.

It is not difficult to check that if $z_J^i$ solves \eqref{eq:plainForm}, then $\nb_i z_J^i$, the divergence of $z_J$, satisfies $\Ddt \nb_i z_J^i = \Ddt w_J$ and thus equals $w_J$ as long as it does so initially.  Thus $z_J^j$ is an anti-divergence for $w_J$.

We now wish to find an anti-divergence for $z_J^j$.  Using the fact that the divergence of $z_J^j$ is $w_J$, we can rewrite equation \eqref{eq:plainForm} as
\ali{
(\pr_t + T^\ell \th_\ep \nb_\ell) z_J^i &= \underbrace{\nb_a T^i \th_\ep z_J^a  - \nb_a T^i z_J^a \th_\ep}_{\text{special term}} + \nb_a O_J^{ia} \label{eq:plainForm2}
}
The special term has a structure that makes it possible to be put in divergence form.  Ultimately the most important point is that the operator $\nb_a T^i$ has a symbol that is even (and degree 1 homogeneous) and the fact that a minus sign appears (which together imply that the special term has integral zero).  Using these properties, we claim that
\ali{
\nb_a T^i[\th_\ep]z_J^a  - \nb_a T^i[z_J^a] \th_\ep &= \nb_j \BB^{ji}_a[ z_J^a, \th_\ep]
}
where $\BB_a^{ji}$ is a bilinear form that we will be able to estimate.  In terms of this anti-divergence, (which will only appear explicitly when we estimate $r_J$) define $r_J^{ij}$ to be the unique two-tensor solving 
\ali{
(\pr_t + T^\ell \th_\ep \nb_\ell) r_J^{ij} &= \RR^{ij}_b \nb_\ell [ \nb_a T^\ell \th_\ep r_J^{ab} ] + \BB^{ij}_b[z_J^b, \th_\ep] + O_J^{ij} \label{eq:transElliptic} \\
r_{J}^{ij}(\tau k, x) &= \tilde{r}_{J}^{ij}(\tau k, x),
}
where the initial data, to be specified shortly, satisfies $\nb_i \tilde{r}_{(k,n)}^{ij}(\tau k) = \tilde{z}_{(k,n)}^j(\tau k)$.  Here $\RR^{ij}_a$ is as defined in Section~\ref{sec:notation}.  Note that any solution to the initial value problem \eqref{eq:transElliptic} satisfies $\nb_i r_J^{ij} = z_J^j$ because the two vector fields have the same advective derivative and initial condition, and hence crucially $\nb_i\nb_i r_J^{ij} = w_J$.  In the above we rely on the existence and uniqueness of smooth solutions of equations of the form
\ali{
(\pr_t + v^\ell \nb_\ell) r^{ij} &= \RR^{ij}_a \nb_\ell[\nb_b v^\ell r^{ab}] + f^{ij}, \\
r^{ij}(0,x) &= r_0^{ij}
}
when $v$ is div-free and $v$, $f$ and $r_0$ are smooth in space.   This well-posedness is proven by contraction mapping in Sobolev spaces in the Appendix\footnote{The cited reference proceeds in dimension $3$, but the only important properties used are that $\RR_a^{ij} \nb_\ell$ commutes with partial derivatives and is bounded on $L^p$ for $1<p<\infty$, which hold as well in dimension $2$.} of \cite{isettOnsag}. 

Our description of the Newton correction is complete once we specify the decompositions of $w_J = \tilde{w}_J + \bar{w}_J$, $z_J^i = \tilde{z}_J^i + \bar{z}_J^i$ and $r_J^{ij} = \tilde{r}_J^{ij} + \bar{r}_J^{ij}$ into a parametrix part plus a remainder.  

Let $h_{f,[J]}(T) = \int_0^T (1 - g_{f,[J]}^2(s) ) ds$ and decompose $w_J = \bar{w}_{J} + \tilde{w}_{J}$, where 
\ali{
\tilde{w}_J &= \sum_{f \in \F} \mu^{-1} h_{f,[J]}(\mu t) \nb_j \nb_\ell A_{(f,J)}^{j\ell} \label{eq:tildewJ}\\
\pr_t \bar{w}_{J} + T^\ell \th_\ep \nb_\ell \bar{w}_J &= - T^\ell w_J \nb_\ell \th_\ep - \tilde{O}_J \label{eq:barwJ} \\
\bar{w}_J(\tau k) &= 0 \\
\tilde{O}_J &= \sum_{f \in \F} \mu^{-1} h_{f,[J]}(\mu t) \Ddt \nb_j \nb_\ell A_{(f,J)}^{j\ell}
}
We also decompose $z_J = \bar z_J + \tilde z_J$, where 
\ali{
\tilde{z}_J^i &= \sum_{f \in \F} \mu^{-1} h_{f,[J]}(\mu t) \nb_j A_{(f,J)}^{ij} \\
\partial_t \bar z_J^i + T^\ell \theta_\ep \nabla_\ell \bar z_J^i &= \nabla_a T^i \theta_\ep z_J^a - T^i w_J \theta_\ep - \tilde{O}_J^{i} \label{eq:barzJ}\\ 
\bar{z}_J^i(\tau k) &= 0 \\
\tilde O_J^{i} &= \sum_{f \in \F} \mu^{-1} h_{f,[J]}(\mu t) \Ddt \nabla_a A_{(f,J)}^{ia}
}
We similarly decompose $r_J = \bar r_J + \tilde r_J$, where
\ali{
\tilde{r}_J^{ij} &= \sum_{f \in \F} \mu^{-1} h_{f,[J]}(\mu t) A_{(f,J)}^{ij} \\
\partial_t \bar r_J^{ij} + T^\ell \theta_\ep \nabla_\ell \bar r_J^{ij} &= \RR^{ij}_b \nabla_\ell [ \nabla_a T^\ell \theta_\ep r_J^{ab} ] + \BB^{ij}_b[z_J^b, \theta_\ep] - \tilde O_J^{ij} \label{eq:barrJ}\\
\bar{r}_J^{ij}(\tau k) &= 0 \\
\tilde O_J^{ij} &= \mu^{-1} \sum_{f \in \F} h_{f,[J]}(\mu t) \Ddt A^{i j}_{(f,J)}. 
}Note that clearly $\tilde{w}_J = \nb_i \tilde{z}_J^i = \nb_i \nb_j \tilde{r}_J^{ij}$, as required at the initial time $\tau k$.

\subsection{Newton Step summary} \label{sec:newtonSummary}

The main Proposition in the Newton step, which summarizes all that is needed to proceed to the convex integration step, is given by Proposition~\ref{prop:newtonBound}.  Establishing this proposition involves estimating on the one hand explicit parametrices ($\tilde{w}_J, \tilde{z}_J, \tilde{r}_J$) for the relevant transport equations, and on the other hand remainders $\bar{w}_J, \bar{z}_J, \bar{r}_J$ that solve the remainder transport equations exactly (compare equations \eqref{eq:NewtonEquation},\eqref{eq:plainForm}, and \eqref{eq:transElliptic} with equations~\eqref{eq:tildewJ}-\eqref{eq:barrJ}).  We use the notation 
\ali{
{\bf F}_J = \{ \bar{w}_J, \bar{z}_J, \bar{r}_J \}
}
to denote the list of tensor fields involved in the main estimate that solve transport type equations for which we require a sharp bound.

Table~\eqref{eq:sizeOfTerms} states how large all of these tensors should be.  The proposition controlling the parametrices as well as the forcing terms $\tilde{O}_J, \tilde{O}^i_J, \tilde{O}^{ij}_J$ in the transport equation for the remainders is Proposition~\ref{prop:tildewBds}.  The main proposition controlling the terms that solve transport equations (${\bf F}_J$) is Proposition~\ref{prop:weightedNorm}, which reduces our desired estimates to applying Gr\"{o}nwall to the appropriate weighted norm.  We reduce Proposition~\ref{prop:weightedNorm} to the more straightforward Proposition~\ref{prop:boiledDownBound} by estimating commutator terms.  The bounds for $\bar{w}_J$ begin on line \eqref{eq:barwJbegins}, those for $\bar{z}_J$ begin on line \eqref{eq:zJbegins}, and those for $r_J$, which require the most substantial novelty, begin on line \eqref{eq:barrJ_recall}.  The estimates for $u_J^\ell = T^\ell w_J$ begin on line \eqref{eq:uJbegins} while those for $\rho_J^{j\ell} = \RR^{j\ell} w_J = \RR^{j\ell} \nb_a\nb_b r_J^{ab}$ begin on line \eqref{eq:rhoJbegins}.  These last two estimates rely on the principle that order zero operators cost a log, including when commuting with an advective derivative (Proposition~\ref{lem:zeroOrder}).  From all these, we prove the claimed estimates for the gluing error $R_{(n+1)}$ beginning on line~\eqref{eq:gluingErrorDef}.

\subsection{Newton Step: the estimates} \label{sec:NewtonStep}

We now recall the estimates we inherit from previous work, where the choice of $\ep$ was the same but slightly different in notation ($\ep = N^{-1/L} \Xi^{-1} = \hn^{-1} \Xi^{-1} = \hxi^{-1}$).  For the coarse scale velocity field and scalar field, we have 
\cite[equation~(6.2)]{isettVicol} the following bounds:
\ali{
\co{\nb_{\va}\th_\ep}+\co{\nb_{\va}u_\ep}
&\lesssim_{\va}
\hn^{(|\va|-L)_+}\Xi^{|\va|}e_u^{1/2},
\label{eq:prelimScalVel1}\\
\co{\nb_{\va}\Ddt\th_\ep}+\co{\nb_{\va}\Ddt u_\ep}
&\lesssim_{\va}
\hn^{(|\va|+1-L)_+}\Xi^{|\va|+1}e_u.
\label{eq:prelimScalVel2}
}
We will refer to derivatives of $(\th, u, R)$ beyond order $L$ as ``borrowed derivatives''.  The idea is that $(\th, u, R)$ come with $L$ derivatives, but factors of $\hn$ start to appear in derivatives of higher order when the derivative hits the mollifier.  The Newton step will lose three derivatives, so during and after the Newton step, borrowed derivatives start to appear after order $\unl \equiv L - 3$. 

We remark some minor differences regarding \eqref{eq:prelimScalVel1}-\eqref{eq:prelimScalVel2}.  In \cite{isettVicol} the proof uses the double convolution structure of the mollification, which is also present in our case since Littlewood Paley projections satisfy 
\[P_{\leq q_\ep} \th = P_{\leq q_\ep + 1} P_{\leq q_\ep} \th = \eta_{\leq q_\ep+1} \ast \eta_{\leq q_\ep} \ast \th.\]  
In \cite{isettVicol}, the  parameters $\Xi$ and $e_v^{1/2}$ play the roles of our $\Xi$ and $e_u^{1/2}$, and the extra loss for borrowed derivatives ($N^{1/L}$ in the source) is denoted here by $\hn$.  The bound \eqref{eq:prelimScalVel1} for $|\va| = 0$ is not present in \cite{isettVicol}, but is immediate from Definition~\ref{defn:freqEnLevels} and Minkowski's inequality.

For the stress, the estimates in \cite[equations~(628)--(630), equation~(661) in the proof of Proposition~18.3, and Proposition~18.6, equation~(699)]{isett} give
\ali{
\co{R - R_{\ep}} &\lesssim N^{-1} D_R \\
\co{\nb_{\va} R_{\ep} } &\lesssim_{\va} \widehat{N}^{(|\va| - L)_+} \Xi^{|\va|} D_R \label{eq:nbRepbd1} \\
\co{\nb_{\va} \Ddt R_{\ep} } &\lesssim_{\va}  (\Xi e_u^{1/2}) \widehat{N}^{(|\va| + 1 - L)_+} \Xi^{|\va|} D_R. \label{eq:nbRepbd2}
} In the cited source the tensor $R_{\epsilon_x}$, stress level $e_R$, velocity level $e_v^{1/2}$, mollified velocity $v_\epsilon$, and material derivative $\overline D/\partial t$ play the roles of our $R_\epsilon$, $D_R$, $e_u^{1/2}$, $u_\epsilon$, and $\Ddt$, respectively.  The hypotheses used are the stress bounds $\|\nabla^{\va}R\|_{C^0}\leq\Xi^{|\va|}D_R$ for $|\va|\leq L$ and $\|\nabla^{\va}D_tR\|_{C^0}\leq\Xi^{|\va|}(\Xi e_u^{1/2})D_R$ for $|\va|\leq L-1$, together with the velocity parts of \eqref{eq:prelimScalVel1}--\eqref{eq:prelimScalVel2}.  Again we use $\hn$ to denote $N^{1/L}$, which is the notation used in the source.

The terms we need to estimate include not only the scalar fields $\bar{w}_J$, $\tilde{w}_J$ and the fields $z_J$ and $r_J$, but also the fields $u_J^\ell = T^\ell w_J$ and a trace-free symmetric tensor field $\rho_J^{i\ell}$ that is defined by $$\rho_J^{i\ell} = \RR^{i\ell} w_J = \RR^{i\ell} \nb_a \nb_b r_J^{ab}$$ and whose second order divergence is $w_J$ (i.e. $\nabla_i\nb_\ell \rho^{i\ell}_J=w_J$).  The operator $\RR^{i\ell}$ is the order $-2$ operator defined in  Section~\ref{sec:notation}.

We will associate to each of these tensor fields $F$ in our problem a positive number $S_F$ that is the ``size'' of $F$.  The following table summarizes the sizes of the fields
\ali{
\mat{c|ccccc}{ F & \bar{w}_J, w_J & \bar{z}_J^i, z_J^i & \bar{r}_J^{ij}, r_J^{ij} & u_J^\ell & \rho_J^{i\ell} \\
		S_F &	\Xi^2 \mu^{-1} D_{R,n}  & \Xi \mu^{-1} D_{R,n} & \mu^{-1} D_{R,n} & \plhxi \Xi^2 \mu^{-1} D_{R,n} & \plhxi \mu^{-1} D_{R,n}  
		},
		\label{eq:sizeOfTerms}
}
Thus $S_w = \Xi^2 \mu^{-1} D_{R,n}$, $S_z = \Xi \mu^{-1} D_{R,n}$, etc.  For convenience we remind the reader of the choice of $\mu = N^{1/2} \Xi e_R^{1/2} = N^{1/2} \Xi^{3/2} D_R^{1/2}$ from \eqref{eq:muChoice}.

At this point it is instructive to note that most of the $C^0$ estimates ($\co{\bar{w}_J} \lesssim S_w$, $\co{\bar{z}_J} \lesssim S_z$, $\co{\tilde{r}_J} \lesssim S_r$) are rather straightforward to foresee.  For the parametrices $\tilde{w}_J, \tilde{z}_J, \tilde{r}_J$, they follow by inspecting the defining formulas using $e_n(t) \lesssim D_{R,n}$.  For $\bar{w}_J$ and $\bar{z}_J$, they come about due to the method of characteristics estimate $\co{\bar{w}_J} \lesssim \tau \co{\Ddt\bar{w}_J}$, $\co{\bar{z}_J} \lesssim \tau \co{\Ddt\bar{z}_J}$ after inspecting the size of the forcing term in the transport equations (note that only three derivatives are taken in these formulas so no ``borrowed derivatives'' appear at the level of $C^0$); the other terms will be treated by Gronwall using the principle that order zero operators cost a log.  The most challenging term to estimate is $\bar{r}_J$, which requires an involved harmonic analysis of the $\BB$-term in the equation.  Once the estimates for $w_J$ and $r_J$ are proven, the estimates for $u_J$ and $\rho_J$ follow from the general principle that zero-order operators are not bounded on $C^0$ but rather cost a factor of log, including when commuting with advective derivatives (equation \eqref{eq:orderZeroBound} and Lemma~\ref{lem:zeroOrder}).

We are now ready to estimate the terms in the Newton step. Define $$\underline L:= L-3,$$ which will be the new threshold for borrowed derivatives, since formulating the Newton step costs three derivatives.  The following is the main result of this section.
\begin{prop} \label{prop:newtonBound}  For all $F \in {\bf F }_J \cup \{ w_J, z_J, r_J \}$, we have the estimates
\ali{
\co{ \nb_{\va} F } &\lesssim_{\va} \widehat N^{(|\va| -\unl)_+} \Xi^{|\va|} S_F, \label{eq:newtonStepBoundnbaF} }
Moreover, we have the following bounds for $w_n$, $u_n^\ell = T^\ell w_n$ and $\rho_n^{j\ell} = \RR^{j\ell} w_n$
\ali{
\co{\nb_{\va} |\nb|^{-1/2} w_n } &\ls \Xi^{|\va|} D_{R,n}^{1/2}, \qquad 0 \leq |\va| \leq 1 \label{eq:sizeOfNewtCorrect} \\
\co{\nb_{\va}  u_n} &\lesssim_{\va} \widehat N^{(|\va| -\unl)_+} \Xi^{|\va|} S_u \label{eq:newtVelocBd}\\
\left\| \nabla_{\vec{a}} \overline{D}_t^r u_n \right\|_{C^0} &\lesssim_{\va} \widehat{N}^{(r+|\vec{a}| -\underline L)_+} \Xi^{|\vec a|}(\Xi e_u^{1/2})^r e_u^{1/2},  \qquad 0 \leq r \leq 1  \label{eq:newtVelocBd2} \\
\co{\nb_{\va} \Ddt^r \rho_J} &\lesssim_{\va} \widehat N^{(|\va| +r -\unl)_+} \Xi^{|\va|} \tau^{-r} S_{\rho} \qquad \mbox{on } \supp \tilde{\chi}_k'(t) \label{eq:newtTrfreeBd}
}

Furthermore there exists a symmetric, trace-free tensor field $R_{(n+1)}$ that solves \eqref{eq:gluingError2} and satisfies the bounds for \ali{
\co{R_{(n+1)}} &\leq D_{R,n+1} \label{eq:RnC0} \\\co{\nb_{\va} R_{(n+1)}} &\lesssim_{\va} \widehat{N}^{(|\va| - L)_+} \Xi^{|\va|} D_{R,n+1} \label{eq:Rnplus1}\\
\co{\nb_{\va} \Ddt R_{(n+1)}} &\lesssim_{\va} \widehat{N}^{(|\va| + 1 - L)_+} \Xi^{|\va|} (\Xi e_u^{1/2}) D_{R,n+1} \label{eq:Rnplus1Ddt} \\
\suppt (w_n,R_{(n+1)}) &\subseteq J_{n+1} \label{eq:newtonTimeEnlarge}
}

\end{prop}
Note that \eqref{eq:RnC0} has implicit constant 1.

We will need the following bounds on the phase functions.
\begin{prop}\label{prop:phasegradients}  For $J=(k,f)$ and $|t-k\tau|\leq\tau$, the phase function gradients satisfy
\ali{
\co{\nb_{\va} \nb \check{\xi}_J} &\lesssim_{\va} \hn^{(|\va| + 1 - L)_+} \Xi^{|\va|} \\
\co{\nb_{\va} \Ddt \nb \check{\xi}_J} &\lesssim_{\va} \hn^{(|\va| + 1 - L)_+} \Xi^{|\va|+1} e_u^{1/2}
}
\end{prop}
These are Proposition~17.3 (equation~(558)) and Proposition~17.5 (equation~(593)) of \cite{isett}, applied with \eqref{eq:prelimScalVel1}--\eqref{eq:prelimScalVel2}.  In the source, $v_\epsilon$, $\overline D/\partial t$, $e_v^{1/2}$, and $N^{1/L}$ are our $u_\epsilon$, $\Ddt$, $e_u^{1/2}$, and $\hn$.  Thus the source estimates
\[
\|\nabla^m\xi\|_{C^0}\lesssim_m \Xi^{m-1}N^{(m-L)_+/L},\qquad
\|\nabla^a(\overline D/\partial t)\nabla\xi\|_{C^0}
\lesssim_a \Xi^{a+1}e_v^{1/2}N^{(a+1-L)_+/L}
\]
give the two displayed bounds above, on the stated lifespan, after writing $N^{1/L}=\hn$.  Besides the frequency--energy hypotheses in Definition~\ref{defn:freqEnLevels}, the argument uses the transport equation \eqref{eq:transportOfPhase}, its affine data at $t=k\tau$, and the restriction $|t-k\tau|\leq\tau\lesssim(\Xi e_u^{1/2})^{-1}$.

The following weighted norm will be useful.
\begin{defn}  The start-weighted norm of a function $F$ is \ali{
H_{\zeta,M}^{(R)}[F] &= \max_{0 \leq r \leq R} \max_{0 \leq |\va| + r \leq M} \fr{\|\nb_{\va} \Ddt^r F \|_{C^0}}{\hn^{(|\va| + 1 - L)_+}\Xi^{|\va|} \zeta^r }
}    
Note that $R \in \{ 0, 1\}$ is a number, not to be confused with the stress tensor.

\end{defn}

When we run into terms that involve a mix of spatial and advective derivatives, the following Lemma is useful. This lemma will be applied to $\tilde O$.
\begin{lem} \label{lem:mixed derivatives} For any multi-indices $\va, \vcb$ such that $|\va| + |\vcb| \leq M$ and for $\zeta \geq \Xi e_u^{1/2}$ we have
\ali{
\co{\nb_{\va} \Ddt \nb_{\vcb} F} &\ls \hn^{(|\va| + |\vcb| + 1 - L)_+} \Xi^{|\va| + |\vcb|} \zeta  H_{\zeta,M}^{(1)}[F]
}    \end{lem}
\begin{proof}  Let $M$ be given.  We proceed by induction on $|\vcb| \leq M$.  The case $|\vcb| = 0$ follows directly from the definition of $H_{\zeta,M}^{(1)}[F]$.  Now assume the bound holds for $|\vcb| - 1$, and write $\nb_{\vcb} = \nb_{b_1} \nb_{\check{b}}$ where $|\check{b}| = |\vcb| - 1$.  We have
\ALI{
\nb_{\va} \Ddt \nb_{\vcb}F &= \nb_{\va} \nb_{b_1} \Ddt \nb_{\check{b}}F - \nb_{\va}[\nb_{b_1} u_\ep^i \nb_i \nb_{\check{b}} F ]\\
\co{\nb_{\va} \Ddt \nb_{\vcb}F} &\leq \hn^{(|\va| + |\vcb| + 1 - L)_+} \Xi^{|\va| + |\vcb|} \zeta  H_{\zeta,M}^{(1)}[F] \\
&+ \overset{\sim}{\sum} \co{\nb_{\va_1} \nb_{b_1} u_\ep^i} \co{\nb_{\va_2} \nb_i \nb_{\check{b}} F }\\
&\lsm\hn^{(|\va| + |\vcb| + 1 - L)_+} \Xi^{|\va| + |\vcb|} \zeta H_{\zeta,M}^{(1)}[F] \\
&+ \tisum \hn^{(|\va_1| + 1 - L)_+} \hn^{( |\va_2| + 1 + |\check{b}| - L)_+} \Xi^{|\va| + |\vcb|} \Xi e_u^{1/2}  H_{\zeta,M}^{(1)}[F]  \\
&\lsm \hn^{(|\va| + |\vcb| + 1 - L)_+} \Xi^{|\va| + |\vcb|} \zeta H_{\zeta,M}^{(1)}[F]
}

\end{proof}

We also have a chain rule for the weighted norm.
\begin{lem} \label{lem:H^0}   $K$ be a compact neighborhood of the image of $(\check{R} = R_{(n)} /( M_e D_{R,n}), \nb \xi_k)$ and let $G$ be $C^\infty$ on a neighborhood of $K$.  Then
\ali{
H_{\Xi e_u^{1/2},M}^{(R)}[G(\check{R}, \nb \xi_k)]
&\ls_{M,K, G} 1
}
\end{lem}
\begin{proof}  By the chain and product rules we have
\ALI{
\nb_{\va} \Ddt^r G(\check{R}, \nb \xi_k) &= \sum_{m=0}^{|\va| + R} \overset{\sim}{\sum} \pr^m G \prod_{i=1}^{m_1} \nb_{\va_i} \Ddt^{r_i} \check{R} \cdot \prod_{j=1}^{m_2}  \nb_{\va_j} \Ddt^{r_j} \nb \xi_{k} 
}
where the sum ranges over indices such that $\sum |\va_i| + \sum |\va_j| = |\va|$ and $\sum_i r_i + \sum_j r_j = r$ and the empty product is $1$.  Hence, \ALI{
\co{\nb_{\va} \Ddt^r G(\check{R}, \nb \xi_k)} &\ls \sum_{m=0}^{|\va| + R}\overset{\sim}{\sum} \prod_{i=1}^{m_1} \hn^{(|\va_i| + r_i- L)_+} \Xi^{|\va_i|} (\Xi e_u^{1/2})^{r_i} \cdot \\
&\,\,\cdot \prod_{j=1}^{m_2}  \hn^{(|\va_j| +1 - L)_+} \Xi^{|\va_j|} (\Xi e_u^{1/2})^{r_j} \\
&\ls \sum_{m=0}^{|\va| + R} \Xi^{|\va|} (\Xi e_u^{1/2})^r \overset{\sim}{\sum} \hn^{(\sum_i(|\va_i| + r_i)- L)_+} \hn^{(\sum_j |\va_j| + 1 - L)_+}
}

where the last line we used the counting inequality with $z = L$ and $z = L -1$.  The proof now follows from 
\ALI{
&(\sum_i (|\va_i| + r_i)- L)_+ + (\sum_j |\va_j| + 1 - L)_+ \leq  \\
&\leq (\sum_i|\va_i| + 1 - L)_+ + (\sum_j |\va_j| + 1 - L)_+ \leq (|\va| + 1 - L)_+ 
}
\end{proof}

\begin{cor}
\ali{
H_{\tau^{-1},M}^{(1)}[A_{(f,J)}^{j\ell}] &\lesssim D_{R,n} \label{eq:AfJbound}
}
\end{cor}
\begin{proof}  This follows quickly from the previous Lemmas applied to \eqref{eq:AfJform}.  Note that the cost of an advective derivative is $\tau^{-1}$ rather than $(\Xi e_u^{1/2})$ due to the factors of $\chi_k^2$ and $e_n(t)$.
\end{proof}

The following proposition summarizes the bounds on terms that do not solve a transport equation 
\begin{prop}\label{prop:tildewBds} For all $\va$, we have the bounds
\ali{
\co{ \nb_{\va}  \tilde{w}_J} + \hxi^{-\a} \cda{\nb_{\va}  \tilde{w}_J} &\lsm_{\va} \hn^{(|\va| - \unl)_+} \Xi^{|\va|}  S_w \\
\co{ \nb_{\va}  \tilde{z}_J} + \hxi^{-\a} \cda{\nb_{\va}  \tilde{z}_J} &\lsm_{\va} \hn^{(|\va| - \unl)_+} \Xi^{|\va|} S_z \\
\co{ \nb_{\va}  \tilde{r}_J} + \hxi^{-\a} \cda{\nb_{\va}  \tilde{r}_J} &\lsm_{\va} \hn^{(|\va| - \unl)_+} \Xi^{|\va|}  S_r 
}
and
\ali{
\co{\nb_{\va} \tilde{O}_J} + \hxi^{-\a} \cda{\nb_{\va}  \tilde{O}_J} &\lsm_{\va} \hn^{(|\va| - \unl)_+} \Xi^{|\va|} \tau^{-1} S_w \\
\co{\nb_{\va} \tilde{O}_J^i} + + \hxi^{-\a} \cda{\nb_{\va}  \tilde{O}_J^i} &\lsm_{\va} \hn^{(|\va| - \unl)_+} \Xi^{|\va|} \tau^{-1} S_z \\
\co{\nb_{\va} \tilde{O}_J^{ij}} + \hxi^{-\a} \cda{\nb_{\va}  \tilde{O}_J^{ij}} &\lsm_{\va} \hn^{(|\va| - \unl)_+} \Xi^{|\va|} \tau^{-1} S_r 
}

\end{prop}

\begin{proof}
It suffices to prove the bounds for the $C^0$ norms since they imply the bounds on the $\dot{C}^\a$ norms by interpolation.    We only prove the bounds for $\ti w_J$ and $\ti O_J$ since the other bounds will then be similar. 
    The bound for $\|\nb_{\va}\ti w_J\|_0 $ follows from \eqref{eq:Rnplus1}.

    The bound for $\|\nb_{\va} \ti O_J\|_0$ follows by taking $\nb_{\vec b} = \operatorname{div}\operatorname{div}$ (thus $|\vec b|=2$) and $F=A_J$ in Lemma~\ref{lem:mixed derivatives}. This choice yields 
    \begin{align*}
        \|\nb_{\va}\ti O_J\|_0&\ls \mu^{-1} \hn^{(|\va| + 3 - L)_+} \Xi^{|\va|+2} \tau^{-1} H_{\tau^{-1},M}^{(1)}[A_J] ,\\
        &\ls \mu^{-1} \hn^{(|\va| - \unl)_+} \Xi^{|\va|+2} \tau^{-1} D_{R,n}\\
        &= \hn^{(|\va| - \unl)_+} \Xi^{|\va|} \tau^{-1} S_w 
    \end{align*}which is the desired bound. \end{proof}

Fix some\footnote{The value of $\a$ is unimportant -- the point is the well-posedness of the equations holds in $C^{k,\a}$ spaces but not in $C^k$ spaces.} $\a \in (0,1)$.  The proof of Proposition~\ref{prop:newtonBound} relies on the following weighted norm:
\ali{
\hh(t) &= \sum_{F \in {\bf F}_J} \sum_{|\va| \leq L'} (S_F \hn^{(|\va| -\unl)_+} \Xi^{|\va|})^{-1} \left( \co{\nb_{\va} F} + \hxi^{-\a} \cda{\nb_{\va} F}\right). \label{eq:defhh}
}
Here and in what follows we suppress the dependence of $\hh(t)$ on the index $J$ and on $L'$. We write $\hh_{L'}$ to emphasize dependence on $L'$. Notice that $\hh(t)$ vanishes at the initial time $t_J = t_{(k,n)} = k\tau$.  In the following analysis, we simplify notation by assuming the initial time is $t_J = 0$.
\begin{prop}[Main proposition in the Newton step] \label{prop:weightedNorm} We have the estimate
\ali{
\hh(t) &\le C \plhxi \Xi e_u^{1/2} \int_{[0,t]} (1 + \hh(s) ) ds.
}
for some $C>0$ independent of the frequency energy levels $(\Xi, D_u, D_R)$ and independent of $N$, but $C$ is allowed to depend on the step $n$ of the Newton iteration and on $L'$, the order of differentiation that $\hh$ controls.

In particular, by Gronwall, $\hh(t) \lesssim 1$ for $|t| \leq \tau$.
\end{prop}
Recall the notation $\hq$ for the integer satisfying $\hxi \sim 2^{\hq}$.  The following criterion will be useful for bounding $\hh(t)$
\begin{lem} \label{lem:cdaCharacterization} For any function $f\in C(\T^d)$ and any multi-index $\va$ we have that
\ali{
\co{\nb_{\va} f} + \hxi^{-\a} \cda{\nb_{\va} f} &\lesssim \co{\nb_{\va} f} + \hxi^{-\a}\sup_{q > \hq} 2^{\a q} \co{ P_q \nb_{\va} f }
}(In fact, the two sides are equivalent up to constants.)
\end{lem}
Indeed, this lemma follows quickly from the following standard Littlewood-Paley characterization of $\dot{C}^\a$ seminorm, $$\cda{f} \sim \sup_q 2^{\a q} \| P_q f \|_{C^0},$$ which is valid for continuous $f$.  (A proof can be found in \cite[Proposition~9.1 and Appendix~A, proof of Proposition~9.1]{isettOnsag}.)

We will also use the following Lemma about commuting spatial derivatives and Littlewood-Paley projections with the advective derivative. We define
$$\co{f(t)} := \sup_{x} |f(t,x)|.$$
\begin{lem} \label{lem:startingToEstimateNewton} For any $F \in {\bf F}_J$ and any multi-index $\va$ of order $|\va| \leq L'$ we have
\ali{
\co{\nb_{\va} F(t) } &\leq \int_0^t \co{\Ddt \nb_{\va} F(s) } ds \label{eq:methChars}\\
\Ddt \nb_{\va} F(s) &= \nb_{\va} \Ddt F(s) + {O}\left(\hn^{(|\va| -\unl)_+} \Xi^{|\va|+1} e_u^{1/2} S_F \hh(s)\right), \label{eq:simpleCommutator} \\
\begin{split}
\Ddt \nb_{\va} P_q F(s) &= P_q \nb_{\va} \Ddt F(s) \\
&+ \min\{ 1,2^{-\a q} \hxi^{\a} \}\, {O}\left(\hn^{(|\va| -\unl)_+} \Xi^{|\va|+1} e_u^{1/2} S_F\hh(s)\right) \label{eq:3}
\end{split}
}\end{lem}
\begin{proof}  Inequality \eqref{eq:methChars} is a consequence of the method of characteristics.  For the second estimate, we compute 
\ALI{
[\nb_{\va}, \Ddt] F &= \overset{\sim}{\sum_{\va_1, \va_2}}\nb_{\va_1} u_\ep^i \nb_{\va_2} \nb_i F \, 1_{|\va_1| \geq 1} \\
\co{[\nb_{\va}, \Ddt] F} &\lesssim \overset{\sim}{\sum_{\va_1, \va_2}} [\hn^{(|\va_1| - L)_+} \Xi^{|\va_1|} e_u^{1/2}] [\hn^{(|\va_2| + 1 - \unl)_+} \Xi^{|\va_2|} S_F]
}
where the sum ranges over multi-indices with $|\va_1| + |\va_2| = |\va|$.  Now employ the counting inequality with $(|\va_1| - 1, |\va_2| + 1, \unl - 1)$, all $\geq 0$, to get
\[ (|\va_1| - L)_+ + (|\va_2| + 1 - \unl)_+ \leq (|\va_1| - 1 + 1 - \unl)_+ + (|\va_2| + 1 - \unl)_+ \leq (|\va_1| - 1 + |\va_2| + 1 - \unl)_+ = (|\va| - \unl)_+ \]
Inequality \eqref{eq:simpleCommutator} follows.

The last bound in this lemma is a special cases of Lemma~\ref{lem:zeroOrder} below where we take ${\cal Q}$ to be the identity map. 
\end{proof}

Based on Lemmas~\ref{lem:cdaCharacterization} and \ref{lem:startingToEstimateNewton}, the proof of the Main Proposition (Proposition~\ref{prop:weightedNorm}) reduces to the following
\begin{prop} \label{prop:boiledDownBound} For any $F \in {\bf F}_J$ we have for all $|\va| \leq L'$ and all $q > \hq$ the bounds
\ali{
\co{\nb_{\va} \Ddt F(s)} \lesssim &\hn^{(|\va| -\unl)_+} \Xi^{|\va|} \tau^{-1} S_F (1 + \hh(t)) \\
\co{P_q \nb_{\va} \Ddt F(s)} \lesssim 2^{-\a q} \hxi^{\a} &\hn^{(|\va| -\unl)_+} \Xi^{|\va|} \tau^{-1} S_F (1 + \hh(t))}
\end{prop}

The following standard spatial derivative bounds will be used.
\begin{prop} \label{prop:standardDegZero}  If $\QQ$ is a convolution operator whose symbol is degree $0$ homogeneous and smooth away from $0$ then  
\ali{
\co{\nb_{\va} P_q \QQ F} &\lesssim 2^{-\a q} \cda{\nb_{\va} F } \label{eq:standardeasy}  \\
\cda{\nb_{\va} \QQ F} &\ls \cda{\nb_{\va} F} \\
\co{\nb_{\va} \QQ F} &\lesssim \plhxi \co{ \nb_{\va} F} + \hxi^{-\a} \cda{\nb_{\va} F}  \label{eq:orderZeroBound}
}
\end{prop}
\begin{proof}[Proof of Proposition~\ref{prop:standardDegZero}]
Let $\eta_q(h) = 2^{d q} \eta_0(2^q h)$ be the convolution kernel representing $P_q \QQ$.  Then $\eta_q$ has integral $0$ and \eqref{eq:standardeasy} follows from
\ALI{
\nb_{\va} P_q \QQ  F &= \int (\nb_{\va}F(x + h) - \nb_{\va}F(x)) \eta_q(h) dh \\
|\nb_{\va}P_q \QQ F| &\leq \cda{\nb_{\va}F} \int |h|^\a |\eta_q(h)| dh.
}
The second bound follows from the first one and the Littlewood Paley characterization of H\"{o}lder spaces.  The third estimate is obtained by summing
\ALI{
\co{\nb_{\va} \QQ F} &\ls \sum_{q=0}^{\hq} \co{P_q \QQ \nb_{\va} F} + \sum_{q = \hq}^\infty 2^{-\a q}\cda{\nb_{\va} F} 
}
using the fact that $\| P_q \QQ \| \lesssim 1$ as an operator on $C^0$.
\end{proof}

\begin{lem} \label{lem:velocBound1}For $|\va| \leq L'$ we have
    \ali{\| \nb_{\va} u_J \|_{C^0} + \hxi^{-\a} \| \nb_{\va} u_J \|_{\dot{C^{\alpha}}} \ls \hn^{(|\va|-L)_+} \Xi^{|\va|}S_u (1 + \hh). \label{eq:velocBounded1}
    }
\label{lem:u in 0 and a}
\end{lem}
\begin{proof}
We use Proposition~\ref{prop:standardDegZero} with $\QQ = T$ and $F = w_J$. This result gives for $|\va| \leq L'$:
\begin{align*}
\co{\nb_{\va} u_J} &= \co{\nb_{\va} T w_J} \nonumber \lesssim \plhxi \co{ \nb_{\va} w_J} + \hxi^{-\alpha} \cda{\nb_{\va} w_J} \nonumber.\end{align*}
This is in turn bounded by
\begin{align*}
&\lesssim \plhxi \hn^{(|\va| -\unl)_+} \Xi^{|\va|} S_w (1+\hh) + \hn^{(|\va| -\unl)_+} \Xi^{|\va|} S_w (1+\hh)\\
&\lesssim  \hn^{(|\va| -\unl)_+} \Xi^{|\va|} S_u  (1+\hh) \label{eq:last} 
\end{align*}
where we have used the definition of $\hh=\hh_{L'}$.

For the $\dot C^\a$ bound, we have 
\ali{
\cda{\nb_{\va}u_J}&\ls \cda{\nb_{\va}w_J} \ls \hxi^\a S_w \hn^{(|\va|-\unl)_+}\Xi^{|\va|}(\hh+1)
}
where the first inequality follows by Proposition~\ref{prop:standardDegZero} and the second inequality follows from the definition of $\hh$.
\end{proof}

\begin{proof}[Proof of Proposition~\ref{prop:boiledDownBound} for $\bar{w}_J$]

For the $C^0$ bound on $\nb_{\va} \Ddt \bar{w}_J$, we use the equation \eqref{eq:barwJ}:
\begin{align}
\co{\nb_{\va} \Ddt \bar{w}_J} &\lesssim \co{\nb_{\va}(T^\ell w_J \nb_\ell \th_\ep)} + \co{\nb_{\va} \tilde{O}_J} \\
&\lesssim \tisum \co{\nb_{\va_1} T^\ell w_J} \co{\nb_{\va_2} \nb_\ell \th_\ep} + \co{\nb_{\va} \tilde{O}_J}. \label{eq:barwJbegins}
\end{align}

For the first term, we use the bounds of Lemma~\ref{lem:velocBound1} for $u_J = T^\ell w_J$ and the bounds in the Main Lemma for $\th_\ep$:
\begin{align*}
\co{\nb_{\va_1} T^\ell w_J} \co{\nb_{\va_2} \nb_\ell \th_\ep} &\lesssim \hn^{(|\va_1| -\unl)_+} \Xi^{|\va_1|} S_u(1 + \hh(t)) \cdot \hn^{(|\va_2| -\unl)_+} \Xi^{|\va_2|+1} e_u^{1/2} \\
&\lesssim \hn^{(|\va| -\unl)_+} \Xi^{|\va|+1} e_u^{1/2} S_u (1 + \hh(t)).
\end{align*}

For the $\|\nb_{\va}\ti O_J\|_0$ term, a sufficient bound was already obtained in Proposition~\ref{prop:tildewBds}.

For the high frequency bound, we recall the equation for $\bar{w}_J$ \eqref{eq:barwJ}:
\begin{align*}
\pr_t \bar{w}_{J} + T^\ell \th_\ep \nb_\ell \bar{w}_J &= - T^\ell w_J \nb_\ell \th_\ep - \tilde{O}_J
\end{align*}
We apply $P_q \nb_{\va}$ to both sides of the equation:
\begin{align*}
P_q \nb_{\va} \Ddt \bar{w}_J
&= -P_q \nb_{\va}(T^\ell w_J \nb_\ell \th_\ep) - P_q \nb_{\va} \tilde{O}_J
\end{align*}
By the product rule,
\begin{align*}
P_q \nb_{\va} \Ddt \bar{w}_J
&= -\tisum P_q (\nb_{\va_1} T^\ell w_J \, \nb_{\va_2} \nb_\ell \th_\ep) - P_q \nb_{\va} \tilde{O}_J
\end{align*}
Then by the triangle inequality and the Holder inequality for the $C^0$ norm,
\begin{align*}
\co{P_q \nb_{\va} \Ddt \bar{w}_J}
&\leq \tisum \co{P_q (\nb_{\va_1} T^\ell w_J \, \nb_{\va_2} \nb_\ell \th_\ep)} + \co{P_q \nb_{\va} \tilde{O}_J}
\end{align*}

By the $C^0$ bounds of Lemma~\ref{lem:velocBound1} for $u_J = T^\ell w_J$, the bounds in \eqref{eq:prelimScalVel1} for $\th_\ep$, and the $C^0$ bound for $\|\nb_{\va}\ti O_J\|_0$ from Proposition~\ref{prop:tildewBds}: 
\begin{align*}
\co{P_q \nb_{\va} \Ddt \bar{w}_J}
&\lesssim 2^{-\alpha q} \left( \tisum \co{\nb_{\va_1} u^\ell_J} \cda{\nb_{\va_2} \nb_\ell \th_\ep} + \cda{\nb_{\va_1} u^\ell_J} \co{\nb_{\va_2} \nb_\ell \th_\ep} \right) \\
&\quad + \co{P_q\nb_{\va} \tilde{O}_J} \\
&\lesssim 2^{-\alpha q} \tisum \hn^{(|\va_1| -\unl)_+} \Xi^{|\va_1|} S_u(1+\hh(t)) \cdot \hxi^\alpha \hn^{(|\va_2|-(L-1))_+} \Xi^{|\va_2|+1}e_u^{1/2} \\
&\quad + 2^{-\alpha q} \tisum \hxi^\a \hn^{(|\va_1| -\unl)_+} \Xi^{|\va_1|} S_w(1+\hh(t)) \cdot \hn^{(|\va_2|-(L-1))_+} \Xi^{|\va_2|+1}e_u^{1/2}\\
&\quad + 2^{-\a q} \cda{\nb_{\va} \tilde{O}_J} \\
&\ls \left(2^{-\a q}\tisum \hn^{(|\va|-\unl)_+ }\Xi^{|\va|+1}e_u^{1/2}S_u(1+\hh)\hxi^\a\right) \\
&\ \ + 2^{-\a q} \hn^{(|\va|-\unl)_+ }\Xi^{|\va|}\tau^{-1}S_w \hxi^\a\\
&\ls 2^{-\a q} \hn^{(|\va|-\unl)_+ }\Xi^{|\va|}\tau^{-1}S_w(1+\hh)\hxi^\a
\end{align*}
where the second line follows from Lemma~\ref{lem:u in 0 and a} and in the last line we recall $\tau^{-1} \sim_{L,\eta} \plhxi (\Xi e_u^{1/2})$ and $S_u = \plhxi S_w$. We have bounded the $\tilde O_J$ term using Proposition~\ref{prop:tildewBds}.  

The bound we used on $\cda{\nb_{\va_2}\nb_\ell \th_\ep}$ above follows by the interpolation $\|F\|_{\dot C^\alpha} \ls \co{F}^{1-\a}\co{\nabla F}^\a$ with $F = \nb_{\va_2}\nb_\ell \th_\ep$, which yields 
$$\cda{\nb_{\va_2}\nb_\ell \th_\ep} \ls \hn^{(|\va_2|+1-L)_+} \hn^\alpha \Xi^{|\va_2|+1+\a}e_u^{1/2}$$

Thus
\begin{align*}
\co{P_q \nb_{\va} \Ddt \bar{w}_J}
&\lesssim 2^{-\alpha q} \hxi^\alpha \hn^{(|\va| -\unl)_+} \Xi^{|\va|+1} e_u^{1/2} S_u (1+\hh(t))\\
&\le 2^{-\alpha q} \hxi^\alpha \hn^{(|\va| -\unl)_+} \tau^{-1} S_w(1 + \hh(t) ).
\end{align*}
\end{proof}

For the purpose of estimating the velocity increment $u_J^\ell = T^\ell w_J$ and $\rho_J = \RR^{j\ell} w_J$, we will use the following estimates, with $\QQ$ being either $T$ or $\calR \operatorname{div} \operatorname{div}$: 
\begin{lem}\label{lem:zeroOrder}  Suppose $\QQ$ is a Fourier-multiplier with a degree zero homogeneous symbol that is smooth away from the origin.  Then for any $F \in \{ w_J \} \cup \{ r_J \}$ and $|\va| \leq L'$
\ali{
\overline{D}_t \nb_{\va} \QQ F(t) &= \QQ \nb_{\va} \overline{D}_t F(t) + \plhxi \,O\left(\hn^{(|\va| -\unl)_+} \Xi^{|\va|+1} e_u^{1/2} S_F (1 + \hh(t))\right) \label{eq:QcommutDdt} \\
\begin{split} \label{eq:DtPqQF}
\overline{D}_t \nb_{\va} \QQ P_q F(t) &= P_q \QQ \nb_{\va} \overline{D}_t F(t) \\
&+ \min\{ 1, 2^{-\a q} \hxi^{\a} \} \,O\left(\hn^{(|\va| -\unl)_+} \Xi^{|\va|+1} e_u^{1/2} S_F (1 + \hh(t))\right)     
\end{split}
}(Note that $P_q$ commutes with $\nb_{\va}$ and with $\QQ$ but not with $\Ddt$.)
\end{lem}

\begin{proof}[Proof of \eqref{eq:DtPqQF}]
Start with
\ALI{
\nb_{\va} \Ddt F &= \Ddt \nb_{\va} F + [\nb_{\va}, \Ddt] F\\
[\nb_{\va}, \Ddt] F &= \overset{\sim}{\sum}\nb_{\va_1} u_\ep^i \nb_{\va_2} \nb_i F \, 1_{|\va_1| \geq 1}
} where the sum is over certain  ${\vec{a}_1, \vec{a}_2}$ with $|\va_1| + |\va_2| = |\va|$.
Now apply the operator $\QQ_q = \QQ P_q$ to obtain
\ALI{
\QQ_q \nb_{\va} \Ddt F &= \Ddt \QQ_q \nb_{\va} F + [\QQ_q, \Ddt] \nb_{\va} F + \QQ_q [\nb_{\va}, \Ddt] F 
}
We start with the third term.

Since $\QQ_q$ localizes an order $0$ operator to frequency $2^q$, we have $\co{Q_q[f]} \lesssim \min \{ \co{f}, 2^{-\a q} \cda{f} \}$, hence
\ALI{
&\co{\QQ_q [\nb_{\va}, \Ddt] F } \lesssim \overset{\sim}{\sum} \co{ \nb_{\va_1} u_\ep^i} \co{\nb_{\va_2} \nb_i F } 1_{|\va_2| \leq L' - 1} \\
&\lesssim \overset{\sim}{\sum} [\hn^{((|\va_1| - 1) + 1 - L)_+} \Xi^{|\va_1|} e_u^{1/2}] 1_{|\va_1| - 1 \geq 0} [\hn^{(|\va_2| + 1 -\unl)_+} \Xi^{|\va_2|} S_F (1 + \hh(t))] \\ 
&\lesssim \hn^{(|\va| -\unl)_+} \Xi^{|\va|} S_F (1 + \hh(t))
}
where in the last line we applied the counting inequality with $(|\va_1| - 1, |\va_2| + 3, \unl)$ all $\geq 0$.
We also have
\ALI{
&\co{\QQ_q [\nb_{\va}, \Ddt] F} \lesssim\\
&\ls 2^{-\a q} \overset{\sim}{\sum} (\cda{ \nb_{\va_1} u_\ep^i} \co{\nb_{\va_2} \nb_i F } + 
 \co{\nb_{\va_1} u_\ep^i} \cda{\nb_{\va_2} \nb_i F } )1_{|\va_2| \leq L' - 1} \\
&\lesssim \overset{\sim}{\sum} 2^{-\a q} \hxi^{\a} [\hn^{((|\va_1| - 1) + 1 - L)_+} \Xi^{|\va_1|} e_u^{1/2}] 1_{|\va_1| - 1 \geq 0} [\hn^{(|\va_2| + 1 -\unl)_+} \Xi^{|\va_2|} S_F (1 + \hh(t))] \\ 
&\lesssim 2^{-\a q} \hxi^{\a} \hn^{(|\va| -\unl)_+} \Xi^{\va} S_F (1 + \hh(t))
}
We conclude by estimating
\ALI{
[\Ddt, \QQ_q] \nb_{\va} F &= \int (u_{\epsilon}^j(x-h) - u^j_{\epsilon}(x)) \nb_j \nb_{\va} F(x-h) Q_q(h) dh \\
&= - \int (u_{\epsilon}^j(x-h) - u^j_{\epsilon}(x)) \nb_j^{(h)} \nb_{\va}F(x-h) Q_q(h)  dh \\
&= - \int (u_{\epsilon}^j(x-h) - u^j_{\epsilon}(x)) \nb_j^{(h)} (\nb_{\va}F(x-h)-\nb_{\va}F(x)) Q_q(h)  dh \\
&= \int (u_{\epsilon}^j(x-h) - u^j_{\epsilon}(x)) (\nb_{\va}F(x-h)-\nb_{\va}F(x)) \nb_j Q_q(h)  dh \\
\co{ [\Ddt, \QQ_q] \nb_{\va} F }&\ls \|\nabla u_\epsilon\|_{C^0} \|\nb_{\va}F\|_{\dot C^\alpha} \int |h|^{1+\alpha} |\nabla Q_q(h)| dh \\
&\ls [\Xi e_u^{1/2}] [\hn^{(|\va| -\unl)_+} \Xi^{|\va|}  S_F  (\hh(t)+1) \hxi^\a ] [2^{-\alpha q}]
}
and a similar integration by parts yields 
\ALI{
[\Ddt, \QQ_q] \nb_{\va} F &= \int (u_{\epsilon}^j(x-h) - u^j_{\epsilon}(x)) \nb_j \nb_{\va} F(x-h) Q_q(h) dh \\
&= - \int (u_{\epsilon}^j(x-h) - u^j_{\epsilon}(x)) \nb_j^{(h)} \nb_{\va}F(x-h) Q_q(h)  dh \\
\co{[\Ddt, \QQ_q] \nb_{\va} F} &= \|\nabla u_\epsilon\|_{C^0} \|\nb_{\va} F\|_{C^0} \int |h|^{1} |\nabla Q_q(h)| dh \\
&\ls [\Xi e_u^{1/2}] [\hn^{(|\va| -\unl)_+} \Xi^{|\va|}  S_F  (\hh(t)+1) ] [1].
}
Combining these bounds concludes the proof.

\end{proof}

\begin{proof}[Proof of \eqref{eq:QcommutDdt}]
For $|\va| \leq L'$ write
\ALI{
[\nb_{\va}\QQ, \Ddt] F &= \sum_{q = 0}^{\infty} [\nb_{\va} \QQ_q, \Ddt] F \\
\co{[\nb_{\va}\QQ, \Ddt] F} &\leq \sum_{q = 0}^{\hq - 1} \Xi e_u^{1/2} [\hn^{(|\va| - \unl)_+} \Xi^{|\va|} S_F ( 1 + \hh(t) ) ] \\
&+ \sum_{q = \hq}^\infty  2^{-\a q} \hxi^{\a} \Xi e_u^{1/2} [\hn^{(|\va| - \unl)_+} \Xi^{|\va|} S_F ( 1 + \hh(t)) ] 
}
Since $\hq \ls \log \hxi $ and  $2^{-\a \hq} \hxi^{\a} \ls 1$, we obtain the desired bound.
\end{proof}

\begin{proof}[Proof of Proposition~\ref{prop:boiledDownBound} for $\bar z_J$]We need to show that for all $|\va| \leq L'$ and all $q > \hq$, we have the bounds
\ali{
\co{\nb_{\va} \Ddt \bar{z}_J(s)} + 2^{\a q} \hxi^{-\a} \co{P_q \nb_{\va} \Ddt \bar{z}_J(s)}  &\lesssim \hn^{(|\va| -\unl)_+} \Xi^{|\va|} \tau^{-1} S_z (1 + \hh(t)) \label{eq:zJbegins}
}

For the $C^0$ bound on $\nb_{\va} \Ddt \bar{z}_J$, we use the equation \eqref{eq:barzJ}:
\begin{align*}
\co{\nb_{\va} \Ddt \bar{z}_J^i} &\lesssim \co{\nb_{\va}(\nabla_a T^i \theta_\ep z_J^a - T^i w_J \theta_\ep - \tilde{O}_J^{i})} \\
&\lesssim \tisum\co{\nb_{\va_1} \nabla_a T^i \theta_\ep} \co{\nb_{\va_2} z_J^a} + \co{\nb_{\va_1} T^i w_J} \co{\nb_{\va_2} \theta_\ep} + \co{\nb_{\va} \tilde{O}_J^{i}}.
\end{align*}

For the first term, we use the bounds for $\theta_\ep$ from Lemma~\ref{lem:new} and the inductive hypothesis for $z_J$:
\begin{align*}
\co{\nb_{\va_1} \nabla_a T^i \theta_\ep} \co{\nb_{\va_2} z_J^a} &\ls \hn^{(|\va_1|+1-L)_+}\Xi^{|\va_1|+1}e_u^{1/2} \cdot \hn^{(|\va_2| -\unl)_+} \Xi^{|\va_2|} S_z (1 + \hh(t)) \\
&\ls \hn^{(|\va| -\unl)_+} \Xi^{|\va|+1} e_u^{1/2} S_z (1 + \hh(t)).
\end{align*}

For the second term, we use the bounds \eqref{eq:newtVelocBd} for $u_J^i = T^i w_J$:
\begin{align*}
\co{\nb_{\va_1} T^i w_J} \co{\nb_{\va_2} \theta_\ep} &\lesssim \hn^{(|\va_1| -\unl)_+} \Xi^{|\va_1|} S_u (1 + \hh(t)) \cdot \hn^{(|\va_2|-L)_+} \Xi^{|\va_2|}e_u^{1/2}  \\
&\ls \hn^{(|\va| -\unl)_+} \Xi^{|\va|+1}  e_u^{1/2} S_z (1 + \hh(t)) \log \widehat \Xi,
\end{align*}

The $\|\nb_{\va}\ti O_J^i\|_0$ term was already bounded in Proposition~\ref{prop:tildewBds}. We have 
$$\|\nb_{\va}\ti O^i_J\|_0\ls \hn^{(|\va| -\unl)_+} \Xi^{|\va|} \tau^{-1} S_z (1 + \hh(t))$$

The bound on the high frequency projection $P_q$ is proved very similarly to the high frequency bound for $\bar w_J$.  The key point is that taking a $\dot{C}^\a$ norm costs at most $\hxi^\a$. 

\end{proof}

\begin{proof}[Proof of Proposition~\ref{prop:boiledDownBound} for $\bar r_J$]
Recall that $\bar{r}_J$ satisfies the equation \eqref{eq:barrJ}:
\begin{align}
\Ddt \bar r_J^{ij} = \partial_t \bar{r}_J^{ij} + T^\ell \theta_\epsilon \nabla_\ell \bar{r}_J^{ij} &= \RR^{ij}_a \nabla_\ell [ \nabla_b T^\ell \theta_\epsilon r_J^{ab} ] + \BB^{ij}_b[z_J^b, \theta_\epsilon] - \tilde{O}_J^{ij} \label{eq:barrJ_recall}
\end{align}

We need to show that for all $|\va| \leq L'$, we have the bounds
\begin{align}
\co{\nb_{\va} \overline{D}_t \bar{r}_J(s)} &\lesssim \hn^{(|\va| -\underline{L})_+} \Xi^{|\va|} \tau^{-1} S_r (1 + \hh(t)) \label{eq:barrJ_C0_bound}\end{align}

For the $C^0$ bound \eqref{eq:barrJ_C0_bound}, we need to bound
\begin{align*}
\nb_{\va} (\RR^{ij}_a \nabla_\ell [ \nabla_b T^\ell \theta_\epsilon r_J^{ab} ]) + \nb_{\va} (\BB^{ij}_b[z_J^b, \theta_\epsilon]) - \nb_{\va} \tilde{O}_J^{ij}.
\end{align*}

For the first term, we use Proposition \ref{prop:standardDegZero} to obtain the bound 
$$\co{\nb_{\va} (\RR^{ij}_a \nabla_\ell [ \nabla_b T^\ell \theta_\epsilon r_J^{ab} ])}\ls(\log\hxi)\co{\nb_{\va}(\nb u_\ep r_J)} + \hxi^{-\a}\cda{\nb_{\va}(\nb u_\ep r_J)}.$$ Then 
\begin{align*}
\co{\nb_{\va}(\nb u_\ep r_J)}
&\ls \tisum \hn^{(|\va_1|+1-L)_+}\Xi^{|\va_1|+1}e_u^{1/2} \cdot S_r(1+\hh)\hn^{(|\va_2|-\unl)_+}\Xi^{|\va_2|}\\
&\ls \hn^{(|\va|-\unl)_+}\Xi^{|\va|}(\Xi e_u^{1/2}) S_r(1+\hh)
\end{align*}
and thus $(\log\hxi)\co{\nb_{\va}(\nb u_\ep r_J)} \ls \hn^{(|\va|-\unl)_+}\Xi^{|\va|}\tau^{-1} S_r(1+\hh)$. 

Using the product rule for $\dot{C}^\alpha$ norms, we have
\begin{align*}
\cda{\nb_{\va}(\nb u_\ep r_J)} 
&\lesssim \tisum \cda{\nb_{\va_1} \nb u_\ep \nb_{\va_2} r_J} \\
&\lesssim \tisum (\cda{\nb_{\va_1} \nb u_\ep} \co{\nb_{\va_2} r_J} + \co{\nb_{\va_1} \nb u_\ep} \cda{\nb_{\va_2} r_J}).
\end{align*}

For the first term in the sum, we use the interpolation inequality for Hölder norms to get
\begin{align*}
\cda{\nb_{\va_1} \nb u_\ep} 
&\lesssim \co{\nb_{\va_1} \nb u_\ep}^{1-\alpha} \co{\nb \nb_{\va_1} \nb u_\ep}^\alpha \lesssim \hn^{(|\va_1|+2-L)_+} \Xi^{|\va_1|+1} e_u^{1/2} \hxi^\alpha  \\
\co{\nb_{\va_2} r_J} &\lesssim \hn^{(|\va_2|-\unl)_+} \Xi^{|\va_2|} S_r (1+\hh).
\end{align*}

Thus, the first term is bounded by
\begin{align*}
\cda{\nb_{\va_1} \nb u_\ep}  \co{\nb_{\va_2} r_J}  \ls \hn^{(|\va|-\unl)_+}\Xi^{|\va|}(\Xi e_u^{1/2}) \hxi^\alpha S_r(1+\hh)
\end{align*}

For the second term, we have
\begin{align*}
\co{\nb_{\va_1} \nb u_\ep} \lesssim \hn^{(|\va_1|+1-L)_+} \Xi^{|\va_1|+1} e_u^{1/2},
\end{align*}
and by the definition of $\hh$,
\begin{align*}
\cda{\nb_{\va_2} r_J} \lesssim \hn^{(|\va_2|-\unl)_+} \Xi^{|\va_2|} S_r (1+\hh) \hxi^\alpha.
\end{align*}

Thus, the second term is bounded by
\begin{align*}
\tisum \co{\nb_{\va_1} \nb u_\ep} \cda{\nb_{\va_2} r_J}
&\lesssim \tisum \hn^{(|\va_1|+1-L)_+ + (|\va_2|-\unl)_+} \Xi^{|\va|+1} e_u^{1/2} S_r (1+\hh) \hxi^\alpha \\
&\lesssim \hn^{(|\va|-\unl)_+} \Xi^{|\va|+1} e_u^{1/2} S_r (1+\hh) \hxi^\alpha.
\end{align*}

Combining these estimates, we get
\begin{align*}
\cda{\nb_{\va}(\nb u_\ep r_J)} 
&\lesssim \hn^{(|\va|-\unl)_+} \Xi^{|\va|+1} e_u^{1/2} S_r (1+\hh) \hxi^\alpha.
\end{align*}

Therefore, using Proposition~\ref{prop:standardDegZero}, we have
\begin{align*}
\hxi^{-\alpha} \cda{\nb_{\va}(\nb u_\ep r_J)} 
&\lesssim \hn^{(|\va|-\unl)_+} \Xi^{|\va|+1} e_u^{1/2} S_r (1+\hh) \\
&\lesssim \hn^{(|\va|-\unl)_+} \Xi^{|\va|} \tau^{-1} S_r (1+\hh),
\end{align*}
where the last inequality uses $\Xi e_u^{1/2} \lesssim \tau^{-1}$. We conclude
$$\co{\nb_{\va} (\RR^{ij}_a \nabla_\ell [ \nabla_b T^\ell \theta_\epsilon r_J^{ab} ])}
\ls  \hn^{(|\va|-\unl)_+} \Xi^{|\va|} \tau^{-1} S_r (1+\hh).$$

For $q \geq \hq$ we must also bound $P_q$ of this term.  To do so, we recall the estimate on the $\dot{C}^\a$ norm that we just proved to obtain
\ALI{
\co{P_q \nb_{\va} (\RR^{ij}_a \nabla_\ell [ \nabla_b T^\ell \theta_\epsilon r_J^{ab} ])} &\ls 2^{-\a q} \cda{[ \nabla_b T^\ell \theta_\epsilon r_J^{ab} ]} \\
&\ls 2^{- \a q} \hxi^{\a} \hn^{(|\va|-\unl)_+} \Xi^{|\va|} \tau^{-1} S_r (1+\hh),
}
which is our desired bound.

It now remains to estimate the other two terms.  For the forcing term $\tilde{O}^{ij}$ the desired estimates follow directly from  Proposition~\eqref{prop:tildewBds} and the Littlewood Paley characterization of the $\dot{C}^\a$ norm.  A more involved analysis is necessary for the $\BB$ term.

\paragraph{The $\BB$ term}

This section is one of the main novelties of our analysis.  We now define and estimate the $\BB$ term, which is required to satistfy
\ALI{
\nb_j \BB^{j\ell}_a[z_J^a, \th] &= \nb_a T^\ell[z_J^a] \th - z_J^a \nb_a T^\ell[\th]
}
Observe that the vector field on the right hand side can be written as
\ali{
=\sum_q &\left[\nb_a T^\ell[P_{\leq q+1} z_J^a] P_{\leq q+1} \th - P_{\leq q+1}z_J^a \nb_a T^\ell[P_{\leq q+1} \th] \right]\\
&-\left[\nb_a T^\ell[P_{\leq q} z_J^a] P_{\leq q} \th - P_{\leq q}z_J^a \nb_a T^\ell[P_{\leq q} \th] \right]
}
We now decompose the right hand side as a paraproduct
\ALI{
&\nb_a T^\ell[z_J^a] \th - z_J^a \nb_a T^\ell[\th] = LH + HL + HH \\
LH &= \sum_q P_{\leq q-1} \nb_a T^\ell[z_J^a] P_{q+1} \th - P_{\leq q-1} \nb_a T^\ell[\th] P_{q+1} z_J^a \\
HL &= \sum_q P_{q+1} \nb_a T^\ell[z_J^a] P_{\leq q-1} \th -  P_{q+1} \nb_a T^\ell[\th] P_{\leq q-1} z_J^a \\
HH &= \sum_q \nb_a T^\ell[P_{q+1} z_J^a] P_{q+1} \th - \nb_a T^\ell[P_{q+1} \th] P_{q+1} z_J^a \\
&+ \sum_q \nb_a T^\ell[P_{q+1} z_J^a] P_q \th - \nb_a T^\ell[P_{q} \th] P_{q+1} z_J^a \\
&+ \sum_q \nb_a T^\ell[P_{q} z_J^a] P_{q+1} \th - \nb_a T^\ell[P_{q+1} \th] P_{q} z_J^a
}
Note that the $HL$ and $LH$ terms both live at frequency $2^q$.  For these we apply an order $-1$ operator $\calR_a^{j\ell}$ that solves the divergence equation.  For the high-high terms, we invoke the divergence form principle of Section~\ref{sec:divFormPrinciple} (in particular the fact that the multiplier for $\nb T$ is {\bf even} and the fact that a minus sign appears) to write them as the divergence of a bilinear convolution.  Hence, for an appropriate Fourier multiplier $P_{\approx q}$ with symbol equal to $1$ on a neighborhood of $|\xi| = 2^q$,  
\ali{
 \BB^{j\ell}_a[z_J^a, \th] &=  \BB^{j\ell}_{H} +  \BB^{j\ell}_{LH} + \BB^{j\ell}_{HL} \\
\BB^{j\ell}_{LH} &=  
\sum_q P_{\approx q} \calR^{j\ell}_b[ P_{\leq q-1} \nb_a T^b [z_J^a] P_{q+1} \th_\ep - P_{\leq q-1} \nb_a T^b[\th_\ep] P_{q+1} z_J^a ] \\ 
\BB^{j\ell}_{HL} &= \sum_q P_{\approx q} \calR_b^{j\ell} [ P_{q+1} \nb_a T^b[z_J^a] P_{\leq q-1} \th_\ep -  P_{q+1} \nb_a T^b[\th_\ep] P_{\leq q-1} z_J^a]  \\
\nb_j \BB^{j\ell}_{H} &= HH \\
\BB^{j\ell}_H &= \sum_q K_{qa}^{j\ell} \ast [ z_J^a, \th_\ep ] = \sum_q K_{qa}^{j\ell} \ast [ P_{\approx q} z_J^a, P_{\approx q} \th_\ep ] \\
&= \sum_q \int_{\R^2 \times \R^2} z_J^a(x - h_1) \th_\ep(x - h_2) K_{qa}^{j\ell}(h_1,h_2) \,dh_1dh_2 
}
where $K_{qa}^{j\ell}(h_1,h_2)$ is a Schwartz function on $\R^2 \times \R^2$ and
\ali{
K_{qa}^{j\ell}(h_1, h_2) &= 2^{4q} K_{0a}^{j\ell}(2^q h_1, 2^q h_2).
}
We begin by estimating the high-high term.  We decompose into high and low frequencies, observing that the spatial derivatives commute with the bilinear convolution kernel
\ALI{
\underline{\BB}_H^{j\ell} &\equiv \sum_{q = 0}^{\hq-1} K_{qa}^{j\ell}\ast[z_J^a, \th_\ep] \\
\nb_{\va} \underline{\BB}_H^{j\ell} &= \tisum \sum_{q=0}^{\hq-1} K_{qa}^{j\ell}\ast[\nb_{\va_1}z_J^a, \nb_{\va_2} \th_\ep]  \\
\co{\nb_{\va} \underline{\BB}_H^{j\ell}} &\ls \sum_{q=0}^{\hq-1} \tisum \| K_{qa}^{j\ell} \|_{L^1} \co{\nb_{\va_1}z_J^a} \co{\nb_{\va_2} \th_\ep} \\
&\ls \sum_{q=0}^{\hq} \tisum \hn^{(|\va_1| - \unl)_+}  \Xi^{|\va_1|}   S_z ( 1 + \hh) [\hn^{(|\va_2| - \unl)_+} \Xi^{|\va_2|} e_u^{1/2} ] \\
&\ls (\log \hxi) \hn^{(|\va| - \unl)_+} \Xi^{|\va|} (\Xi e_u^{1/2}) S_r(1+ \hh) \qquad \mbox{since } S_z = \Xi S_r
}
For the high frequencies we bound
\ALI{
 \overline{\BB}_H^{j\ell} &\equiv \sum_{q = \hq}^\infty K_{qa}^{j\ell}\ast [P_{\approx q} z_J^a, P_{\approx q} \th] \\
\co{\nb_{\va} \overline{\BB}_H^{j\ell}} &\ls \sum_{q = \hq}^\infty \tisum \| K_{qa}^{j\ell}\|_{L^1} \co{ P_{\approx q} \nb_{\va_1} z_J} \co{\nb_{\va_2}\th_\ep} \\
&\ls \sum_{q=\hq}^\infty \tisum 2^{-\a q} \cda{\nb_{\va_1} z_J} \co{\nb_{\va_2}\th_\ep} \\
&\ls \hxi^{-\a} \hxi^\a (\hn^{(|\va_1| - \unl)_+} \Xi^{|\va_1|} S_z(1+\hh)) (\hn^{(|\va_2| - \unl)_+} \Xi^{|\va_2|} e_u^{1/2})  \\
&\ls (\Xi e_u^{1/2}) \hn^{(|\va| - \unl)_+} \Xi^{|\va|} S_r (1+\hh) \qquad \mbox{since } S_z = \Xi S_r
}
Finally, for $q' > \hq$, we bound $P_{q'} \nb_{\va} \BB_H^{j\ell}$ by observing that, due to frequency truncation, only terms with $q > q' - 2$ can contribute.  That is, from the formula
\ALI{
 \int_{\R^2 \times \R^2} P_{\approx q} z_J^a(x - h_1) P_{\approx q}\th_\ep(x - h_2) K_{qa}^{j\ell}(h_1,h_2)\, dh_1dh_2 
}
we see that the biconvolution only translates each factor in physical space and therefore modulates in frequency space.  The integral above will still be localized to frequencies below $2^{q+2}$ since the Fourier transform maps products to convolutions.  Therefore, we are able to bound

\ALI{
P_{q'} \nb_{\va} \BB_H^{j\ell} &= \sum_{q = q' - 3}^\infty P_{q'}  \tisum K_{qa}^{j\ell}\ast[ P_{\approx q} \nb_{\va_1} z_J^a, P_{\approx q} \nb_{\va_2} \th_\ep ]  \\
\co{P_{q'} \nb_{\va} \BB_H^{j\ell}} &\ls \sum_{q'-3}^\infty \tisum \co{P_{\approx q} \nb_{\va_1} z_J} \co{ \nb_{\va_2} \th_\ep} \\
&\ls \sum_{q'-3}^\infty 2^{-\a q} \cda{\nb_{\va_1} z_J} \co{ \nb_{\va_2} \th_\ep} \\
&\ls 2^{-\a q'} \hxi^\a \hn^{(|\va| - \unl)_+} \Xi^{|\va|} S_z (1+\hh) e_u^{1/2}  \\
&\ls 2^{-\a q'} \hxi^\a \hn^{(|\va| - \unl)_+} \Xi^{|\va|} (\Xi e_u^{1/2}) S_r (1+\hh) \qquad \mbox{since } S_z = \Xi S_r.
}

Proof for the $\calB$ Term, Part 2 (High-Low terms):  The philosophy of the High-Low term estimates is similar to the gluing in \cite{isettOnsag}.  Namely, for low frequencies we substitute $z_J^j = \nb_i r_J^{ij}$ while for high frequencies we use our bounds on $w_J$ and $z_J$.

Specifically, we split the High-Low term into high frequency and low frequency components
\begin{align}
\BB^{j\ell}_{HL} &= \BB^{j\ell}_{HLH} + \BB^{j\ell}_{HLL} \notag \\
\BB^{j\ell}_{HLH} &= \sum_{q> \bar{q}} P_{\approx q} \calR_b^{j\ell} [ P_{q+1} \nb_a T^b[z_J^a] P_{\leq q-1} \th_\ep -  P_{q+1} \nb_a T^b[\th_\ep] P_{\leq q-1} z_J^a] \label{eq:HLHterm}\\
\BB^{j\ell}_{HLL} &= \sum_{q=0}^{\bar{q}} P_{\approx q} \calR_b^{j\ell} [ P_{q+1} \nb_a T^b[z_J^a] P_{\leq q-1} \th_\ep -  P_{q+1} \nb_a T^b[\th_\ep] P_{\leq q-1} z_J^a] \label{eq:HLLterm}
\end{align}
Here $\bar{q}$ is chosen so that $2^{\bar{q}} \sim \Xi$.

We start with the $HLH$ term.  
\ali{
\co{\nb_{\va} \eqref{eq:HLHterm}} &\lesssim \sum_{q > \bar{q}} 2^{-q} \tisum \co{P_{q+1} T^b[\nb_{\va_1} \nb_a z_J^a]} \co{\nb_{\va_2} P_{\leq q - 1} \th_\ep } \label{eq:HLHtermStep1} \\
&+ \sum_{q > \bar{q}} 2^{-q} \tisum \co{ P_{q+1} T^b \nb_{\va_1} \nb_a \th_\ep } \co{P_{\leq q-1} \nb_{\va_2} z_J^a} \label{eq:HLHtermStep2}
}
We substitute $\nb_a z_J^a = w_J$ and use the bounds $\| P_{\approx q} \RR_b^{j\ell} \| \lesssim 2^{-q}$, $ \| P_{\leq q-1} \| \lesssim 1, \| P_{q+1} T \| \lesssim 1$ as operators on $C^0$ to obtain
\ali{
\eqref{eq:HLHtermStep1} &= \sum_{q > \bar{q}} 2^{-q} \tisum \co{P_{q+1} T^b[\nb_{\va_1} w_J]} \co{\nb_{\va_2} P_{\leq q - 1} \th_\ep } \\
&\lesssim 2^{-\bar{q}} \tisum \hn^{(|\va_1| - \unl)_+} \Xi^{|\va_1|} S_w (1+\hh) \cdot \hn^{(|\va_2| - \unl)_+} \Xi^{|\va_2|} e_u^{1/2} \\
&\leq \hn^{(|\va| - \unl)_+} \Xi^{|\va|} (\Xi e_u^{1/2}) S_r (1+\hh) \qquad \mbox{since } S_w = \Xi^2 S_r 
}
The other term is bounded by the same quantity
\ali{
\eqref{eq:HLHtermStep2} &\lesssim \sum_{q > \bar{q}} 2^{-q} \tisum \hn^{(|\va_1| - (L-1) )_+} \Xi^{|\va_1|} e_u^{1/2} \cdot \hn^{(|\va_2| - \unl)_+}\Xi^{|\va_2|} S_z (1+\hh) \\
&\lesssim \hn^{(|\va| - \unl)_+} \Xi^{|\va|} (\Xi e_u^{1/2}) S_r (1+\hh) \qquad \mbox{since } S_z = \Xi S_r
}
We now treat the low-frequency term by recalling $z_J^a = \nb_c r_J^{ca}$ and using the bounds $\| P_q \nb^L T \| \lesssim 2^{Lq}$, $\| P_{\leq q-1} \nb \| \lesssim 2^q $ as operators on $C^0$:
\ALI{
\eqref{eq:HLLterm} &= \sum_{q=0}^{\bar{q}} P_{\approx q} \calR_b^{j\ell} [ P_{q+1} \nb_a T^b[\nb_c r_J^{ca}] P_{\leq q-1} \th_\ep -  P_{q+1} \nb_a T^b[\th_\ep] P_{\leq q-1} \nb_c r_J^{ca}] \\
\co{\nb_{\va} \eqref{eq:HLLterm}} &\leq \sum_q 2^{-q} \tisum \| P_{q+1} \nb_a T^b \nb_c \| \co{\nb_{\va_1} r_J} \co{\nb_{\va_2} \th_\ep} \\
&+ \sum_q 2^{-q} \tisum \| P_{q+1} \nb_a T^b \| \|P_{\leq q-1} \nb_c \| \co{\nb_{\va_2} \th_\ep} \co{\nb_{\va_1} r_J} \\
&\lesssim \sum_q 2^q \tisum \co{\nb_{\va_1} r_J} \co{\nb_{\va_2} \th_\ep} \\
&\lesssim \Xi \cdot \hn^{(|\va| - \unl)_+} \Xi^{|\va|} S_r(1+\hh) e_u^{1/2} =  \hn^{(|\va| - \unl)_+} \Xi^{|\va|} (\Xi e_u^{1/2}) S_r (1+\hh) 
}
This bound is what is desired for the low frequency contributions.

For all $q \geq \hq$, we need to also bound the Littlewood Paley projections of $\BB_{HL}^{j\ell}$.  By frequency support considerations, note that
\ali{
P_q \nb_{\va} \BB^{j\ell}_{HL} 
&= \sum_{q' = q-2}^{q+2}  P_{\approx q'} \calR_b^{j\ell} P_q \nb_{\va} [ P_{q'+1} T^b[w_J] P_{\leq q'-1} \th_\ep -  P_{q'+1} \nb_a T^b[\th_\ep] P_{\leq q'-1} z_J^a] \label{eq:PqBHLterm}
}
Then since $\| P_{q'} \RR \| \lesssim 2^{-q} \lesssim \Xi^{-1}$ we have
\ali{
\eqref{eq:PqBHLterm} &\lesssim \Xi^{-1} 2^{-\a q} \cda{\nb_{\va} [ P_{q'+1} T^b[w_J] P_{\leq q'-1} \th_\ep -  P_{q'+1} \nb_a T^b[\th_\ep] P_{\leq q'-1} z_J^a]} 
}
Note that the estimate for the $HLH$ term bounded the $C^0$ norm of the right hand side rather than the $\dot{C}^\a$ seminorm.  It is clear from the Littlewood-Paley characterization of the seminorm that $\cda{T f} \lesssim \cda{f}$.  Using this information, we see that taking the $\dot{C}^\a$ seminorm costs at most $\hxi^{\a}$ compared to the estimate for the $C^0$ norm, which is the bound we require.

We are left with the $LH$ term, which we decompose as
\ali{
\BB^{j\ell}_{LH} &= \BB_{LHL}^{j\ell} + \BB_{LHH}^{j\ell} \\
\BB_{LHL}^{j\ell}&= \sum_{q=0}^{\bar{q}} P_{\approx q} \calR^{j\ell}_b[ P_{\leq q-1} \nb_a T^b [z_J^a] P_{q+1} \th_\ep - P_{\leq q-1} \nb_a T^b[\th_\ep] P_{q+1} z_J^a ] \\
\BB_{LHH}^{j\ell}&= \sum_{q > \bar{q}} P_{\approx q} \calR^{j\ell}_b[ P_{\leq q-1} \nb_a T^b [z_J^a] P_{q+1} \th_\ep - P_{\leq q-1} \nb_a T^b[\th_\ep] P_{q+1} z_J^a ]
}
For the $LHL$ term, we substitute $z_J^a = \nb_c r_J^{ca}$ and use $\| P_{\leq q} \nb^L T \| \lesssim 2^{Lq}$ for $L= 1, 2$ (which follows by dyadic decomposition) as well as $\| P_{\approx q} \calR \| \lesssim 2^{-q}$, to obtain
\ALI{
\co{\nb_{\va} \BB_{LHL}^{j\ell}} &\lesssim  \sum_{q=0}^{\bar{q}} 2^{-q} \tisum \left( \| P_{\leq q-1} \nb_a \nb_c T^b \| + \| P_{\leq q-1} \nb_a T^b \| \|P_{q+1} \nb_c \| \right) \| \co{\nb_{\va_1} r_J^{ca}} \co{\nb_{\va_2} \th_\ep} \\
&\lesssim \sum_{q=0}^{\bar{q}} 2^q \tisum \co{\nb_{\va_1} r_J^{ca}} \co{\nb_{\va_2} \th_\ep} \\
&\lesssim \Xi \hn^{(|\va| - \unl)_+} \Xi^{|\va|} S_r (1+\hh) e_u^{1/2} = \hn^{(|\va| - \unl)_+} \Xi^{|\va|} (\Xi e_u^{1/2}) S_r (1+\hh)
}
For the $LHH$ term, we substitute $T^b \nb_a z_J^a = T^b w_J = u_J^b$ and $T^b \th_\ep = u_\ep^b$, and use \eqref{eq:velocBounded1} to obtain
\ali{
\nb_{\va} \BB_{LHH}^{j\ell} &= \sum_{q > \bar{q}} \tisum P_{\approx q} \calR^{j\ell}_b[ P_{\leq q-1} \nb_{\va_1} u_J^b P_{q+1} \nb_{\va_2}\th_\ep - P_{\leq q-1} \nb_{\va_2} \nb_a u_\ep^b] P_{q+1} \nb_{\va_1} z_J^a ] \label{eq:readyForLHH} \\
\co{\eqref{eq:readyForLHH} } &\lesssim \sum_{q > \bar{q}} 2^{-q} \tisum \left( [\hn^{(|\va_1| - \unl)_+} \Xi^{|\va|} S_u(1+\hh)][\hn^{(|\va_2| - \unl)_+}\Xi^{|\va_2|} e_u^{1/2}] \right.\notag \\
&\left.+ [\hn^{(|\va_2| + 1 - L)_+} \Xi^{|\va_2| + 1} e_u^{1/2}][\hn^{(|\va_1| - \unl)_+} \Xi^{|\va_1|} S_z (1+\hh)] \right) \notag \\
\co{\eqref{eq:readyForLHH} } &\lesssim \sum_{q > \bar{q}} 2^{-q} \plhxi \hn^{(|\va| - \unl)_+} \Xi^{|\va|} (\Xi e_u^{1/2}) S_z (1+\hh) \qquad \mbox{since } S_u = \plhxi \Xi S_z \notag \\
&\lesssim \hn^{(|\va| - \unl)_+} \Xi^{|\va|} \plhxi (\Xi e_u^{1/2}) S_r (1+\hh)  \qquad \mbox{since } S_z = \Xi S_r  \mbox{ and } 2^{-\bar{q}} \sim \Xi^{-1} \notag
}
Finally we need to bound $P_q \nb_{\va} \BB_{LHH}^{j\ell}$.  Applying $P_q$ to \eqref{eq:readyForLHH}, we have
\ALI{
&\co{ P_q \nb_{\va} \BB_{LHH}^{j\ell}} = \\
&= \sum_{q' = q-2}^{q+2} \tisum \| P_{\approx q'} \RR_b^{j\ell} \| \cdot 2^{-\a q} \cda{ P_{\leq q'-1} \nb_{\va_1} u_J^b \, P_{q'+1} \nb_{\va_2}\th_\ep - P_{\leq q'-1}[ \nb_{\va_2} \nb_a u_\ep^b] \, P_{q'+1} \nb_{\va_1} z_J^a } \\
&\lesssim \Xi^{-1} \cdot 2^{-\a q}\sum_{q'=q-2}^{q+2}  \cda{ P_{\leq q'-1} \nb_{\va_1} u_J^b \, P_{q'+1} \nb_{\va_2}\th_\ep - P_{\leq q'-1} [ \nb_{\va_2} \nb_a u_\ep^b] \, P_{q'+1} \nb_{\va_1} z_J^a }
}
This term can now be bounded using the method to estimate $\co{\nb_{\va} \BB_{LHH}^{j\ell}}$, with the extra input that a $\dot{C}^\a$ seminorm costs an extra factor of $\hxi^{\a}$ compared to the $C^0$ norm.  With this estimate we conclude the bound for $\bar{r}_J$. \end{proof}

Now that we have estimated all of ${\bf F}_J = \{ \bar{w}_J, \bar{z}_J, \bar{r}_J \}$, Proposition~\ref{prop:weightedNorm} guarantees that \[\hh(t) \ls 1\] is bounded on $|t - \tau k| \leq \tau$.  We may now use this bound in the estimates that follow.  

\begin{proof}[Proof of \eqref{eq:newtVelocBd} and \eqref{eq:newtVelocBd2} for $u_J$]
Lemma~\ref{lem:zeroOrder} applied to $\nb_{\va} w_J$ and \eqref{eq:NlowerBd} give
\ali{
\co{\nb_{\va} T^\ell w_J} &\lesssim \plhxi \co{\nb_{\va} w_J} + \hxi^{-\a} \cda{\nb_{\va} w_J} \label{eq:uJbegins} \\
&\lesssim \plhxi \hn^{(|\va| - \unl)_+} \Xi^{|\va|} S_w \notag \\
&\lesssim \plhxi \hn^{(|\va| - \unl)_+} \Xi^{|\va|} \Xi^{1/2} D_R^{1/2} N^{-1/2} \ll \plhxi^{-1} \hn^{(|\va| - \unl)_+} \Xi^{|\va|} e_u^{1/2} \notag
}

The bound \eqref{eq:newtVelocBd} follows from the estimates for $w_J$ by Lemma~\ref{lem:zeroOrder} 

We now consider $\Ddt u_J$ and its bound \eqref{eq:newtVelocBd2}.  By the commutator estimate of Lemma~\ref{lem:zeroOrder}, it suffices to estimate $\co{T^\ell \nb_{\va} \Ddt w_J}$ once we check that
\ALI{
\plhxi N^{(|\va| - L)_+} \Xi^{|\va|} (\Xi e_u^{1/2}) S_w &\leq N^{(|\va| + 1 - L)_+} \Xi^{|\va|} (\Xi e_u^{1/2}) e_u^{1/2}.
}
This inequality follows from the lower bound \eqref{eq:NlowerBd} and $S_w = \Xi^2 \mu^{-1} D_{R,n} \leq \Xi^{1/2} D_R^{1/2} N^{-1/2}$.

It now suffices to obtain the same upper bound for
\ali{
\co{ T^i \nb_{\va} \Ddt w_J} &\leq \co{ T^i[ \nb_{\va}[T^\ell w_J \nb_\ell \th_\ep] } \label{eq:firstAdvecwJ} \\
&+ \co{T^i \nb_j \nb_\ell \nb_{\va}  A_{(f,J)}^{j\ell}} \label{eq:secondAdvecwJ}
}
Term \eqref{eq:firstAdvecwJ} is bounded using the product rule, the lower bound \eqref{eq:NlowerBd} and \eqref{eq:orderZeroBound} (the fact that applying $T$ costs a log) by
\ALI{
\eqref{eq:firstAdvecwJ} &\lesssim \tisum_{\va_1, \va_2} \plhxi [ \plhxi \hn^{(|\va_1| - \unl)_+} \Xi^{|\va_1|}  S_w ] \cdot [\hn^{(|\va_2| - L)_+} \Xi^{|\va_2| + 1} e_u^{1/2}] \\
&\lesssim \hn^{(|\va| - \unl)_+} \Xi^{|\va|}(\Xi e_u^{1/2}) \plhxi^2 (\Xi^{1/2} D_R^{1/2} N^{-1/2}) \\
&\lesssim \hn^{(|\va| + \unl)_+} \Xi^{|\va|} (\Xi e_u^{1/2}) e_u^{1/2}
}
For the term \eqref{eq:secondAdvecwJ} we decompose into high and low frequencies.  Choosing an integer $\bar{q}$ such that $2^{\bar{q}} \sim \Xi$, we separate
\ali{
\eqref{eq:secondAdvecwJ} &\leq \sum_{q=0}^{\bar{q}} \co{T^i \nb_j \nb_\ell P_q \nb_{\va} A_{(f,J)} }  \label{eq:lowFreqTransuJ}\\
&+ \sum_{q > \bar{q}} \co{T^i P_q \nb_{\va} \nb_j \nb_\ell A_{(f,J)}^{j\ell} } \label{eq:highFreqTransuJ}
}
The lower frequency term we estimate using the bound $\| T^i \nb_j \nb_\ell P_q \| \lesssim 2^{2q}$ on the operator norm to obtain
\ALI{
\eqref{eq:lowFreqTransuJ} &\lesssim \Xi^2 \hn^{(|\va| - \unl)_+} \Xi^{|\va|} D_{R,n} \lesssim \hn^{(|\va| - \unl)_+} \Xi^{|\va|} (\Xi e_u^{1/2}) e_u^{1/2}
}
The high frequency term is a bit more subtle.  Recalling $e_u^{1/2} = \Xi^{1/2} D_u^{1/2} \geq \Xi^{1/2} D_R^{1/2}$, we have by \eqref{eq:AfJbound} that 
\ALI{
\eqref{eq:highFreqTransuJ} &\lesssim  \sum_{q > \bar{q}} 2^{-\a q} \cda{\nb_{\va} \nb_j\nb_\ell A_{(f,J)}^{j\ell} } \\
&\lesssim \sum_{q > \bar{q}} 2^{-\a q} \hn^{(|\va| + 2 + \a - (L-1))_+} \Xi^{|\va| + 2 + \a} D_{R,n} \\
&\lesssim \hn^{(|\va| + 1 - \unl)_+} \Xi^{|\va|} (\Xi e_u^{1/2}) e_u^{1/2} 
}
Note that factors of $\hn$ do not appear until one takes more than $L - 1$ derivatives of $A_{(f,J)}$, thus explaining the power of $\hn$ above.

\end{proof}

\begin{proof}[Proof of \eqref{eq:newtVelocBd} and \eqref{eq:newtVelocBd2} for $u_n$]
Recall that the change in velocity due to one step of the Newton iteration is $u_n^\ell = T^\ell w_n = \sum_k \tilde{\chi}_k(t) u_{(k,n)}^\ell$.  Inequalities \eqref{eq:newtVelocBd} and \eqref{eq:newtVelocBd2} for $u_n^\ell$ therefore follow directly follow from the same inequalities for the $u_J^\ell = u_{(k,n)}^\ell$ once we bound the term where the advective derivative hits $\tilde{\chi}_k(t)$.  This term has size
\ALI{
\co{\tilde{\chi}_k'} \co{\nb_{\va} u_{(k,n)}^\ell} &\lesssim \tau^{-1} \hn^{(|\va| - \unl)_+} \Xi^{|\va|} S_u \\
&\lesssim \hn^{(|\va| - \unl)_+} \plhxi (\Xi e_u^{1/2}) \Xi^{|\va|}  \plhxi \Xi^{1/2} D_R^{1/2} N^{-1/2} \\
&\lesssim \hn^{(|\va| - \unl)_+} \Xi^{|\va|} (\Xi e_u^{1/2}) e_u^{1/2}
}
as desired, using the lower bound \eqref{eq:NlowerBd} on $N$.
\end{proof}
\begin{proof}[Proof of \eqref{eq:newtTrfreeBd} for $\rho_J$]
Recall the following bounds on $r_J$ from Proposition~\ref{prop:newtonBound}:
\ali{
\co{ \nb_{\va} r_J } &\lesssim_{\va} \widehat N^{(|\va| -\unl)_+} \Xi^{|\va|} S_r, \label{eq:rhoJbegins}}
For $\rho_J^{j\ell} = \RR^{j\ell}\nb_a \nb_b r_J^{ab}$, we thus obtain
\ALI{
\co{\nb_{\va} \rho_J}
&\ls (\log \hxi)\|\nb_{\va}r_J\|_0+\hxi^{-\a}\|\nb_{\va}r_J\|_{\dot C^\a}\\
&\ls (\log \hxi)\hn^{(|\va|-\unl)_+} \Xi^{|\va|}S_r+\hxi^{-\a} \hxi^{\a} \hh S_r \hn^{(|\va|-\unl)_+} \Xi^{|\va|}\\
&\ls \hn^{(|\va|-\unl)_+}\Xi^{|\va|}S_\rho (1+\hh)  \ls \hn^{(|\va|-\unl)_+}\Xi^{|\va|}S_\rho.
}
To bound the advective derivative of $\rho_J$ with a cost of $\tau^{-1}$ rather than $\mu$ we must examine the evolution equation for $r_J$.  The crucial point is that the forcing term $A_J^{j\ell}$ vanishes on the support of $\tilde{\chi}'$.  

Lemma~\ref{lem:zeroOrder} gives
\ALI{
\nb_{\va} \Ddt \rho_J^{j\ell} &= \RR^{j\ell} \nb_a \nb_b \Ddt r_J^{ab} + (\log \hxi) O(\hn^{(|\va| - \unl)_+} \Xi^{|\va|+1} e_u^{1/2} S_r) 
}
The term in the $O(\cdot)$ is acceptable since $\tau^{-1} \sim \log \hxi \Xi e_u^{1/2}$ and $S_\rho = (\log \hxi ) S_r$.  For the first term, let us define the order zero operators $\QQ^{j\ell}_{ab} = \RR^{j\ell} \nb_a \nb_b$ and also $\tilde \QQ^{j\ell}_{cd} = \QQ^{j\ell}_{ab} \RR^{ab}_c \nb_d$.  We return to the equation for $r_J$ to obtain, for $t \in \supp \partial_t\chi_k$,
\ali{
\nb_{\va} \QQ^{j\ell}_{ab} \Ddt r_J^{ab} &= \nb_{\va} \QQ^{j\ell}_{ab} \RR^{ab}_c \nb_d[ \nb_e T^d \th_\ep  r_J^{c e} ] + \nb_{\va}\QQ^{j\ell}_{ab} \BB^{ab}_c[ z_J^c, \th_\ep] \\
&= \tilde \QQ^{j\ell}_{cd} \nb_{\va} [\nb_e T^d \th_\ep  r_J^{c e} ] + \QQ^{j\ell}_{ab}  \nb_{\va}\BB^{ab}_c[ z_J^c, \th_\ep] \label{eq:rhoJEqn2}
}
Note that the latter equation has exactly the same form as the equation \eqref{eq:barrJ_recall} for $\bar{r}_J$ except for the additional zeroth order operator $\QQ_{ab}^{j\ell}$ appearing in front of $\BB^{ab}_c$.  Thus we can repeat the analysis that was done for $\bar{r}_J$ and use the inequality
\ALI{
\co{\QQ^{j\ell}_{ab}  \nb_{\va}\BB^{ab}_c[ z_J^c, \th_\ep]} &\ls \log \hxi \co{ \nb_{\va}\BB^{ab}_c[ z_J^c, \th_\ep]} + \hxi^{-\a} \sup_{q \geq \hq} 2^{\a q} \co{ P_q \nb_{\va}\BB^{ab}_c[ z_J^c, \th_\ep] }.
}
Doing so we conclude that the estimate for $\nb_{\va} \Ddt \rho_J$ on $\supp \tilde{\chi}'_k$ is the same as the estimate for $\nb_{\va} \Ddt \bar{r}_J$, but with a loss of one power of $\log \hxi$ that comes from the presence of the additional zeroth order operator in front of $\calB_c^{ab}$ in \eqref{eq:rhoJEqn2}.  We omit the remaining details.
\end{proof}

Let us now conclude the proof of Proposition~\ref{prop:newtonBound}. 

\begin{proof}[Proof of \eqref{eq:sizeOfNewtCorrect}]
Let $0 \leq |\va| \leq 1$.  Then
\ALI{
\co{ \nb_{\va}|\nb|^{-1/2} w_J} &\ls \sum_{q \leq \bar{q} }\co{ |\nb|^{-1/2}  P_q \nb_i \nb_{\va} z_J^i} + \sum_{q \geq \bar{q}} \co{|\nb|^{-1/2} P_q \nb_{\va} w_J} \\
&\ls \sum_{q \leq \bar{q} } 2^{q/2} \co{\nb_{\va} z_J} + \sum_{q \geq \bar{q}} 2^{-q/2} \co{\nb_{\va} w_J} \\
&\ls \Xi^{1/2} \Xi^{|\va|} [\Xi^{-1} S_w] + \Xi^{-1/2} \Xi^{|\va|} S_w \\
&\sim \Xi^{-1/2} \Xi^{|\va|} S_w \ls \Xi^{|\va|} D_{R,n}^{1/2}.
}
The desired bound for $w_{n} = \sum_{k,n} \tilde{\chi} w_{k,n}$ now follows.
\end{proof}

\begin{proof}[Proof of \eqref{eq:RnC0}, \eqref{eq:Rnplus1}, and \eqref{eq:Rnplus1Ddt}]

We first prove \eqref{eq:RnC0}. We have
\ali{
R_{(n+1)} = \sum_k \tilde{\chi}_k'(t) \rho_{(k,n)} \label{eq:gluingErrorDef}
}
where \(\rho\) is a trace-free double anti-divergence of \(w_{(k,n)}\). From $\mu = N^{1/2} \Xi^{3/2} D_{R}^{1/2}$, we have
\ALI{
\|R_{(n+1)}\|_0 &\le 
\frac{C}{\tau} \max_k \|\rho_{(k,n)}\|_0 \leq \frac{C}{\tau} S_\rho = \frac{C}{\tau} (\log \hxi) \mu^{-1} D_{R,n}  \\ 
\|R_{(n+1)}\|_0&\leq C (\log \hxi)^2 (D_u/D_R)^{1/2}  N^{-1/2} D_{R,n} \\
\co{R_{(n+1)}} &\leq (\log \hxi)^{-1} D_{R,n+1}
}
where in the last line we use the lower bound on $N$ given by \eqref{eq:NlowerBd}, $(\log \hxi)^3 \lesssim_\eta N^{ \eta} \Xi^{\eta}$ and we take $\hc$ sufficiently large.  This proves \eqref{eq:RnC0} with a log of room to spare.  

The proof of \eqref{eq:Rnplus1} is a simpler version of the proof of \eqref{eq:Rnplus1Ddt}, so we prove only the latter.  Notice that $(|\va| + 1 - \unl)_+ \leq (|\va| + 1 - L)_+ + 3$ since $\unl = L - 3$.  We then obtain
\ALI{
\co{\nb_{\va} \Ddt R_{(n+1)}} &\leq \sup_k \co{\tilde{\chi}_k''(t)} \co{ \nb_{\va} \rho_{(k,n)}} + \sup_k \co{\tilde{\chi}_k'(t)} \co{ \nb_{\va} \Ddt \rho_{(k,n)}} \\
&\lesssim \tau^{-2} [\hn^{(|\va| - \unl)_+} \Xi^{|\va|} \plhxi \mu^{-1} D_{R,n}]  + \tau^{-1} [\hn^{(|\va| + 1 - \unl)_+} \Xi^{|\va|} \tau^{-1} \plhxi \mu^{-1} D_{R,n}] \\
&\lesssim \hn^{(|\va| + 1 - L)_+} \Xi^{|\va|} [\hn^{3} \tau^{-2} \plhxi \mu^{-1} ] D_{R,n} \\
&\lesssim \hn^{(|\va| + 1 - L)_+} \Xi^{|\va|} (\Xi e_u^{1/2}) \fr{D_{R,n}}{N^\eta \Xi^\eta} = \hn^{(|\va| + 1 - L)_+} \Xi^{|\va|} (\Xi e_u^{1/2}) D_{R,n+1}
}
where in the last line we recall $\mu^{-1} = N^{-1/2} \Xi^{-3/2} D_R^{-1/2}$, $\tau^{-1} = b^{-1}\plhxi \Xi e_u^{1/2} = b^{-1}\plhxi \Xi^{3/2} D_u^{1/2}$, $\hn^3 = N^{3/L}$ and we use the lower bound \eqref{eq:NlowerBd} on $N$ and the estimate $(\log \hxi)^3 \lesssim_\eta N^{ \eta} \Xi^{\eta}$.
\end{proof}

\subsection{Errors after the Newton step}

Upon completing the Newton step we have new errors described by \eqref{eq:gluingError2}-\eqref{eq:notDivForm20}.  The error term in \eqref{eq:gluingError2} has already been estimated.  Let us now estimate the terms in \eqref{eq:notDivForm10}-\eqref{eq:notDivForm20}.  

We introduce the notation $R_{M,n}$ and $R_{Q,n}$ to denote special solutions to 
\ali{
\nb_j\nb_\ell R_{M,n}^{j\ell} &= T^\ell(\th - \th_\ep) \nb_\ell w_n + T^\ell w_n \nb_\ell(\th - \th_\ep) \\
\nb_j \nb_\ell R_{Q,n}^{j\ell} &= T^\ell w_n \nb_\ell w_n + \sum_{j = 0}^{n-1} (T^\ell w_n \nb_\ell w_j + T^\ell w_j \nb_\ell w_n)
}
\begin{lem}  For $0 \leq |\va| + r \leq L$ and $0\leq r\le 1$
\ali{
    \co{\nb_{\va} \pr_t^r ( \th - \th_\ep) } + \co{\nb_{\va} \pr_t^r ( u - u_\ep ) } &\lesssim (N \Xi)^{|\va|} (\ost{\Xi} \ost{e}_u^{1/2})^r N^{-1} e_u^{1/2}, \label{eq:mollThUerror}
}\label{lem:new}
Furthermore, for all $0 \leq r \leq 1$ and all $\va$ one has
\ali{
\co{\nb_{\va} \pr_t^r w_n} + \Xi \co{\nb_{\va} \pr_t^r z_n} &\ls (N \Xi)^{|\va|} \mu^r S_w. \label{eq:lossyBdwn}
}

\end{lem}

\begin{proof}
Recall that $\th_\ep = P_{\leq q_\ep} \th$ and $u_\ep^\ell = T^\ell \th_\ep$, where $q_\ep$ is chosen such that $2^{q_\ep} \sim \ep^{-1} = N^{1/L} \Xi$. We begin by estimating the difference $\th - \th_\ep$:
\begin{align*}
\th - \th_\ep &= \th - P_{\leq q_\ep} \th = P_{> q_\ep} \th.
\end{align*}

When we bound $\theta - \theta_\ep$ in $C^0$, we need to be very precise and use the fact that the moments
$$\int h^{\vec{a}} \eta_\ep(h) \, dh = 0$$
all vanish for $0 < |a| \leq L$. This implies
$$\|\theta - \theta_\epsilon\|_0 \lesssim \epsilon^L \|\nabla^L \theta\|_0 \lesssim \fr{e_u^{1/2}}{N}.$$
Now we move on to the $|a| \geq 1$ case.

Now we consider $1 \leq |\va| + r \leq L$.  We use a trivial bound of
\ali{
\| \nb_{\va} \pr_t^r (\th - \th_\ep ) \|_0 &\leq \co{\nb_{\va} \pr_t^r \th } + \co{\nb_{\va} \pr_t^r \th_\ep } \\
&\lesssim \Xi^{|\va|} (\Xi e_u^{1/2})^r e_u^{1/2} = \Xi^{|\va| + 3r/2} D_u^{(r+1)/2}
}
Our goal is to bound this expression by
\ali{
(N \Xi)^{|\va|} (\ost{\Xi} \ost{e}_u^{1/2})^r D_u^{1/2}/N &= (N \Xi)^{|\va|} (N \Xi)^{3r/2} D_R^{r/2} D_u^{1/2} / N
}
Thus we must check that for $1 \leq |\va| + r \leq L$ we have
\ali{
(D_u / D_R )^{r/2} &\ls N^{3 r / 2 + |\va| - 1} = N^{r/2 + ( r + |\va| - 1)}
}
This lower bound follows from \eqref{eq:NlowerBd}, which implies $N \geq D_u / D_R$.  The same proof applies to $u$ since we have assumed the same bounds on $u$ as for $\th$.

To prove \eqref{eq:lossyBdwn}, recall that $w_n = \sum_k \tilde{\chi}_k w_{k,n}$, with $w_{k,n} = \tilde{w}_{k,n} + \bar{w}_{k,n}$, each of size bounded by $S_w$.  In order to estimate $\nb_{\va} \pr_t^r w_{k,n}$, it suffices to observe that: 
\begin{itemize}
    \item Taking a spatial derivative never costs more than $\hn \Xi \leq N \Xi$.
    \item Taking an advective derivative $\Ddt$ of $\bar{w}_{k,n}$ costs at most $\tau^{-1} \ll \mu$.
    \item Taking a pure time derivative $\pr_t = \Ddt - u_\ep \cdot \nb$ of either $\tilde{w}_{k,n}$ or $\bar{w}_{k,n}$ costs at most $\mu$.
\end{itemize}
Similar considerations hold for $z_{k,n}$, which has size $S_z = \Xi^{-1} S_w$.
\end{proof}

\begin{prop} \label{prop:RmRqnewt} For appropriately chosen $R_{M,n}$ and $R_{Q,n}$, we have the estimates 
\ali{
\|\nabla_{\vec a} R_{M,n}\|_0 &\lesssim (N \Xi)^{|\va|} N^{-1} D_R  \\
\co{ \nb_{\va} R_{Q,n} } &\lesssim (N\Xi)^{|\va|}  S_w^2 \Xi^{-1} \lesssim (N\Xi)^{|\va|} N^{-1} D_R 
}
for $0 \leq |\va| \leq L$ and 
\ali{
\co{\nb_{\va} \pr_t R_{M,n} } &\lesssim (N \Xi)^{|\va|} \ost{\tau}^{-1}  N^{-1} D_R \\
\co{\nb_{\va} \pr_t R_{Q,n} } &\lesssim(N\Xi)^{|\va|} \ost{\tau}^{-1} S_w^2 \Xi^{-1} \lesssim (N\Xi)^{|\va|} \ost{\tau}^{-1} N^{-1} D_R   \label{eq:withoutLogLoss}
}
for $0 \leq |\va| \leq L-1$.  Here $\ost{\tau}^{-1} = (N \Xi)^{3/2} D_R^{1/2}$, and $S_w$ is as documented in the table \eqref{eq:sizeOfTerms}, namely $S_w = \Xi^2 \mu^{-1} D_{R,n}$, $\mu = N^{1/2} \Xi^{3/2} D_R^{1/2}$.
\end{prop}

\subsection*{The quadratic terms $R_{Q,n}$.}
We begin by estimating an inverse double divergence of $T^\ell w_n \nabla_\ell w_n = \nabla_\ell(w_n T^\ell w_n)$. It suffices to only estimate this term since the other terms in the equation for $R_{Q,n}$ are similar.
	 
We must estimate a solution to
\begin{align}
\begin{split}
	\nabla_j R^{j\ell}_{Q,n} &= T^\ell w_n w_n\\
	&= \sum_q P_{\le q-1} w_J T^\ell P_{q+1} w_J \\
	&\quad + P_{q+1} w_J P_{\le q-1} T^\ell w_J \\
	&\quad + P_{q+1} w_J T^\ell P_q w_J + P_q w_J T^\ell P_{q+1} w_J + P_{q+1} w_J T^\ell P_{q+1} w_J.
 \end{split}\label{eq:Qn}
\end{align}
Specifically, we achieve bounds for  
$$\nb_{\va}\pr_t^rR_{Q,n}^\ell = \nb_{\va}\pr_t^r \operatorname{div}^{-1}(w_n T^\ell w_n).$$ We decompose this as $LH + HL + HH$ in the manner of \eqref{eq:Qn}.

For brevity, we omit $n$ in the subscript.

\paragraph{Terms $R_{QHL}$ and $R_{QLH}$.}  

The low-high terms are analogous to the high-low terms; thus, we concentrate our analysis on the latter. Its $q$'th frequency component is \ali{
\nb_{\va}\pa_t^r R_{QHLq}^{j\ell} = \nb_{\va}\pa_t^r \mathcal{R}^{j\ell}_a P_{\approx q}[P_{q+1}w_J T^a P_{\le q-1}w_J]. \label{eq:qQuadHLterm}
}
We select $\bar{q}$ such that $2^{\bar{q}} \sim \Xi$.

Consider the case $q \leq \bar{q}$.  In this case we express $w_J = \nabla_i \nabla_b r_J^{ib}$ in the rightmost copy of $w_J$ and bound the operator norm of $T P_{\le q-1} \nabla \nabla$. By doing so, we obtain:
\ali{
\nb_{\va}\pa_t^r R_{QHLq}^{j\ell} &= \tisum \mathcal{R}^{j\ell}_a P_{\approx q}[ \nb_{\va_1} \pr_t^{r_1} P_{q+1} w_J \nb_{\va_2} \pr_t^{r_2} T^a P_{\leq q-1} \nb_i \nb_b r_J^{ib}]  \label{eq:QuadHLlow} \\
\co{\eqref{eq:QuadHLlow}} &\ls \| \mathcal{R}^{j\ell}_a P_{\approx q} \| \tisum  \co{\nb_{\va_1} \pr_t^{r_1} P_{q+1} w_J} \| T^a P_{\leq q-1} \nb_i \nb_b \| \co{\nb_{\va_2} \pr_t^{r_2} r_J^{ib} } \\
&\ls 2^{-q} 2^{2 q} \tisum (N \Xi)^{|\va_1|} \mu^{r_1} S_w (N \Xi)^{|\va_2|} \mu^{r_2} S_r  \\
&\ls 2^q \Xi^{-2} (N \Xi)^{|\va|} \mu^{r} S_w^2 \label{eq:QhighLowboundedLow}
}
For high frequencies $q \geq \bar{q}$, we first prove the preliminary
bound
\ali{
\co{\nb_{\va} \pr_t^r P_{\leq q} T^a w_J } &\ls \co{\nb_{\va} \pr_t^r P_{\leq \bar{q}} T^a w_J } + \sum_{q = \bar{q}}^\infty \co{ T^a P_q \nb_{\va} \pr_t^r w_J } \notag \\
&\ls \| P_{\leq \bar{q}} T^a \nb_i \nb_b \| \co{\nb_{\va} \pr_t^r r_J^{ib}} + \sum_{q = \bar{q}}^\infty 2^{-q} \co{\nb \nb_{\va} \pr_t^r w_J } \notag \\
&\ls 2^{2 \bar{q}} \Xi^{-2} (N \Xi)^{|\va|} \mu^{r} S_w + 2^{-\bar{q} } \Xi (N \Xi)^{|\va|} \mu^{r} S_w \notag \\
\co{\nb_{\va} \pr_t^r P_{\leq q} T^a w_J }&\ls (N \Xi)^{|\va|} \mu^{r} S_w
}
We now apply this estimate to 
\ali{
\co{\eqref{eq:qQuadHLterm}} &\ls \tisum \| \mathcal{R}^{j\ell}_a P_{\approx q} \| \co{\nb_{\va_1} \pr_t^{r_1} P_{q+1} w_J} \co{\nb_{\va_2} \pr_t^{r_2} P_{\leq q-1} T^a w_J} \notag \\
&\ls \tisum 2^{-q} [(N \Xi)^{|\va_1|} \mu^{r_1} S_w][(N \Xi)^{|\va_2|} \mu^{r_2} S_w] \\
&\ls 2^{-q} (N \Xi)^{|\va|} \mu^{r} S_w^2. \label{eq:highTermHLqbd}
}
Summing \eqref{eq:QhighLowboundedLow} over $q < \bar{q}$ and \eqref{eq:highTermHLqbd} over $q \geq \bar{q}$ yields \eqref{eq:withoutLogLoss} for $R_{QHL}$.

\paragraph{Term $R_{QHH}.$} 

We decompose the high-high frequency interactions into three parts: those with the operators applied in the order \(P_{q+1}, P_q\); those with the order reversed; and those involving both \(P_{q+1}\).

We begin with the third group of terms. We can consider the other two terms similarly, as a single group. Note that we don't consider them separately because we need to consider those two together in order to get an anti-divergence. For brevity, we only demonstrate the part with both operators being $P_{q+1}.$

We need to bound 
\begin{align}
\nb_{\va}\pa_t^r K_{q1}^{j\ell} \ast [w_J,w_J] = \nb_{\va}\pa_t^r \int P_{q+1}w_J(x-h_1) P_{q+1} w_J(x-h_2) K_{q1}^{j\ell}(h_1,h_2) dh_1 dh_2. \label{eq:bilinConvFormq}
\end{align}
Note that we can distribute the derivatives inside the integral using the product rule.  We first consider the case where $q > \bar{q}$.
\ali{
\co{\eqref{eq:bilinConvFormq}} &\ls \tisum \| K_q \|_{L^1} \co{\nb_{\va_1}\pa_t^{r_1} w_J} \co{ \nb_{\va_2}\pa_t^{r_2} w_J} \\
&\ls \tisum 2^{-q} (N \Xi)^{|\va|} \mu^{r} S_w^2 \label{eq:highHHbd}
}
For $q \leq \bar{q}$, we write $w_J = \nb_i z_J^i$ and integrate by parts to find
\ali{
K_{q1}^{j\ell} \ast [w_J,w_J] &= \int z_J^a(x - h_1) z_J^b(x - h_2) \nb_a \nb_b K_q^{j\ell}(h_1, h_2) dh_1 dh_2 \label{eq:integratedByParts} \\
\co{\nb_{\va} \pr_t^r \eqref{eq:integratedByParts} } &\ls \tisum \| \nb^2 K_q \|_{L^1} \co{\nb_{\va_1} \pr_t^{r_1} z_J} \co{\nb_{\va_2} \pr_t^{r_2} z_J} \\
&\ls 2^{q} (N \Xi)^{|\va|} \mu^r S_z^2 \sim 2^q \Xi^{-2} (N \Xi)^{|\va|} \mu^r S_w^2  \label{eq:HHlow}
}
Now we sum \eqref{eq:highHHbd} over $q > \bar{q}$ and \eqref{eq:HHlow} over $q \leq \bar{q}$ to obtain \eqref{eq:withoutLogLoss} for $R_{QHH}$.

\subsection*{The mollification terms $R_{M,n}$.}
Recall that $R_{M,n}$ solves
$$\nabla_j\nabla_\ell R_{M,n}^{j\ell} = T^\ell(\th - \th_\ep) \nabla_\ell w_n + T^\ell w_n \nabla_\ell(\th - \th_\ep).$$
Here, by definition, $\theta_\epsilon := P_{\le \hat q} \theta.$ Thus $\theta-\theta_\epsilon$ only has frequencies above $2^{\hat q}$. The idea is to expand these terms and observe that every single one of the $\theta-\theta_\epsilon$ terms is of high frequency $>2^{\hat q}$. Thus 
$\theta-\theta_\epsilon = P_{>\hat q} \theta.$

We have 
$$\nabla_j R_{M,n}^{j\ell} = \sum_{J(n)} \left(  (\theta-\theta_\epsilon)T^\ell w_J + w_J T^\ell (\theta-\theta_\epsilon) \right) \tilde{\chi}_k(t).$$
For simplicity we write
$$\nabla_j R_{M,n}^{j\ell} \simeq(\theta-\theta_\epsilon)T^\ell w_J + w_J T^\ell (\theta-\theta_\epsilon) .$$
From now on, we will suppress the $\tilde{\chi}_k(t)$ and summation notation.
 
We have  
$$R^{j\ell}_{M,n} = R_{MLH,n}^{j\ell} + R_{MHL,n}^{j\ell} + R_{MHH,n}^{j\ell}.$$

Taking spatial and time derivatives of the LH term, we have: 
\begin{align*}
\partial_t \nabla_{\vec a} R^{j\ell}_{MLHq,n} 
&= \overset{\sim}{\sum} \partial_t \calR^{j\ell}_a P_{\approx q} [\nabla_{\vec a_1} P_{\le q-1} T^a P_{>\hat q}\theta  \nabla_{\vec a_2} P_{q+1} w_J] \\
&= \overset{\sim}{\sum} \calR^{j\ell}_a P_{\approx q} [\partial_t \nabla_{\vec a_1} P_{\le q-1} T^a P_{>\hat q}\theta  \nabla_{\vec a_2} P_{q+1} w_J] \\
&\quad + \overset{\sim}{\sum} \calR^{j\ell}_a P_{\approx q} [\nabla_{\vec a_1} P_{\le q-1} T^a P_{>\hat q}\theta  \nabla_{\vec a_2} P_{q+1} \partial_t w_J] 
\end{align*}

Taking spatial and time derivatives of the HL term, we have:  
\begin{align*}
\partial_t \nabla_{\vec a} R^{j\ell}_{MHLq,n}
&= \overset{\sim}{\sum} \partial_t  \calR^{j\ell}_a P_{\approx q} [\nabla_{\vec a_1}P_{q+1}P_{>\hat q}\theta T^a P_{\le q-1} \nabla_{\vec a_2}w_J] \\
&= \overset{\sim}{\sum} \calR^{j\ell}_a P_{\approx q} [\partial_t \nabla_{\vec a_1}P_{q+1}P_{>\hat q}\theta T^a P_{\le q-1} \nabla_{\vec a_2}(w_J)] \\
&\quad + \overset{\sim}{\sum}\calR^{j\ell}_a P_{\approx q} [\nabla_{\vec a_1}P_{q+1}P_{>\hat q}\theta T^a P_{\le q-1} \partial_t \nabla_{\vec a_2} w_J].
\end{align*}
We can obtain a similar expression for the derivatives of \(R_{MHHq,n}\), which for conciseness we omit.

\paragraph{The term $R_{MHH}.$}
We have 
$$\nabla_j R^{j\ell}_{MHHq,n} = P_{q+1}(\theta - \theta_\epsilon) T^\ell P_q w_J + P_{q} w_J T^\ell P_{q+1} (\theta - \theta_\epsilon).$$
We must treat both terms together (rather than only one of the two terms at a time), since there is no anti-divergence if these two terms are separated from each other. 

We have
\begin{align*}
\pa_t^r\nb_{\va}R^{j\ell}_{MHHq,n} &= \overset{\sim}{\sum}\int \pa_t^{r_1}\nb_{\va_1}(\theta(x-h_1) - \theta_\epsilon(x-h_1))\pa_t^{r_2}\nb_{\va_2} w_J(x-h_2) K_q^{j\ell}(h_1,h_2) \, dh_1 dh_2\\
&=\overset{\sim}{\sum} \int \pa_t^{r_1}\nb_{\va_1}(P_{\approx q}[\theta-\theta_\epsilon](x-h_1)) \pa_t^{r_2}\nb_{\va_2}w_J(x-h_2) K_q^{j\ell}(h_1,h_2) \, dh_1 dh_2 \end{align*}

We can bound each term as follows:

1. For the first term, we have:
   \begin{align*}
   &\left\|\int \partial_t \nabla_{\vec a_1}(P_{\approx q}[\theta-\theta_\epsilon]) \nabla_{\vec a_2} w_J K_q^{j\ell} \, dh_1 dh_2\right\|_0 \\
   &\lesssim \|K_q^{j\ell}\|_1 \|\partial_t \nabla_{\vec a_1}(P_{\approx q}[\theta-\theta_\epsilon])\|_0 \|\nabla_{\vec a_2} w_J\|_0  \\
   &\lesssim 2^{-q} [(N\Xi)^{|\va_1|} (\ost{\Xi}\ost{e}_u^{1/2}) N^{-1} e_u^{1/2} ] [(N \Xi)^{|\vec a_2|} S_w] 
   \end{align*}

2. For the second term, we have:
   \begin{align*}
   &\left\|\int \nabla_{\vec a_1}(P_{\approx q}[\theta-\theta_\epsilon]) \partial_t \nabla_{\vec a_2} w_J K_q^{j\ell} \, dh_1 dh_2\right\|_0 \\
   &\lesssim \|K_q^{j\ell}\|_1 \|\nabla_{\vec a_1}(P_{\approx q}[\theta-\theta_\epsilon])\|_0 \|\partial_t \nabla_{\vec a_2} w_J\|_0  \\
   &\lesssim [2^{-q}] [(N\Xi)^{|\vec a_1|}N^{-1} e_u^{1/2}] [(N\Xi)^{|\vec a_2|} \mu S_w]  
   \end{align*}
The sum of these terms is bounded by 
$$[2^{-q}] \tsi [(N\Xi)^{|\vec a_1|}N^{-1} e_u^{1/2}] [(N\Xi)^{|\vec a_2|} S_w]$$
where $\tsi$ is the inverse timescale $$\tsi:= (N \Xi)^{3/2} D_R^{1/2}.$$

Now, summing over $q > \hat q-1$, we get $\sum 2^{-q} \sim \hxi^{-1}$ and:
\begin{align*}
\|\partial_t \nabla_{\vec a} R^{j\ell}_{MHH,n}\| \ls \sum_q \|\partial_t \nabla_{\vec a} R^{j\ell}_{MHHq,n}\|_0 
&\lesssim \hxi^{-1} \tsi N^{|\va|-1} \Xi^{|\va|} e_u^{1/2} S_w
\end{align*}
More generally,
\begin{align*}
\|\partial_t^r \nabla_{\vec a} R^{j\ell}_{MHH,n}\| 
\ls \hxi^{-1} \ost{\tau}^{-r} N^{|\va|-1} \Xi^{|\va|} e_u^{1/2} S_w
\end{align*}

\paragraph{The terms $R_{MHL}$ and $R_{MLH}$.}
Our first group of terms is
$$R_{MHL1,n}^{j\ell} 
= 
\sum_{q \ge \hat q - 1}
\calR^{j\ell}_a P_{\approx q}
\left[ P_{\le q-1} T^a(\theta-\theta_\epsilon) P_{q+1} w_J \right].$$As usual, we add a subscript $q$ to label each term in the sum. So, for $R_{MHL1,n}$, we'll call the individual pieces $R_{MHL1q,n}$.
$$R_{MHL1q,n} := \calR^{j\ell}_a P_{\approx q}
\left[ P_{\le q-1} T^a(\theta-\theta_\epsilon) P_{q+1} w_J \right]$$
For $0 \leq r + |\va| \leq L$, we have
\begin{align*}
\|\partial_t^r\nb_{\va}R_{MHL1q,n}^{j\ell} \|_0
&\ls 
\|\calR P_{\approx q}
\|_{op} \|\pa_t^{r_1}\nb_{\va_1} P_{\le q-1} T (\theta-\theta_\epsilon)\|_0 \|\pa_t^{r_2}\nb_{\va_2} P_{q+1} w_J\|_0 \\
&\ls [2^{-q}]  
[\ost{\tau}^{-r_1}(N\Xi)^{|\va_1|}N^{-1}e_u^{1/2}] [\mu^{r_2}(N\Xi)^{|\va_2|}S_w]
\end{align*}
Thus, since $\mu \leq \ost{\tau}^{-1}$, 
$$\|\partial_t^r\nb_{\va}R_{MHL1,n}^{j\ell} \|_0
\ls \hxi^{-1}\ost{\tau}^{-r}N^{|\va|-1} \Xi^{|\va|} e_u^{1/2} S_w \leq (N \Xi)^{|\va|} \ost{\tau}^{-r} [(N \Xi)^{-1} e_u^{1/2} S_w ] $$
We would like the expression in brackets to be bounded by $C D_R/N$, which is indeed the case.  One can check this by recalling that $e_u^{1/2} = \Xi^{1/2} D_u^{1/2}$, while
\[ S_w \leq \mu^{-1} \Xi^2 D_{R} = N^{-1/2} \Xi^{-3/2 + 2} D_R^{1/2} \leq D_u^{-1/2} \Xi^{1/2} D_R \]
thanks to \eqref{eq:NlowerBd}, which implies $N \geq D_u/D_R$.

The bounds for $\pa_t^r \nb_{\va} R^{j\ell}_{MLHq,n}$ are similar to the bounds for $\pa_t^r \nb_{\va} R^{j\ell}_{MHLq,n}$.

\section{Convex integration}
Define the representative index set ${\mathcal I}_+ := \F \times \Z \times \{0,\ldots,\Ga-1\}$ and the full index set ${\mathcal I}:={\mathcal I}_+\cup\{\bar I:I\in{\mathcal I}_+\}$, where the conjugate index is defined by $\overline{(f,k,n)}=(-f,k,n)$.  Thus every conjugate pair is represented exactly once by an index $I=(f,k,n)\in{\mathcal I}_+$.  The last index corresponds to precisely the $\Gamma$ stages $n=0,\ldots,\Gamma-1$ appearing in the recurrence for $S_{(\Gamma)}$.  Choose $\la \in 2\pi \Z$ such that $\la \sim N \Xi $.  The oscillatory wave has the form
\ali{
\Th &= \sum_{I\in\mathcal I} \Th_I, \qquad \Th_I = g_{[f,k,n]}(\mu t)  P_I [ e^{i \la \xi_I}  \th_I ] \label{eq:ThI} \\
\th_I &= (2\la)^{1/2}\ga_I, \qquad 
\ga_{(f,k,n)} = \chi_k e_n^{1/2}(t) \ga_f\left(p_I\right) \label{eq:gaIdefn} \\
\check{p}_I &= \left(\fr{R_{(n)}^{j\ell}}{M_e D_{R,n}}, \nb \check{\xi}_k\right) \label{eq:pchkI} \\
p_I &= \left(\fr{\tilde{R}_{(n)}^{j\ell}}{M_e D_{R,n}}, \nb \xi_k\right) \label{eq:pI}
}
where $P_I$ is a frequency localization operator whose symbol, which is defined via scaling, is equal to $1$ on a $\la|f|/5$-neighborhood of $\la f$ and is supported in a $\la|f|/3$-neighborhood of $\la f$.  With these definitions, the conjugate wave indexed by $\bar{I}=(-f,k,n)$ satisfies 
$\Th_{\bar{I}} = \overline{\Th_I}$, $\xi_{\bar{I}} = - \xi_I$, $g_{[\bar I]}=g_{[I]}$, and the symbol of $P_{\bar{I}}$ is a reflection of that of $P_I$.

We will use mollification to define $\widetilde{R}_{(n)}$.  We postpone for now the necessary estimates on $\widetilde{R}_{(n)}$ and $\nb \xi_I$ that ensure the construction is well-defined.  In particular, we will have to show that $\widetilde{R}_{(n)}$ and $\nb \xi_I$ do not escape the domains of $\ga_f$ and $B^{j\ell}$.  We will also require that $e_n(t) = M_e D_{R,n}$ on the support of $\tilde{R}_{(n)}$, and that $\la \nb \xi_I$ takes values in the set where the symbol of $P_I$ equals $1$.

Notice that, by construction and the disjointness of supports of the functions $g_{[f,k,n]}$, we have the crucial disjointness property
\ali{
\suppt \Th_I \cap \suppt \Th_J &= \emptyset \qquad \mbox{ if } I \notin \{J, \bar{J} \} \label{eq:nonsimultaneous}
}
Now let $\tilde{\th}_{\ep} = \th_\ep + w = \th_\ep +\sum_{n=0}^{\Gamma-1}w_n$ and $\tilde{u}_\ep^\ell  = T^\ell \tilde{\th}_\ep = u_\ep^\ell + T^\ell w$. 

We obtain the following estimates for $\tilde{u}_\ep$.  
\ali{
 \co{\nb_{\va} \tilde{u}_\ep } &\lesssim_{\va} \hn^{(|\va| - \unl)_+} \Xi^{|\va|} e_u^{1/2}, \qquad  \qquad\qquad  \label{eq:tildeuepbd1}\\
 \co{\nb_{\va} \Ddt \tilde{u}_\ep } &\lesssim_{\va} \hn^{(|\va| + 1 - \unl)_+} \Xi^{|\va|}  (\Xi e_u^{1/2}) e_u^{1/2},  \label{eq:tildeuepbd2}
}
Notice that these are the same estimates that hold for $u_\ep$ except that the losses of powers of $\hn$ occur earlier (after $\unl$ derivatives rather than $L$).  These bounds follow from \eqref{eq:newtVelocBd2}.

We will also need a bound on the advective derivative of $\tilde{u}_\ep$ along its own flow.  Setting $\Dtdt = \pr_t + \tilde{u}_\ep \cdot \nb$, the following bound suffices:
\ali{
 \co{\nb_{\va} \Dtdt \tilde{u}_\ep } &\lesssim_{\va} (\Xi e_u^{1/2}) N^{(|\va| + 1 - L)_+/L} \Xi^{|\va|} e_u^{1/2}. \label{eq:tildeuepbd3}
}
This bound is a corollary of \eqref{eq:tildeuepbd1}-\eqref{eq:tildeuepbd2} and the following Lemma, which is generally useful when converting bounds between different time and advective derivatives.

\begin{lem}[Interchange Lemma] \label{lem:interchangeDerivs} Let $\mrg{D}_t$ be one of the operators $\mrg{D}_t \in \{ \pr_t, \Ddt, \Dtdt \}$.  Consider any inverse timescale $\zeta  \geq \Xi e_u^{1/2}$.  Define the weighted norm of a smooth tensor field $F$ by
\ali{
\mrg{H}_\zeta[F] &= \max_{0 \leq r \leq 1} \max_{ 0 \leq |\va| + r \leq L'} \fr{\co{\nb_{\va} \mrg{D}_t^r F}}{\hn^{(|\va| + r  - \unl)_+} \Xi^{|\va|} \zeta^r  }
}
If $\zeta$ is omitted in the notation, set $\mrg{H}[F] = \mrg{H}_{\Xi e_u^{1/2}}$.   

Then there exist constants depending only on $L'$ such that \[ \tilde{H}_\zeta[F] \lesssim \bar{H}_\zeta[F] \lesssim H^{\pr_t}_\zeta[F] \lesssim \tilde{H}_\zeta[F] .\]
\end{lem}
\begin{proof}
We show only that $H^{\pr_t}_\zeta[F] \lesssim \tilde{H}_\zeta[F]$ as the other directions are similar
\ALI{
\nb_{\va} \pr_t F &= \nb_{\va} \Dtdt F - \nb_{\va}[ \tilde{u}_\ep^i \nb_i F ] \\
\co{\nb_{\va} \pr_t F} &\lesssim \hn^{(|\va| + 1  - \unl)_+} \Xi^{|\va|} \zeta \tilde{H}[F] + \sum_{|\va_1| + |\va_2| = |\va|} \co{ \nb_{\va_1} \tilde{u}_\ep } \co{\nb_{\va_2} \nb_i F } \\
&\lesssim \hn^{(|\va| + 1  - \unl)_+} \Xi^{|\va|} \zeta \tilde{H}[F] \\
&+ \sum_{|\va_1| + |\va_2| = |\va|} \hn^{(|\va_1| - \unl)_+} \Xi^{|\va_1|} e_u^{1/2} \hn^{(|\va_2| + 1 - \unl)_+} \Xi^{|\va_2| + 1} \tilde{H}[F]
}
We now apply the counting inequality $(x - z)_+ + (y-z)_+ \leq (x + y - z)_+$, $x, y, z \geq 0$ with $x = |\va_1|$, $y = |\va_2| + 1$, $z = \unl \geq 0$, and recall $\zeta \geq \Xi e_u^{1/2}$, to obtain
\ALI{
\co{\nb_{\va} \pr_t F } &\lesssim \hn^{(|\va| + 1  - \unl)_+} \Xi^{|\va|} \zeta \tilde{H}[F],
}
which is the desired estimate after dividing through by the prefactor of $\tilde{H}[F]$.
\end{proof}

We will also use the following chain rule and product rule
\begin{prop}[Chain rule and product rule for weighted norms]  \label{prop:chainRule}
Consider the operators $\mrg{D}_t \in \{ \pr_t, \Ddt, \Dtdt \}$  and let $F$ be $C^\infty$.  Let $G$ be a $C^\infty$ function defined on a compact neighborhood of the image of $F$.  Then
\ali{
\mrg{H}_\zeta[G(F)] &\ls (1 + \mrg{H}_\zeta[F])^{L'} \\
\mrg{H}_\zeta[F_1 F_2] &\ls \mrg{H}_\zeta[F_1]\mrg{H}_\zeta[F_2]
}
with implicit constants depending on $L'$.
\end{prop}
\begin{proof} We compute for $0 \leq r \leq 1$, $0 \leq r + |\va| \leq L'$ 
\ALI{
\nb_{\va} \mrg{D}_t^r G(F) &= \sum_{k = 0}^{|\va| + r} \overset{\sim}{\sum} \pr^k G(F) \prod_{i=0}^k \nb_{\va_i} \mrg{D}_t^{r_i} F 
}
where the sum is over appropriate indices such that $\sum_i |\va_i| = |\va|$ and $\sum_i r_i = r$.  Then
\ALI{
\co{\nb_{\va} \mrg{D}_t^r G(F)} &\ls \sum_{k = 0}^{|\va| + r} \co{\pr^k G} \prod_{i=0}^k \co{ \nb_{\va_i} \mrg{D}_t^{r_i} F} \\
&\ls \sum_{k = 0}^{|\va| + r} \prod_{i=0}^k \left[\hn^{(|\va_i| + r_i - \unl)_+} \Xi^{|\va_i|} \zeta^{r_i} \mrg{H}_\zeta[F] \right] \\
&\ls \hn^{(|\va| + r - \unl)_+} \Xi^{|\va|} \zeta^{r} (1 +  \mrg{H}_\zeta[F])^{L'},
}
which is the desired bound.

The product rule can be proven by direct computation, but it can also be deduced from the Chain Rule as follows.  The vector-valued function $(\fr{F_1}{\mrg{H}_\zeta[F_1]}, \fr{F_2}{\mrg{H}_\zeta[F_2]})$ takes values in $\{ (u,v) 
: \max \{ \|u \|, \| v \| \} \leq 1 \}$ and $G(u,v) = uv$ is smooth in a compact neighborhood of this set.  We then have by the chain rule
\ALI{
\mrg{H}_\zeta[F_1 F_2] &= \mrg{H}_\zeta[F_1] \mrg{H}_\zeta[F_2] \mrg{H}_\zeta\left[\fr{F_1}{\mrg{H}_\zeta[F_1]} \fr{F_2}{\mrg{H}_\zeta[F_2]} \right] \\
&\ls \mrg{H}_\zeta[F_1] \mrg{H}_\zeta[F_2] \left( 1 + \mrg{H}_\zeta\left[\fr{F_1}{\mrg{H}_\zeta[F_1]}\right] + \mrg{H}_\zeta\left[\fr{F_2}{\mrg{H}_\zeta[F_2]}\right] \right)^{L'} \\
&\ls \mrg{H}_\zeta[F_1] \mrg{H}_\zeta[F_2].
}
\end{proof}

\begin{prop} \label{prop:primeWeighted} For $\mrg{D}_t \in \{ \pr_t, \Ddt, \Dtdt \}$ define the prime weighted norm
\ali{
\mrg{H}'_\zeta[F] = \max_{0 \leq r \leq 1} \max_{0 \leq |\va| + r \leq L'} \fr{\co{\nb_{\va}\mrg{D}_t F}}{\hn^{(|\va| + r- (\unl - 1))_+} \Xi^{|\va|}\zeta^r}
}
Then the natural analogues of Lemma~\ref{lem:interchangeDerivs} and Proposition~\ref{prop:chainRule} hold for the prime weighted norms as well.
\end{prop}
We omit the proof, which is essentially the same as that of Lemma~\ref{lem:interchangeDerivs} and Proposition~\ref{prop:chainRule}.

We are now ready to define $\tilde{R}_{(n)}$.  We mollify along the flow at the scale $\epsilon_t$ from \eqref{eq:ept}, which is defined using the natural stress timescale, rather than mollifying in time at fixed spatial points.  This allows the estimates from earlier work to apply directly.  Set
\ali{
\tilde{R}_{(n)} &= \eta_{\ep_t} \ast_\Phi R_{(n)} = \int R_{(n)}(\Phi_s(t,x)) \eta_{\ep_t}(s) ds 
}
where $\eta_{\ep_t}(s) = \ep_t^{-1} \eta(s/\ep_t) \geq 0$ is a standard mollifying kernel supported in $|s| < \ep_t$ and where $\Phi_s(t)$ is the flow map of $\pr_t + \tilde{u}_\ep \cdot \nb$, which is the unique solution to
\ali{
\Phi_s(t,x) &= (t+s, \Phi_s^1(t,x), \Phi_s^2(t,x)),  \\
\fr{d}{ds} \Phi_s^i &= \tilde{u}_\ep^i(\Phi_s(t,x)) \qquad i = 1, 2 \\
\Phi_0(t,x) &= (t,x).
}
This is the flow convolution defined in \cite[Section~18.2, equations~(617)--(618)]{isett}.  We adapt the estimates of \cite[Sections~18.4--18.6]{isett} to our parameters.  In this comparison, the source quantities $R_{\epsilon_x}$, $R_\epsilon$, $e_R$, $v_\epsilon$, and $\overline D/\partial t$ correspond to our $R_{(n)}$, $\tilde R_{(n)}$, $D_{R,n}$, $\tilde u_\epsilon$, and $\Dtdt$, respectively.

With this definition we can check that the coefficients $\ga_{(f,k,n)}$ are well-defined.  To see this claim, first observe that, since $\eta_{\ep_t} \geq 0$, we have $\co{\tilde{R}_{(n)}} \leq \co{R_{(n)}} \leq A_L D_{R,(n)}$ for all $n$.  Furthermore, since mollification along the flow grows the time support by at most $\ep_t \leq \tau$, we also have by construction that $e_n(t) = M_e D_{R,(n)}$ on the support of $\tilde{R}_{(n)}$.  Hence the construction of the coefficient $\ga_{(f,k,n)}$ is well-defined.  We now need estimates for $\tilde{R}_{(n)}$ and its spatial and advective derivatives.

The Newton estimates \eqref{eq:Rnplus1}--\eqref{eq:Rnplus1Ddt} are stated with $\Ddt$.  Since $\Dtdt=\Ddt+(T w)\cdot\nabla$, Lemma~\ref{lem:interchangeDerivs} and the Newton-velocity bounds give, for $n\geq1$,
\[
\|\nabla_{\vec a}\Dtdt R_{(n)}\|_{C^0}
\lesssim \widehat N^{(|\vec a|+1-\underline L)_+}\Xi^{|\vec a|}(\Xi e_u^{1/2})D_{R,n}.
\]
The same interchange lemma using the bounds \eqref{eq:nbRepbd1}-\eqref{eq:nbRepbd2} yields the above estimate for $R_{(0)}=R_\epsilon$.  Thus the flow-mollification estimates of \cite[Sections~18.4--18.6]{isett} apply uniformly with stress level $D_{R,n}$ and loss-free order $\underline L=L-3$.
\begin{prop}[Stress estimates] \label{prop:stress} 
\ali{
\co{R_{(n)}-\tilde{R}_{(n)}} 
&\lesssim \ep_t \co{\Dtdt R_{(n)}}
\lesssim (D_u/D_R)^{-1/2}N^{-1/2}D_{R,n} \label{eq:RerrorMollPhi} \\
\co{\nb_{\va}\tilde{R}_{(n)}} &\lesssim_{\va} \hn^{(|\va| -\unl)_+} \Xi^{|\va|} D_{R,n} \\
\co{\nb_{\va} \Dtdt \tilde{R}_{(n)} } &\lesssim_{\va} (\Xi e_u^{1/2}) \hn^{(|\va| + 1 - \unl)_+} \Xi^{|\va|} D_{R,n} \\
\co{\nb_{\va} \Dtdt^2 \tilde{R}_{(n)} } &\lesssim_{\va} \ep_t^{-1}(\Xi e_u^{1/2}) \hn^{(|\va| + 1 - \unl)_+} \Xi^{|\va|} D_{R,n}.
}   
\end{prop}
The first line is also immediate from 
\ALI{
\tilde{R}_{(n)} - R_{(n)} &= \int \left[ R_{(n)}(\Phi_s(t,x)) - R_{(n)}(t,x) \right] \eta_{\ep_t}(s) ds,
}
the fundamental theorem of calculus along $\Phi_s$ and
$\int |s|\eta_{\epsilon_t}(s)\,ds\lesssim\epsilon_t.$ 
The reason for the appearance of $\ep_t^{-1}$ in the last estimate is the appearance of $\eta_{\ep_t}'(s)$ in following calculation
\ALI{
\Dtdt \int R_{(n)}(\Phi_s(t,x)) \eta_{\ep_t}(s) ds &= \int \fr{d}{ds} R_{(n)}(\Phi_s(t,x)) \eta_{\ep_t}(s) ds = - \int R_{(n)}(\Phi_s(t,x)) \eta_{\ep_t}'(s) ds \\
\Dtdt^2 \int R_{(n)}(\Phi_s(t,x)) \eta_{\ep_t}(s) ds &= - \int \Dtdt R_{(n)}(\Phi_s(t,x)) \eta_{\ep_t}'(s) ds
}
Thus, the second advective derivative costs $\|\eta'_{\epsilon_t}\|_{L^1}\lesssim\epsilon_t^{-1}$.  We retain the stage-relative factor $D_{R,n}$ for the $\Gamma$ active stages $n=0,\ldots,\Gamma-1$.

We define the phase functions $\xi_I$ to solve
\ali{
(\pr_t + \tilde{u}_\ep^j \nb_j ) \xi_I &= 0 \label{eq:transportPhase} \\
\xi_{(f,k,n)}(k\tau, x) &= \check{\xi}_{(f,k)}(k\tau, x) = f \cdot x 
}
Notice that $\xi_I$ and $\check{\xi}_I$ have the same initial data but differ in terms of which vector field transports them.

We obtain the following estimates for $\xi_I$:
\begin{prop}[Phase function estimates] \label{prop:phase}  The phase functions satisfy the following bounds on the interval $[t(I) - \tau, t(I) + \tau]$
\ali{
\co{\nb_{\va} \nb \xi_I} &\lesssim_{\va} \hn^{(|\va| + 1 -\unl)_+} \Xi^{|\va|}, \label{eq:phaseGrad}  \\
\co{\nb_{\va} \Dtdt \nb \xi_I } &\lesssim_{\va} \hn^{(|\va| + 1 -\unl)_+} \Xi^{|\va|} (\Xi e_u^{1/2}),  \\
\co{\nb_{\va} \Dtdt^2 \nb \xi_I } &\lesssim_{\va} \hn^{(|\va| + 2 -\unl)_+} \Xi^{|\va|} (\Xi e_u^{1/2})^2  \label{eq:nbADtdt2nbxi}
}
\end{prop}
\begin{proof}
These estimates are derived by differentiating the transport equation \eqref{eq:transportPhase} and applying Gronwall to the weighted jet of $\xi_I$.  Propositions~17.3 and~17.5 and Propositions~17.6--17.7 of \cite{isett} give, respectively, the spatial, first material derivative, and second material derivative estimates.  The inputs are exactly \eqref{eq:tildeuepbd1}--\eqref{eq:tildeuepbd3} and the restriction $|t-t(I)|\leq\tau\lesssim(\Xi e_u^{1/2})^{-1}$.  To match notation, take the source loss-free order $L$ to be our $\underline L=L-3$, denote the source factor $N^{1/L}$ by $\widehat N$, and replace the source $v_\epsilon$, $e_v^{1/2}$, and $\overline D/\partial t$ by $\tilde u_\epsilon$, $e_u^{1/2}$, and $\Dtdt$.  For example, the source factor $N^{(m-L)_+/L}$ becomes $\hn^{(m-\underline L)_+}$, which is the loss appearing in the displayed bounds.
\end{proof}

We will need a good estimate on how close the phase gradients are to those that were used in the Newton step.  The equation we need to analyze is
\ali{
(\pr_t + u_\ep^j \nb_j) (\check{\xi}_J - \xi_J) &= T^j w \nb_j \xi_J \\
(\pr_t + u_\ep^j \nb_j) (\nb_a \check{\xi}_J - \nb_a \xi_J) &= - \nb_a(T^j w \nb_j \xi_J) - \nb_a u_\ep^j \nb_j(\check{\xi}_J - \xi_J)
}
Again, the initial data for $\nb \check{\xi}_J$ and $\nb \xi_J$ are equal at time $t(I)$.  From this equation, we use the fact that the time scale $\tau \leq (\log \hxi)^{-1} (\Xi e_u^{1/2})^{-1}$, and apply the method of characteristics and Gronwall to obtain
\ali{
\| \nb \check{\xi}(t) - \nb \xi_I(t) \|_0 &\leq (\Xi e_u^{1/2}) \int_{t(I)}^t \| \nb \check{\xi}(s) - \nb \xi_I(s) \|_0 ds + \tau ( \co{\nb(T^j w \nb_j \xi_J)} ) \notag \\
\| \nb \check{\xi}(t) - \nb \xi_I(t) \|_0 &\lesssim e^{C \Xi e_u^{1/2} \tau} \tau \Xi S_{u} \lesssim N^{-1/2} (\D_u/\D_R)^{-1/2}. \label{eq:flowErrorFirst}
}
In particular, since $N \geq \hat{C}$ by \eqref{eq:NlowerBd}, we have that $|\nb \xi_I - \nb \check{\xi}_I| \lesssim \hc^{-1/2}$.  For all $\hc$ large enough, $\nb \xi_I$ take values in the domain of the functions $\ga_f$ and $B^{j\ell}$, so that the construction of the convex integration wave $\Th$ will be well-defined, and also satisfies $|\nb \xi_I - f| \leq |f|/5$, so that $\la \nb \xi_I$ takes values in the region where the multiplier for $P_I$ is equal to $1$.

We obtain the following bounds on the amplitudes $\ga_I$ defined in \eqref{eq:gaIdefn}.

\begin{prop}[Amplitude bounds] \label{prop:amplitudeBounds}
\ali{
\co{\nb_{\va} \Dtdt^r \ga_I } &\lesssim_{\va} \hn^{(|\va| + 1 - \unl)_+} \Xi^{|\va|} \tau^{-r} D_{R,n}^{1/2}\, \qquad 0 \leq r \leq 1,\label{eq:firstAdvecAmplitudes} \\
\co{\nb_{\va} \Dtdt^2 \ga_I } &\lesssim_{\va} \hn^{(|\va| + 2 - \unl)_+} \Xi^{|\va|}  \ep_t^{-1}(\Xi e_u^{1/2})D_{R,n}^{1/2}. \label{eq:SecondAdvecGaI}
}    
\end{prop}
\begin{proof}  Put $\check R_n=\tilde R_{(n)}/(M_eD_{R,n})$ and apply the weighted chain and product estimates to
\[
\gamma_{(f,k,n)}=\chi_k e_n^{1/2}(t)\gamma_f(\check R_n,\nabla\xi_k).
\]
For the stress and phase gradients, the advective derivative costs $\Xi e_u^{1/2}$ by Propositions~\ref{prop:stress} and~\ref{prop:phase}, whereas differentiating $\chi_k$ or $e_n^{1/2}$ costs $\tau^{-1}$.  Equivalently, we may use the standard estimates from \cite{isett} with the source velocity level $e_v^{1/2}$ replaced by $\widetilde e_u^{1/2}$ from \eqref{eq:effectiveEu}, since $\Xi\widetilde e_u^{1/2}\sim_{L,\eta}\tau^{-1}$.  This proves \eqref{eq:firstAdvecAmplitudes}.

For two material derivatives, Proposition~\ref{prop:stress} shows that taking two advective derivatives of $\check{R}_{n}$ costs $\epsilon_t^{-1}(\Xi e_u^{1/2})$.  When two advective derivatives fall on $\chi_k$, $e_n^{1/2}(t)$ or $\nb \xi_k$, the cost is at most $\tau^{-2}$, and
\ali{
\tau^{-2}\lesssim_{L,\eta}\epsilon_t^{-1}(\Xi e_u^{1/2}). \label{eq:tauTwoNatural}
}
by \eqref{eq:natTimescale}, \eqref{eq:ept}, and the lower bound \eqref{eq:NlowerBd}, which yields the $|\va| = 0$ case of \eqref{eq:SecondAdvecGaI}.  Upon taking $|\va|$ spatial derivatives, the worst case is that a factor of $\hn^{(|\va| + 2 - \unl)_+} \Xi^{|\va|}$ may appear when all spatial derivatives land on $\Dtdt^2 \nb \xi_k$ (see \eqref{eq:nbADtdt2nbxi}), which yields the claimed spatial derivative bound in \eqref{eq:SecondAdvecGaI}.
\end{proof}
The coarser versions with $D_R^{1/2}$ in place of $D_{R,n}^{1/2}$ follow immediately.

Having estimated the phase functions, we apply the Microlocal Lemma \cite[Lemma~4.1]{isettVicol}, which states that convolution operators such as $P_I = K_I \ast$ and $T^\ell P_I = K_I^\ell \ast$ that are localized to frequency $\la$ behave to leading order like multiplication operators with an explicit error.  The kernels here are Schwartz and defined by scaling 
$K_I(h) = \la^d K_{f(I)}(\la h)$ and $K_I^\ell(h) = \la^d K_{f(I)}^\ell(\la h)$ and therefore inherit natural moment bounds $\| |h|^k K_I \|_{L^1} \lesssim_k \la^{-k}$.  We have $\widehat{K}_I(\la \nb \xi_I) = 1$ and $\widehat{K}_I^\ell(\la \nb \xi_I) = m^\ell(\la \nb \xi_I) = m^\ell(\nb \xi_I)$ 
where the last equality uses the degree-zero homogeneity of $m$.  Consequently
\ali{
\Th_I &= g_{[I]}(\mu t) e^{i \la \xi_I} ( \th_I + \de \th_I ) \\
T^\ell \Th_I &= g_{[I]}(\mu t) e^{i \la \xi_I} ( u_I^\ell + \de u_I^\ell ) \\
u_I^\ell &= m^\ell(\nb \xi_I) \th_I
}

We state estimates for $\de \th_I$ and $\de u_I^\ell$.  The upshot of \cite[Lemma~7.5]{isettVicol} is that, for the sake of estimates in $C^k$ spaces, one has the heuristic 
\ali{
\de \th_I \sim \de u_I &\sim \la^{-1} (\nb \th_I + \th_I \nb^2 \xi_I). \label{eq:microHeuristic}
}
(See \cite[Equations (7.21)-(7.22)]{isettVicol} for a more precise formula.)  Hence, for $r = 0,1$, we obtain
\ali{
\co{\nb_{\va} \Dtdt^r \de \th_I} + \co{\nb_{\va} \Dtdt^r \de u_I} &\lesssim \la^{1/2} N^{-1} \hn^{(|\va| + 2 - \unl)_+} \Xi^{|\va|} \tau^{-r} D_{R,n}^{1/2}, \label{eq:deBounds1}
}
The inputs for these bounds are the phase, principal-amplitude, coarse-velocity, and scaled-kernel moment bounds, which are supplied here by Propositions~\ref{prop:phase} and~\ref{prop:amplitudeBounds}, \eqref{eq:tildeuepbd1}, and the preceding kernel estimates.  For comparison to the analogous estimates of \cite[Lemma~7.5]{isettVicol}: the gain of $B_\lambda^{-1}N^{-1}=\Xi/\lambda$ in the source is replaced here by $N^{-1} = \Xi / \la$; the number of derivatives without loss here is $\unl$ in place of $L$; the loss for borrowed derivatives is $\hn = N^{1/L}$ in place of $N^{1/L}$; the cost of an advective derivative is $\Xi \widetilde{e}_u^{1/2} = \tau^{-1}$ in place of the anomalous loss of $\tau^{-1}$ in \cite{isettVicol}; the amplitude of $\th_I = (2\la)^{1/2} \ga_I$ is $\la^{1/2} D_{R,n}^{1/2}$ in place of $e_R^{1/2}$.  
The same heuristic \eqref{eq:microHeuristic} leads to our bounds on the second advective derivative:
\ali{
\co{\nb_{\va} \Dtdt^2 \de \th_I} + \co{\nb_{\va} \Dtdt^2 \de u_I} &\lesssim \la^{1/2} N^{-1} \hn^{(|\va| + 3 - \unl)_+} \Xi^{|\va|}  \ep_t^{-1}(\Xi e_u^{1/2})D_{R,n}^{1/2}. \label{eq:deBounds2}
}
Intuitively, as in \cite[Lemma 7.5]{isettVicol} we take two advective derivatives of \cite[Equations (7.21)-(7.22)]{isettVicol}, which supply more complicated but rigorous versions of \eqref{eq:microHeuristic}, and we apply the the local jet and coarse-velocity bounds \eqref{eq:tildeuepbd1}--\eqref{eq:tildeuepbd3} and Propositions~\ref{prop:phase} and~\ref{prop:amplitudeBounds}.  The cost of $\ep_t^{-1} (\Xi e_u^{1/2})$ ultimately arises from Proposition~\ref{prop:amplitudeBounds} when two advective derivatives hit the $\ga_I$, which dominates all the other terms by inequality \eqref{eq:tauTwoNatural}.  The factor of $\hn^{(|\va| + 3 - \unl)_+} \Xi^{|\va|}$ arises from Proposition~\ref{prop:phase} in the worst case when one sees a term of the form $\nb_{\va} \Dtdt^2 \nb^2 \xi_I$.

\section{Estimating the corrections}

Here we gather estimates for the corrections
\ali{
\Th_I &= g_{[I]}(\mu t) P_I[ e^{i \la \xi_I} (2\la)^{1/2} \ga_I ] \\
T^\ell \Th_I &= g_{[I]}(\mu t) T^\ell P_I[ e^{i \la \xi_I} (2\la)^{1/2} \ga_I ]
}
Since $g_{[I]}$ has size $1$ and both $\| P_I \| \ls 1$ and $\| T^\ell P_I \| \ls 1$, we immediately obtain from \eqref{eq:firstAdvecAmplitudes} that
\ali{
\co{\Th_I} + \co{T^\ell \Th_I} &\ls \la^{1/2}  \co{\ga_I} \ls \la^{1/2} D_R^{1/2}
}
Since $P_I = P_{\approx \la} P_I$ and $T^\ell P_I = P_{\approx \la} T^\ell P_I$ both localize to frequency $\la$, this bound implies 
\ali{
\co{\nb_{\va}\Th_I} + \co{\nb_{\va} T^\ell \Th_I} &\ls_{\va} \la^{|\va|+1/2} D_R^{1/2} \label{eq:bdFornbaThTTh}
}
Writing $|\nb|^{-1/2} P_I = [|\nb|^{-1/2} P_{\approx \la}] P_I$, where $\| |\nb|^{-1/2} P_{\approx \la} \| \ls \la^{-1/2}$, we also obtain
\ali{
\co{\nb_{\va} |\nb|^{-1/2} \Th_I } &\ls_{\va} \la^{|\va|} D_R^{1/2},  \label{eq:nbminhalfTh}
}
which finishes the verification of the claim \eqref{eq:LaWBd}.

Now define the vector field $\ost{u}^\ell = \tilde{u}^\ell + (u^\ell - u_\ep^\ell) + T^\ell \Th$ and the associated advective derivative
\ali{
\ost{D}_t &= \pr_t + \ost{u} \cdot \nb
}
\begin{defn} \label{defn:finalWeight} For $\mrg{D}_t \in \{ \pr_t, \Ddt, \Dtdt, \ost{D}_t \}$ define the final weighted norm
\ali{
\mrg{H}^*[F] &= \max_{0 \leq r \leq 1} \max_{0 \leq |\va| + r \leq L} \fr{\co{\nb_{\va} \mrg{D}_t^r F }}{(N\Xi)^{|\va|} \ost{\tau}^{-r}} 
}
where we recall $\ost{\tau}^{-1} = (N \Xi)^{3/2} D_R^{1/2}$. \end{defn}
\begin{prop} \label{prop:finalComp}
The final weighted norms are comparable up to implicit constants
\ali{
\ost{H}^*[F] \ls H^*[F] \ls \bar{H}^*[F] \ls \tilde{H}^*[F] \ls \ost{H}^*[F] \label{eq:comparable}
}
Furthermore there is a product rule $\mrg{H}^*[F G] \ls \mrg{H}^*[F] \mrg{H}^*[G]$.
\end{prop}
\begin{proof}  Since all the inequalities are proven similarly, we only give the proof of $\ost{H}^*[F] \ls H^*[F]$, which contains all the needed ideas.  We have, for $0 \leq |\va| \leq L - 1$,
\ALI{
&\nb_{\va} \ost{D}_t F = \nb_{\va} \pr_t F + \nb_{\va} ( \tilde{u}^i \nb_i F )   + \nb_{\va}[( u^i - u_\ep^i) \nb_i F] + \nb_{\va}[T^i \Th \nb_i F] \\
&\co{\nb_{\va} \ost{D}_t F}\ls \co{\nb_{\va} \pr_t F} \\
&\qquad \qquad + \overset{\sim}{\sum} ( \co{\nb_{\va_1} \tilde{u}} + \co{\nb_{\va_1} (u - u_\ep) } + \co{\nb_{\va_1} T \Th} )\co{ \nb_{\va_2} \nb F } 1_{|\va_2| < |\va_1|} \\
&\ls (N \Xi)^{|\va|} \ost{\tau}^{-1} H^*[F] + \overset{\sim}{\sum} (N \Xi)^{|\va_1|} \left( e_u^{1/2} + \fr{e_u^{1/2}}{N} + (N \Xi)^{1/2} D_R^{1/2} \right) (N \Xi ) ^{|\va_2| + 1} H^*[F] \\
&\ls (N \Xi)^{|\va|} \ost{\tau}^{-1} H^*[F]
} 
In the last line we used \eqref{eq:tildeuepbd3}, Lemma~\ref{lem:interchangeDerivs}, Lemma~\ref{lem:new} and \eqref{eq:bdFornbaThTTh}.    
\end{proof}
We obtain the following estimate for the new velocity field
\begin{prop} \label{prop:newVelocBound}
\ali{
\ost{H}^*[\ost{u}]   &\ls (N \Xi)^{1/2} D_R^{1/2} \\
\ost{H}^*[\ost{\th}]   &\ls (N \Xi)^{1/2} D_R^{1/2}
}    
\end{prop}
\begin{proof}
We have $\ost{H}^*[\ost{u}] \leq \ost{H}^*[\tilde{u}] +  + \ost{H}^*[u - u_\ep] + \ost{H}^*[T[\Th]]$.  By Proposition~\ref{prop:finalComp} it suffices to bound
\ali{
\co{\nb_{\va} \Ddt^r \ti u_\ep }  &\ls \Xi^{|\va|} (\Xi e_u^{1/2})^r e_u^{1/2}  \\
\co{\nb_{\va} \pr_t^r ( u - u_\ep) } &\ls (N \Xi)^{|\va|} \ost{\tau}^{-r} (e_u^{1/2} / N ) \\
\co{ \nb_{\va} \pr_t^r T[\Th] } &\ls (N \Xi)^{|\va|} \ost{\tau}^{-r} ( N \Xi )^{1/2} D_R^{1/2} \label{eq:biggestAdvecBound}
}
since the right hand side of each of these inequalities is bounded by the right hand side of \eqref{eq:biggestAdvecBound}, which is our goal estimate.  The first of these bounds follows from \eqref{eq:tildeuepbd3}, the second from Lemma~\ref{lem:new}, and the third from the following calculation, which establishes the case $|\va| = 0$:
\ALI{
\pr_t T^\ell[\th] &= \pr_t T^\ell P_I[ g_{[I]}(\mu t) e^{i \la \xi_I} \th_I ] \\
\co{\pr_t T[\th] } &\ls \overset{\sim}{\sum} \mu^{r_1} [\la \co{\pr_t\xi_I}]^{r_2} \co{\pr_t^{r_3} \th_I} \\
&\stackrel{\eqref{eq:transportPhase}}{\ls} \overset{\sim}{\sum} \mu^{r_1} [ \la \co{\tilde{u}} \co{ \nb \xi_I} ]^{r_2} \ost{\tau}^{-r_3} H^*[\th_I] \\
&\ls \overset{\sim}{\sum} \mu^{r_1} [ \la \co{\tilde{u}} \co{ \nb \xi_I} ]^{r_2} \ost{\tau}^{-r_3} H^*[\th_I]  \\
&\stackrel{\eqref{eq:tildeuepbd1},\eqref{eq:phaseGrad},\eqref{eq:comparable}}{\ls} \overset{\sim}{\sum} \mu^{r_1} [ (N \Xi) (\Xi D_u)^{1/2} ]^{r_2}  \ost{\tau}^{-r_3} \tilde{H}^*[\th_I] \\
&\stackrel{\eqref{eq:NlowerBd}}{\ls} \overset{\sim}{\sum} \ost{\tau}^{-r_1} \ost{\tau}^{-r_2}  \ost{\tau}^{-r_3} \tilde{H}^*[\th_I] \\
&\ls \ost{\tau}^{-1} \tilde{H}^*[\th_I] \\
&\stackrel{\eqref{eq:firstAdvecAmplitudes}}{\ls} \ost{\tau}^{-1} \la^{1/2} D_R^{1/2}
}
Our desired bound on $\nb_{\va} \pr_t T[\Th]$ follows from the fact that the operator $T P_I = P_{\approx \la} T P_I$ localizes to frequency $\la$.  

The bounds for $\ost{\th}$ follow the same argument, but are easier as the operator $T$ is not involved.
\end{proof}

\section{The error terms in the convex integration step}

Recall that prior to the convex integration we have
\ali{
\pr_t \th_\Ga + T^j \th_\Ga \nb_j \th_\Ga &= \nb_j \nb_\ell [S_{(\Ga)}^{j\ell} + P_{(\Ga)}^{j\ell} + R_{(\Ga)}^{j\ell}] \\
S_{(\Ga)} &= - \sum_{I=(f,k,n)\in\mathcal I_+} g_{[I]}^2(\mu t) e_n(t) \chi_k^2  \ga_f^2\left( \check{p}_I\right)  B^{j\ell}(\nb \check{\xi}_k)
}
The term $P_{(\Ga)}$ is the ``acceptable'' error from the Newton steps.

When we construct $\ost{\th} = \th_\Ga + \Th$, we get the following error terms
\ali{
\pr_t \ost{\th} + \ost{u}^j \nb_j \ost{\th} &= \nb_j \nb_\ell \ost{R}^{j\ell} \\
\ost{R}^{j\ell} &= R_T^{j\ell} +  R_H^{j\ell} + R_{M}^{j\ell} + R_S^{j\ell} + P_{(\Ga)}^{j\ell} + R_{(\Ga)}^{j\ell} \label{eq:errorList}\\
\nb_j \nb_\ell R_T^{j\ell} &= \pr_t \Th + \tilde{u}_\ep^a \nb_a \Th + T^a \Th \nb_a \tilde{\th}_\ep  \\
\nb_j \nb_\ell R_H^{j\ell} &=  \sum_{I\in\mathcal I} T^a \Th_I \nb_a \Th_I \label{eq:selfInteract}\\
\nb_j \nb_\ell R_M^{j\ell} &= T^j[(\th - \th_\ep)] \nb_j \Th + T^j \Th \nb_j (\th - \th_\ep) \label{eq:convexMollError}\\
\nb_j \nb_\ell R_S^{j\ell} &= \sum_{I\in\mathcal I_+}\left\{T^j \Th_I \nb_j \Th_{\bar{I}} + T^j \Th_{\bar{I}} \nb_j \Th_I - \nb_j \nb_\ell[ g_{[I]}^2 e_n(t) \chi_k^2 \ga_f^2(\check{p}_I) B^{j\ell}(\nb \check{\xi}_I) ]\right\}
}

Note that there are no terms where $\Th_I$ interacts with $\Th_J$ for $J \notin \{ I, \bar{I} \}$.  This is the case thanks to \eqref{eq:nonsimultaneous}.  The fact that self-interaction terms such as \eqref{eq:selfInteract} are well-controlled was first observed in \cite{isettVicol}.

The term $R_S$ is the ``flow error''.  Using the divergence form principle of Section~\ref{sec:divFormPrinciple}, we can write 
\ALI{
T^j \Th_I \nb_j \Th_{\bar{I}} + T^j \Th_{\bar{I}} \nb_j \Th_I  &= \nb_j [ T^j \Th_I \Th_{\bar{I}} + T^j \Th_{\bar{I}} \Th_I] \\
&= \nb_j \nb_\ell K_\la^{j\ell}\ast[\Th_I, \Th_{\bar{I}}]
}
where $K_\la^{j\ell}$ is a specific trace free kernel.
Since only its double divergence is used, replacing $K_\lambda^{j\ell}$ by its symmetric part does not change the preceding identity.

Let $K_{\lambda,\mathrm{sym}}$ denote this symmetric part, constructed as in \cite[Section~4.5.2, equations~(15)--(17)]{isett2021direct} with a fixed common annular cutoff chosen to equal one on the supports of the symbols of $P_I$ and $P_{\bar I}$.  The same scaling gives the moment bound in equation~(19) there.  After factoring out $g_{[I]}^2(\mu t)$, the present pair has the form of the conjugate corrections in \cite[Lemma~4.2]{isett2021direct}, and that lemma gives
\ali{
K_{\lambda,\mathrm{sym}}^{j\ell} \ast [\Th_I, \Th_{\bar{I}}] &=  g_{[I]}^2(\mu t)\left[\widehat{K}_{\lambda,\mathrm{sym}}^{j\ell}(\la \nb \xi_I, -\la \nb \xi_I) |\th_I|^2 + \de B_I^{j\ell}\right]
}
where $\de B_I^{j\ell}$ denotes the symmetrized remainder.  By \cite[equation~(20)]{isett2021direct}, equivalently \eqref{eq:divFormDiagonal}, and the degree-zero homogeneity of the SQG multiplier, we have $\widehat{K}_{\la,\mathrm{sym}}^{j\ell}(\la p, - \la p) = (2\la)^{-1} B^{j\ell}(p)$ whenever $\lambda p$ and $-\lambda p$ lie in the corresponding cutoff plateaux where the symbol of $P_I$ and $P_{\bar{I}}$ are equal to $1$; the preceding phase estimate verifies this for $p=\nabla\xi_I$.  Here 
\ALI{
B^{j\ell}(p) &= -i(\nb^j m^\ell(p) + \nb^\ell m^j(p)),
}
and $m^\ell(p) = i \ep^{\ell a} p_a |p|^{-1}$ is the SQG multiplier.  Since $\theta_I$ is real, the factor $(2\lambda)^{-1}$ is exactly canceled by $|\theta_I|^2=2\lambda e_n\chi_k^2\gamma_f^2(p_I)$.  Thus one representative $I\in\mathcal I_+$ contributes the full prepared tensor, and we arrive at the following expression for the conjugate interactions:
\ali{
T^j \Th_I \nb_j \Th_{\bar{I}} + T^j \Th_{\bar{I}} \nb_j \Th_I &= \nb_j \nb_\ell \left[ g_{[I]}^2(\mu t)\left[ e_n(t) \chi_k^2 \ga_f^2(p_I) B^{j\ell}(\nb \xi_I) + \de B_I^{j\ell} \right] \right],
}
where $\de B_I^{j\ell}$ is the remainder furnished by that lemma.  Proposition~4.5 of \cite{isett2021direct} applies to this same kernel and estimates the remainder.  We can then write the $R_S$ term as
\ali{
R_S^{j\ell} &= \sum_{I\in\mathcal I_+}g_{[I]}^2(\mu t)\left\{e_n(t) \chi_k^2 \left[\ga_f^2(p_I) B^{j\ell}(\nb \xi_I){-} \ga_f^2(\check{p}_I) B^{j\ell}(\nb \check{\xi}_I)\right] + \de B_I^{j\ell}\right\} \\
&= \sum_{I\in\mathcal I_+} (R_{SI}^{j\ell} + g_{[I]}^2(\mu t) \de B_I^{j\ell}).
}
We bound this error using \eqref{eq:flowErrorFirst} and our other estimates for the construction components.

\section{Estimating $R_S$} \label{sec:estimateRS}

We now begin our estimates on the stress errors.  We rely the following propositions:
\begin{prop}[Chain rule] 
Let $\mrg{D}_t \in \{ \pr_t, \Ddt, \Dtdt, \ost{D}_t \}$  and $F$ be smooth.  Let $G$ be a $C^\infty$ function defined on a compact neighborhood of the image of $F$.  Then
\ali{
\mrg{H}^*[G(F)] &\ls (1 + \mrg{H}^*[F])^L
}
\end{prop}
\begin{proof} We compute for $0 \leq r \leq 1$, $0 \leq r + |\va| \leq L$ 
\ALI{
\nb_{\va} \mrg{D}_t^r G(F) &= \sum_{k = 0}^{|\va| + r} \overset{\sim}{\sum} \pr^k G(F) \prod_{i=0}^k \nb_{\va_i} \mrg{D}_t^{r_i} F 
}
where the sum is over appropriate indices such that $\sum_i |\va_i| = |\va|$ and $\sum_i r_i = r$.  Then
\ALI{
\co{\nb_{\va} \mrg{D}_t^r G(F)} &\ls \sum_{k = 0}^{|\va| + r} \co{\pr^k G} \prod_{i=0}^k \co{ \nb_{\va_i} \mrg{D}_t^{r_i} F} \\
&\ls \sum_{k = 0}^{|\va| + r} \prod_{i=0}^k\left[ (N \Xi)^{|\va_i|} \ost{\tau}^{-r_i} \mrg{H}[F] \right] \\
&\ls ( N \Xi )^{|\va|} \ost{\tau}^{-r} (1 +  \mrg{H}[F])^{|\va| + r},
}
which is the desired bound.
\end{proof}

For the following proposition, recall that there exists a ball of radius $K$ about $(0,(1,2),(2,1))$ such that the range of $(R_{(n)}/(M_e D_{R,n}), \nb \check{\xi}_I)$ and also the range of $(\tilde{R}_{(n)}/(M_e D_{R,n}), \nb \xi_I)$ are guaranteed to lie in this ball.

\begin{prop} \label{prop:smallDiffFinal}
    Let $G$ be a $C^\infty$ function defined on the closed ball of radius $K$ about $(0,(1,2),(2,1))$.   Then
\ali{
\bar{H}^*[G(R_{(n)}/(M_e D_{R,n}), \nb \check{\xi}_I) - G(\tilde{R}_{(n)}/(M_e D_{R,n}), \nb \xi_I) ] &\ls (D_u/D_R)^{-1/2} N^{-1/2}  \label{ineq:differenceNext}
}
\end{prop}
\begin{proof}
The $C^0$ bound is given by the mean value theorem and \eqref{eq:RerrorMollPhi} and \eqref{eq:flowErrorFirst}:
\ALI{
\co{ G(R_{(n)}/(M_eD_{R,n}), \nb \check{\xi}_I) &- G(\tilde{R}_{(n)}/(M_eD_{R,n}), \nb \check{\xi}_I) } \\
&\lesssim \|\nabla G\|_{C^0}\left(\frac{\co{R_{(n)} - \tilde{R}_{(n)}}}{M_eD_{R,n}} +  \co{\nb \xi_k - \nb \check{\xi}_k } \right) \\
&\lesssim (D_u/D_R)^{-1/2} N^{-1/2}.
} 
For $1 \leq |\va| + r \leq L$ and $0\leq r\leq1$, we recall the chain rule and definition of the prime weighted norm from Proposition~\ref{prop:primeWeighted} (choosing the inverse timescale $\zeta = \Xi e_u^{1/2}$) to obtain
\ALI{
\co{ \nb_{\va}\Ddt^r &[G(R_{(n)}/(M_eD_{R,n}), \nb \check{\xi}_k)  - G(\tilde{R}_{(n)}/(M_eD_{R,n}), \nb \xi_k) ] } \\
&\lesssim \hn^{(|\va| + r +1 - \unl)_+}\Xi^{|\va|}(\Xi e_u^{1/2})^r \left(\overline{H}'[G(R_{(n)}/(M_eD_{R,n}), \nb \check{\xi}_k)] + \widetilde{H}'[G(\tilde{R}_{(n)}/(M_eD_{R,n}), \nb \xi_k)] \right)\\
&\lesssim \hn^{(|\va| + r +1 - \unl)_+}\Xi^{|\va|}(\Xi e_u^{1/2})^r. \\
&\leq N^{(|\va| + r + 1 - \unl)_+} \Xi^{|\va|+3r/2} D_u^{r/2}
}
In the first inequality we use the Interchange Lemma.  In the second inequality we used that the prime weighted norms of the arguments of $G$ are $\lesssim 1$ by the estimates we have proven for the stress and phase gradients (which indeed is the rationale for the definition of the prime weighted norm).

By definition of the final weighted norm, our goal is to bound this quantity by 
\[(N \Xi)^{|\va|} (N \Xi)^{\fr{3r}{2}} D_R^{r/2}[N^{-1/2}(D_u/D_R)^{-1/2}] \]
In the cases $r = 0$, $|\va| = 1$ as well as $r = 1$, $|\va| = 0$, the desired bound reduces to $N \geq D_u/D_R$, which follows from \eqref{eq:NlowerBd}.    The cases where $|\va| \geq 1$ follow by inducting on $|\va|$; hence we obtain \eqref{ineq:differenceNext}.
\end{proof}

We now estimate $R_{SI}$ with the product rule and Proposition \eqref{prop:smallDiffFinal} to obtain
\begingroup
\ALI{
\ost{H}^*[R_{SI}] &\ls H^*[g_{[I]}^2(\mu t)] H^*[e_n(t)]H^*[\chi_k^2(t)] \cdot \\
&\ost{H}^*[\ga_f^2(p_I) B^{j\ell}(\nb \xi_I) - \ga_f^2(\check p_I) B^{j\ell}(\nb \check{\xi}_I) ] \\
&\ls 1 \cdot D_R \cdot 1 \cdot (D_u/D_R)^{-1/2} N^{-1/2}.
}
Here $H^*[\chi_k^2]\lesssim1$ because differentiating $\chi_k$ costs $\tau^{-1}\leq\ost\tau^{-1}$.  We also used Proposition~\ref{prop:smallDiffFinal} 
\endgroup

\begingroup
We now consider the $\de B_I$ term, which has already been considered and well-estimated in \cite[Proposition~4.5]{isett2021direct}.  
In order to apply that Proposition, we make the following changes in notation: we replace $\la = B_\la N \Xi$ by our $\la = N \Xi$, we replace the coarse scale velocity field by $\tilde{u}_\ep$, its material derivative by $\Dtdt$, the number of ``unborrowed'' derivatives by $\unl$ instead of $L$, the stress level for each $I = (f,k,n)$ by $D_{R,n}$, and the cost of the (first) advective derivative by $\tau^{-1}$.  The hypotheses of \cite[Proposition~4.5]{isett2021direct} are then supplied by our \eqref{eq:tildeuepbd1} and Propositions~\ref{prop:phase} and~\ref{prop:amplitudeBounds}.  Applying that proposition directly gives an estimate of the form    \[ \co{\nb_{\va} \Dtdt^r \de B_I } \lsm_{\va} N^{\fr{3}{2\unl}(|\va| + 2 - \unl)_+} \Xi^{|\va|} \tau^{-r} N^{-1} D_{R,n} \]
for $r=0,1$, which is sharper than we need.  It suffices to observe that, for $0 \leq r \leq 1$, we have
\ali{
\co{\nb_{\va} \Dtdt^r \de B_I} &\lsm_{\va} (N \Xi)^{|\va|} \tau^{-r} N^{-1}  D_{R,n}
\leq (N \Xi)^{|\va|} \tau^{-r} N^{-1} D_R.
}
Differentiating the external profile $g_{[I]}^2(\mu t)$ costs $\mu\leq\ost\tau^{-1}$, so the final weighted norm gives
\ali{
\tilde{H}^*[ g_{[I]}^2(\mu t) \de B_I ]\ls 
\tilde{H}^*[ g_{[I]}^2(\mu t) ]\tilde{H}^*[ \de B_I ] &\ls 1 \cdot N^{-1}D_R.
}
\endgroup
Thus our final bound on the stress error is
\ali{
\ost{H}^*[R_S] &\ls (D_u/D_R)^{-1/2} N^{-1/2} D_R + N^{-1}D_R \notag \\
&\ls (D_u/D_R)^{-1/2} N^{-1/2} D_R,
}
since $N \geq D_u/D_R$.

\section{Nonstationary phase} \label{sec:nonstatPhase}

The transport term and the high-frequency interference terms are both high frequency and our treatment involves nonstationary phase, which is a by now a standard tool in convex integration arguments.  Interestingly, this application of nonstationary phase and the power loss it gives rise to can be avoided (see Section~\ref{sec:avoidNonStat}).

We first introduce a weighted norm.
\begin{defn}  The nonstationary phase weighted norm of $F$ is
\ali{
\tilde{H}_M''[F] &= \max_{r \leq 1} \max_{0 \leq |\va| + r \leq M} \fr{\co{\nb_{\va} \Dtdt^r F}}{(N^{1/2} \Xi)^{|\va|} \ost{\tau}^{-r}}
}
\end{defn}
\begin{lem}  The $\tilde{H}_M''$ norm satisfies the usual product rule and chain rules.  Also, one has
\ali{
\tilde{H}_{M-1}''[\la^{-1} \nb_i F] &\ls N^{-1/2} \tilde{H}_M[F]  \label{eq:nextTermBd}
}
\end{lem}
\begin{proof} We omit the proof of the product and chain rules, since they are almost identical to the proof of Propositions~\ref{prop:chainRule} and~\ref{prop:finalComp}.  As for \eqref{eq:nextTermBd}, the bound on spatial derivatives is immediate from the definition, so we need only bound
\ALI{
\nb_{\va} \Dtdt \nb_i F &= \nb_{\va} \nb_i \Dtdt F + \tisum \nb_{\va_1} \nb_i \tilde{u}_\ep^b \nb_{\va_2} \nb_b F
}
where the sum ranges over $|\va_1| + |\va_2| = |\va| \leq M - 1$.  We bound this sum by
\ALI{
\co{\nb_{\va} \Dtdt \nb_i F} &\ls (N^{\fr{1}{2}} \Xi)^{|\va| + 1} \ost{\tau}^{-1} \tilde{H}''_M[F] + \tisum (N^{\fr{1}{2}}\Xi)^{|\va_1| + |\va_2| + 1} (\Xi e_u^{1/2}) \tilde{H}''_M[F] \\
&\ls (N^{\fr{1}{2}} \Xi)^{|\va| + 1} \ost{\tau}^{-1} \tilde{H}''_M[F]
}
Dividing by $\la \sim N \Xi$ yields the result.
\end{proof}

\begin{prop}[Nonstationary phase]  For any $D>0$ there is a constant $C_D$ so that the following holds.  Whenever $G = e^{i \la \xi_I} g$ has integral $0$ there is a traceless symmetric tensor field $Q^{j\ell}$ that satisfies $\nb_j \nb_\ell Q^{j\ell} = G$ and the bound
\ali{
\tilde{H}^*[Q] &\leq C_D  ((N \Xi)^{-2} + N^{-D/2}) \tilde{H}''_{2D + L}[g] \label{eq:statPhaseBd}
}
\end{prop}
\begin{proof}
Consider the function $q^{j\ell}(p) = A |p|^{-4}p^j p^\ell + B |p|^{-2} \de^{j\ell}$.  Then if $A$ and $B$ solve the equations $A + B = 1$ and $A + d B = 0$, $d =2 $, we have that $q^{j\ell}(p)$ is trace-free and satisfies $p_j p_\ell q^{j\ell}(p) = 1$.

We write $Q^{j\ell} = Q_{(D)}^{j\ell} + \tilde{Q}_{(D)}^{j\ell}$ where 
\ali{
Q_{(D)}^{j\ell} &= \la^{-2} \sum_{k = 0}^D e^{i \la \xi_I} q_{(k)}^{j\ell} \label{eq:parametrix} \\
\tilde{Q}_{(D)}^{j\ell} &= \calR^{j\ell}[G - \nb_j \nb_\ell Q_{(D)}^{j\ell}]
}
We define the $q_{(k)}^{j\ell}$ recursively by
\ali{
g_{(0)} &= g, \quad \qquad q_{(k)}^{j\ell} = q^{j\ell}(\nb \xi_I) g_{(k)} \\
g_{(k+1)} &= - \la^{-1} [ \nb_j \xi_I  \nb_\ell q_{(k)}^{j\ell} + \nb_\ell \xi_I  \nb_j q_{(k)}^{j\ell}]] - \la^{-2} \nb_j \nb_\ell q_{(k)}^{j\ell} \label{eq:gkplus1}
}
These inductive rules are defined so that
\ali{
\nb_j \nb_\ell Q_{(D)} - e^{i \la \xi_I} g &= e^{i \la \xi_I} g_{(D+1)}^{j\ell} \label{eq:telescoping}
}

We claim the following estimates inductively on $k$.
\ali{
\tilde{H}''_{2D + L - 2 k} [g_{(k)}] &\ls N^{-k/2}\tilde{H}_{2 D + L}''[g] \label{eq:inductgk} \\
\tilde{H}''_{2D + L - 2 k} [q_{(k)}] &\ls N^{-k/2}\tilde{H}_{2 D + L}''[g] \label{eq:inductqk}
}
Indeed \eqref{eq:inductgk} holds for $k = 0$ trivially.  Then \eqref{eq:inductqk} holds for $k$ by
\ALI{
\tilde{H}''_{2D + L - 2 k} [q_{(k)}] &\ls \tilde{H}''_{2D + L - 2 k} [q^{j\ell}(\nb \xi_I)] \tilde{H}''_{2D + L - 2 k} [g_{(k)}] \\
&\ls (1 + \tilde{H}''_{2D + L - 2 k}[\nb \xi_I])^{2D} \tilde{H}''_{2D + L} [g] \\
&\ls 1 \cdot \tilde{H}''_{2D + L} [g]
}
where we applied the product rule and chain rule for the weighted norm and the inductive hypothesis for $g_{(k)}$.  Now we estimate \eqref{eq:gkplus1} by the product rule, \eqref{eq:inductqk} for $k$, and \eqref{eq:nextTermBd}
\ALI{
\tilde{H}''_{2D + L - 2 (k + 1)}[ g_{(k+1)}] &\ls \tilde{H}''_{2D + L - 2 (k + 1)}[ \nb \xi_I] \tilde{H}''_{2D + L - 2 (k + 1)}[ \la^{-1}\nb q_{(k)}] \\
&+ \la^{-2} \tilde{H}''_{2D + L - 2 (k + 1)}[ \nb \nb q_{(k)}]  \\
&\ls 1 \cdot \tilde{H}''_{2D + L - 2 k - 1}[ \la^{-1}\nb q_{(k)}] + \la^{-1} \tilde{H}''_{2D + L - 2 k - 1}[ \nb q_{(k)}] \\
&\ls N^{-1/2} \tilde{H}''_{2D + L - 2 k}[ \la^{-1}\nb q_{(k)}] \ls N^{-(k+1)/2}\tilde{H}''_{2D + L} [g],
}
which concludes the induction.

We can now estimate the $Q_{(D)}$ defined in \eqref{eq:parametrix} by first observing
\ali{
\ost{H}^*[e^{i \la \xi_I}] &\ls 1 \label{eq:sizeOfPhase} \\
\ost{H}^*[F] &\ls \tilde{H}''_{L}[F] \label{eq:defnWeightedNorm}
}
We postpone the proof of \eqref{eq:sizeOfPhase}.  The second of these bounds is directly from the definition.  We now estimate
\ALI{
\ost{H}^*[Q_{(D)}] &\ls \la^{-2} \sum_{k=0}^D \ost{H}^*[e^{i \la \xi_I}] \ost{H}^*[q_{(k)}] \\
&\stackrel{\eqref{eq:sizeOfPhase}}{\ls} \la^{-2} \sum_{k=0}^D \tilde{H}''_{L}[q_{(k)}] \\
&\ls \la^{-2} \sum_{k=0}^D \tilde{H}''_{2D +L - 2k}[q_{(k)}] \ls \la^{-2} \tilde{H}_{2D+L}''[g].
}
This bound suffices to prove \eqref{eq:statPhaseBd} for the parametrix.  

To bound the error, we need a trivial bound for the operator $\calR^{j\ell}$.  Specifically
\ali{
\tilde{H}^*[\calR^{j\ell}[F]] &\ls \tilde{H}^*[F], \label{eq:trivBound}
}
and we also will use
\ali{
\tilde{H}^*[e^{i \la \xi_I} ] &\ls 1. \label{eq:phaseBd1}
}
Taking these two estimates as given, we now have
\ALI{
\tilde{H}^*[\tilde{Q}_{(D)}] = \tilde{H}^*[ \calR^{j\ell}[e^{i \la \xi_I} g_{(D+1)} ]] &\ls \tilde{H}^*[ e^{i \la \xi_I} g_{(D+1)} ] \\
&\ls \tilde{H}^*[ e^{i \la \xi_I} ] \tilde{H}^*[ g_{(D+1)} ] \\
&\stackrel{\eqref{eq:inductgk}}{\ls} 1 \cdot N^{-D/2}  \tilde{H}_{2D+L}''[g],
}
which completes the proof subject to \eqref{eq:trivBound} and \eqref{eq:phaseBd1}.
\end{proof}
\begin{proof}[Proof of \eqref{eq:trivBound} and \eqref{eq:phaseBd1}]    To prove \eqref{eq:phaseBd1}, we observe that $\Dtdt e^{i \la \xi_I} = 0$, so it suffices to bound spatial derivatives.  By the chain rule and product rule we obtain
\ALI{
\co{\nb_{\va} e^{i\la \xi_I}} &\ls \sum_{m=0}^L \underset{\va_j}{\tisum} \co{e^{i \la \xi_I}} \la^m \prod_{j=1}^m \co{\nb_{\va_j} \xi_I} \\
&\ls \sum_{m=0}^L \tisum \la^m \prod_{j=1}^m \hxi^{|\va_j| - 1} \\
\co{\nb_{\va} e^{i\la \xi_I}} &\ls \sum_{m=0}^L \tisum \prod_{j=1}^m \la^{|\va_j|} \ls \la^{|\va|},
}
where we have used that derivatives of $\xi$ cost at most $\hxi$, which is smaller than $\la$.

To prove \eqref{eq:trivBound}, we first note that $\calR^{j\ell}$ is bounded on $C^0(\T^2)$.  For example,
\[ \|\calR^{j\ell}\| \leq \sum_q \| P_q \calR^{j\ell} \| \ls \sum_{q=0}^\infty 2^{-2q} \]
Then $H^*[\calR^{j\ell} F] \ls H^*[F]$ follows from the fact that $\calR$ commutes with $\nb_{\va}$ and $\pr_t$.  By comparability of weighted norms, this estimate suffices.
\end{proof}

\subsection{High frequency error terms}  \label{sec:highFreqErrors}
We now apply the nonstationary phase estimate to the high frequency error terms.

We start with $R_H$.  There is an important cancellation in this term that was first observed in \cite{isettVicol}.  Namely, since $u_I^a \nb_a \xi_I = \th_I m^a(\nb \xi_I) \nb_a \xi_I = 0$, we have
\ALI{
\sum_I T^a \Th_I \nb_a \Th_I &=  \sum_I \la g_{[I]}^2(\mu t) e^{2 i \la \xi_I } \de u_I^a (i  \nb_a \xi_I) (\th_I + \de \th_I ) \\
&+ \sum_I g_{[I]}^2(\mu t) e^{2 i \la \xi_I} (u_I^a + \de u_I^a) (\nb_a \th_I + \nb_a \de \th_I ) 
}

For the next computation, let $\tilde{H}''$ be a shorthand for $\tilde{H}''_{2 D + L}$.  By nonstationary phase, for any $D \geq 0$ there exists a traceless second-order anti-divergence that obeys the estimate
\ALI{
\tilde{H}^*[R_H^{j\ell}] &\ls (\la^{-2} + N^{-D/2})( A + B) \\
A &= \la\tilde{H}''[g_{[I]}^2(\mu t)] \tilde{H}''[\de u_I ]  \tilde{H}''[\nb \xi_I ] (\tilde{H}''[\th_I ] + \tilde{H}''[\de \th_I ]) \\
B&= \tilde{H}''[g_{[I]}^2(\mu t)] \tilde{H}''[u_I + \de u_I] \tilde{H}''[\nb \th_I + \nb \de \th_I] 
}
For all these terms inside the weighted norms, we claim that the bounds for the $H''$ norm of each term are the same as the bound we have stated for the $C^0$ norm of each term. Indeed, for each of these terms, a spatial derivative costs at most $\hxi = N^{1/L} \Xi$, which is smaller than $N^{1/2} \Xi$, while an advective derivative costs at most $\mu = \Xi^{3/2} N^{1/2} D_R^{1/2}$, which is smaller than $\ost{\tau}^{-1} = (N \Xi)^{3/2}D_R^{1/2}$.  Combining \eqref{eq:deBounds1}, Proposition~\ref{prop:phase}, Proposition~\ref{prop:amplitudeBounds}, and the following estimate
\ALI{
\tilde{H}''[u_I^\ell] &= \tilde{H}''[m^\ell(\nb \xi_I)] \tilde{H}''[\th_I] \ls \tilde{H}''[\th_I] \ls \la^{1/2} D_R^{1/2}
}
yields (recall $\la = N \Xi$)
\ALI{
A &\ls ( N \Xi ) \cdot 1 \cdot \fr{\la^{1/2} D_R^{1/2}}{N}  1 (\la^{1/2} D_R^{1/2}) \\
B &\ls 1 \cdot \la^{1/2} D_R^{1/2} (\Xi \la^{1/2} D_R^{1/2}) 
}
Hence we conclude,
\ali{
\tilde{H}^*[R_H^{j\ell}] &\ls (\la^{-2} + N^{-D/2})( \la \Xi D_R)
}
Recall that $N \geq N^{4\eta} \Xi^{4 \eta} \sim \la^{4\eta}$.  Choosing $D$ large depending on $\eta$, we have $\tilde{H}^*[R_H^{j\ell}] \ls \fr{D_R}{N}$.

The other high frequency term is the transport term.  Since the advective derivative annihilates the phase function, we have\ALI{
\nb_j \nb_\ell R_T^{j\ell}
&= \sum_I \mu g_{[I]}'(\mu t)e^{i\la\xi_I}(\th_I+\de\th_I) \\
&\quad + \sum_I g_{[I]}(\mu t)e^{i\la\xi_I}
\left[\Dtdt\th_I+\Dtdt\de\th_I
 +(u_I^a+\de u_I^a)\nb_a\tilde\th_\ep\right].
}
Nonstationary phase with the same choice of $D$ as before yields a solution of weighted norm
\ALI{
\tilde{H}^*[R_T]
&\lesssim (\la^{-2}+N^{-D/2})\Big\{
\mu H''[g_{[I]}'(\mu t)]H''[\th_I+\de\th_I] \\
&\qquad +H''[g_{[I]}(\mu t)]\big(
H''[\Dtdt\th_I]+H''[\Dtdt\de\th_I]
+H''[u_I+\de u_I]H''[\nb\tilde\th_\ep]\big)\Big\}.
}
For these terms it is again true that the $H''$ weighted norm is the same size as the bound on the $C^0$ norm modulo constants, since spatial derivatives cost at most $N^{1/L} \Xi < N^{1/2} \Xi$, while the additional material derivative costs at most
\[
\epsilon_t^{-1}=(\Xi e_u^{1/2})(D_u/D_R)^{1/2}N^{1/2}\leq\ost\tau^{-1}=(N\Xi)^{3/2}D_R^{1/2}.
\]
The first material derivative of $\theta_I$ retains the cutoff-driven cost $\tau^{-1}$.  Combining \eqref{eq:deBounds1}-\eqref{eq:deBounds2}, Proposition~\ref{prop:phase}, and Proposition~\ref{prop:amplitudeBounds}, we therefore obtain
\ALI{
\tilde{H}^*[R_T]
&\lesssim (\la^{-2}+N^{-D/2})\Big[
\mu\la^{1/2}D_R^{1/2}
+\tau^{-1}\la^{1/2}D_R^{1/2}
+\la^{1/2}D_R^{1/2}\,\Xi e_u^{1/2}\Big] \\
&\lesssim (\la^{-2}+N^{-D/2})\mu\la^{1/2}D_R^{1/2}
\lesssim N^{-1}D_R.
}
Both the estimate for $R_H$ and the estimate for $R_T$ are satisfactory for the Main Lemma, since we have $N^{-1} D_R \leq (D_u/D_R)^{-1/2} N^{-1/2} D_R$.

\subsection{How to avoid nonstationary phase} \label{sec:avoidNonStat}

We include this section to note that one can avoid nonstationary phase in the proof in a way such that the only source of double exponential frequency growth occurs during the Newton step.  To do so, let $\tilde{u}^\ell = u_\ep^\ell + T^\ell w$, where $w$ is the Newton correction.  Instead of transporting the phase functions by the flow of $\tilde{u}^\ell$ as we have done, first apply a frequency truncation $P_{\leq q_\ep}$ to $T^\ell w$ and transport the phase functions by the resulting frequency localized vector field. 

With such a frequency localization, both the high frequency interference terms and the transport term now live at frequency $\sim \la$, and one can simply apply an operator $\RR^{j\ell}$ to both those terms to find a suitable anti-divergence that gains a smallness of $\la^{-2}$.  This technique avoids nonstationary phase (which was also avoided in \cite{isettVicol,buckShkVicSQG,isett2021direct}), and also avoids the power loss in frequency incurred during the nonstationary phase, which is important for deriving an endpoint type result \cite{isett2017endpoint}.  However, it comes with two complications.
\begin{enumerate}
    \item  The estimates of the convex integration step are a bit different in terms of powers of $N$ although the final bounds for $\ost{R}$ are the same.
    \item  One has to handle an error term of the form 
\ALI{
(T^\ell w - P_{\leq q_\ep} T^\ell w) \Th + T^\ell \Th ( w - P_{\leq q_\ep} w)
}
\end{enumerate}
The latter term can be treated similarly to the mollification term addressed below.

\subsection{The mollification error} \label{sec:cvxMollErr}

The term $R_M$ is also a new error term compared to \cite{buckShkVicSQG,isett2021direct}.  In those works there was no need to regularize $\th$ since it could be enforced that $\th$ had compact frequency support.  In other words, we had $\th = \th_\ep$.  Here $\th$ does not have compact frequency support, so we have to bound this term, which resembles the term \eqref{eq:notDivForm10}.  Again we use our simplified version of the observation in \cite{isett2024conservation} showing how to write the nonlinearity in a divergence form.

We are estimating a solution to
\begin{align}
\nb_\ell R_M^{j\ell} &= [T^j\th - T^j \th_\ep] \Th + [T^j \Th] (\th - \th_\ep)
\label{eq:notDivForm1Again}
\end{align}
We know $\supp \hat \Theta \subset \{ \xi ~:~ \la/10^2 \le |\xi| \le 10^2\la\}$. 
We decompose $\nb_j R_M^{j\ell}$ into the sum of three kinds of terms (HH, HL and LH) (or really five kinds of terms, but we can group them into three kinds).
\begin{align}
\begin{split}
\nb_j R^{j\ell}_M &= \sum_q P_q T^\ell(\th-\th_\epsilon) P_{q+1}\Th + T^\ell P_{q+1}\Th P_q(\th-\th_\epsilon) + \text{ Similar}\\
&+ \sum_q P_{\le q-1} T^\ell (\th-\th_\epsilon) P_{q+1}\Theta \\
&+ \sum_q P_{\le q-1} (\th-\th_\epsilon) T^\ell P_{q+1} \Theta\\
&+ \sum_q P_{q+1}(\th-\th_\epsilon) T^\ell P_{\le q-1} \Theta\\
&+ \sum_q P_{q+1} T^\ell(\th-\th_\epsilon) P_{\le q-1} \Theta
\end{split}
\label{eq:notDivForm2Again}
\end{align}
(Recall that $T^\ell$ and $P_k$ commute for any $k$, and same for $P_{\le k-1}$)

Define $q^\lambda \in \N$ by $q^\lambda :\approx \log_2(\lambda)$. We have that the Fourier support of $\hat\Th$ is essentially in a single dyadic shell (or a bounded number of shells). By consideration of frequency support, the five sums can be simplified as sums over, respectively,
$$q \sim q^\lambda, q \sim q^\lambda, q \sim q^\lambda, q \ge q^\lambda, q \ge q^\lambda.$$

To each of these we associate an antidivergence, and we just have to bound that antidivergence. We will only do this for three cases (one being the High-High), which are representative of all five. 

Looking at the very first term in \eqref{eq:notDivForm2Again}, we use the divergence form principle of Section~\ref{sec:divFormPrinciple} to define
\begin{align}
	R_{MHq}^{j\ell} &= K_q \ast [\th-\th_\ep,\Th] \\
	&= K_q \ast [P_{\approx q^\lambda}\th, \Th] \\
	&= \int (\th(x-h_1) - \th_\ep(x-h_1)) \Th(x-h_2) K_q^{j\ell}(h_1,h_2) dh_1 dh_2
\end{align}
For all $0 \leq r + |\va| \leq L$ we bound
\begin{align}
\co{\nb_{\va} \pr_t^r R_{MHq}} &\ls \sum_{q \sim q^\lambda} \|K_q\|_{L^1} \|\nb_{\va_1} \pr_t^{r_1}(\th - \th_\ep)\|_0 \|\nb_{\va_1} \pr_t^{r_1} \Th\|_0 \\
	&\ls \sum_{q \sim q^\lambda} 2^{-q} [(N \Xi)^{|\va_1|} \ost{\tau}^{-r_1} \fr{e_u^{1/2}}{N}] [
(N \Xi)^{|\va_2|} \ost{\tau}^{-r_2} H^*[\Th]] \\
	&\ls (N \Xi)^{-1} (N \Xi)^{|\va|} \ost{\tau}^{-r} \fr{e_u^{1/2}}{N} (N \Xi)^{1/2} D_R^{1/2}. \\
 &\ls (N \Xi)^{|\va|} \ost{\tau}^{-r} \fr{D_R}{N},
\end{align}
where in the last line we used $N \geq D_u/D_R$, while in the second line we used Lemma~\ref{lem:new} and the bound on $\Th$ obtained in the proof of Proposition~\ref{prop:newVelocBound}.

Our next representative term is 
\ali{
R_{MHL}^{j\ell}
 &= \sum_{q \sim q_\la} \calR_a^{j\ell}[ P_{q+1}(\th - \th_\ep) P_{\leq q-1} T^a \Th ] \\
 &= \sum_{q \sim q_\la}\calR_a^{j\ell}P_{\approx q} (\th - \th_\ep) P_{\approx \la} T^a \Th \label{eq:MollHL}
}
where the representation in the second line is due to the Fourier support of $\Th$ being in $|\xi| \sim \la$.  We bound this term by
\ali{
\co{\nb_{\va} \pr_t^r \eqref{eq:MollHL}} &\ls \sum_{q \sim q^\la} \tisum \| \calR_a^{j\ell}P_{\approx q} \| [(N\Xi)^{|\va_1|} \ost{\tau}^{-r_1} \fr{e_u^{1/2}}{N}] [(N\Xi)^{|\va_2|} \ost{\tau}^{-r_2} H^*[\Th]]  \\
&\ls \sum_{q \geq q_\la} 2^{-q} (N \Xi)^{|\va|} \ost{\tau}^{-r}  \fr{e_u^{1/2}}{N} (N \Xi)^{1/2} D_R^{1/2} \\
&\ls (N \Xi)^{|\va|} \ost{\tau}^{-r} \fr{D_R}{N}.
}
Here in the second line we used the trivial observation that $H[T^\ell P_q \Th ] \ls H[\Th]$ by the fact that $T^\ell P_q$ is bounded on $C^0$ and commutes with spatial derivatives and $\pr_t$.  In the last line we again used $N \geq D_u/D_R$.

The last of the three representative terms is
\ali{
\sum_{q\sim a^\lambda} P_{\le q-1} T^\ell(\th-\th_\ep) P_{q+1}\Th.
}
We write
\ali{
R_{MLH}^{j\ell} &= \sum_{q\sim q^\lambda} \RR_a^{j\ell} P_{\approx q} [  P_{\leq q-1} T^a(\th-\th_\ep) P_{q+1} \Th ] \label{eq:hardMollTerm}\\
\co{\nb_{\va} \pr_t^r \eqref{eq:hardMollTerm}} &\ls  \sum_{q\sim q^\lambda} \la^{-1} \tisum \co{ \nb_{\va_1} \pr_t^{r_1} (T \th - T \th_\ep) } \co{ \nb_{\va_2} \pr_t^{r_2} \Th } \\
&\ls \la^{-1}\tisum  [(N \Xi)^{|\va_1|} \ost{\tau}^{-r_1} \fr{e_u^{1/2}}{N}] [(N \Xi)^{|\va_2|} \ost{\tau}^{-r_2} H[\Th]]  \\
&\ls (N \Xi)^{-1} (N \Xi)^{|\va|} \ost{\tau}^{-r} \fr{e_u^{1/2}}{N} (N \Xi)^{1/2} D_R^{1/2} \\
&\ls (N \Xi)^{|\va|} \ost{\tau}^{-r} \fr{D_R}{N}
}
Here again we used Lemma~\ref{lem:new}, the bound on $\Th$ obtained in the proof of Proposition~\ref{prop:newVelocBound}, and $N \geq D_u/D_R$.

Combining these estimates we have
\ali{
H^*[R_M] &\ls \fr{D_R}{N}.
}

\section{The Main Lemma implies the Main Theorem} \label{sec:mainLemImpThm}

We start with the following auxiliary theorem, which is enough to prove regularity of solutions but in itself is not enough to prove nontriviality.  Nontriviality will be a corollary of the h-principle.

For any $\de > 0$  given, we choose $L \geq 7$ and $\eta > 0$ depending on $\de$ so that the parameter
\ali{
\varep = \fr{6}{L} + 4 \eta \leq \de^3.
}
Let $\hc$ be the constant in the Main Lemma associated to this choice of $\eta, L$.

\begin{thm} \label{thm:approxThem}  Let $\de > 0$ be given.  There is a constant $C_\de$ depending on $\de$ and an integer $L$ such that the following holds.  Let $(\th_0, R_0)$ be an SQG-Reynolds flow with frequency energy levels of order $L$ bounded by $(\Xi_0, D_{u,0}, D_{R,0})$ and with $\suppt R_0$ compactly contained in a nonempty set $J_0$.  Then there exists a solution $\th$ to SQG of class $|\nb|^{-1/2} \th \in C^{1/2 - 2 \de}$
such that $\suppt \th - \th_0$ is contained in a $C_\de \hc (\Xi_0 e_{u,0}^{1/2})^{-1}$ neighborhood of $J_0$ and such that
\ali{
\co{~ |\nb|^{-1/2}(\th - \th_0)} &\leq C_\de D_{R,0}^{1/2}
}
\end{thm}

\begin{proof}
We define a sequence of SQG-Reynolds flows $(\th_n, R_n)$ by iteration of the Main Lemma.  We set $(\Xi, D_u, D_R)_{(0)} = (\Xi_{0}, D_{u,0}, D_{R,0})$ and evolve according to the parameter rules
\ALI{
\Xi_{(k+1)} &= \hc N_{(k)} \Xi_{(k)} \\
D_{u,(k+1)} &= D_{R,(k)} \\
D_{R,(k+1)} &= \fr{D_{R,(k)}^{1+\de}}{Z}.
}
where $Z$ is to be chosen depending on $\hc$ and on the initial frequency energy levels.  

\begin{rem} 
We note that, following \cite{isett}, the above parameter evolution rules are essentially the simplest formulas that guarantee a slight double exponential growth of frequencies, while allowing an arbitrarily large stress at the initial stage of the iteration, which is essential to Theorem~\ref{thm:approxThem}. \end{rem}

The above parameter evolution rules will imply a double exponential decay of $D_{R,(k)}$, but for the moment we impose that $Z \geq \max \{ 2 D_{R,(0)}^{\de}, D_{R,(0)}^{1+\de} \}$ in order to ensure that $D_{R,(1)} \leq 1$ and that
\ali{
D_{R,(k+1)} &\leq \fr{1}{2} D_{R,(k)} \label{eq:geometricDecay}
}
for all $k$.

Our choice of $N_{(k)}$ is dictated by the estimate in the Main Lemma:
\ali{
D_{R,(k+1)} &= {\hc} \left(\fr{D_u}{D_R}\right)_{(k)}^{-1/2} N^{-1/2}_{(k)} D_{R,(k)}  \\
\Rightarrow N_{(k)} &= \hc^2 Z^2 (D_u/D_R)_{(k)}^{-1} D_{R(k)}^{-2\de}. \label{eq:NkFormula}
}

The most delicate task in this framework is to check that $N_{(k)}$ is admissible, since this condition barely holds (see Remark~\ref{rem:barelyNotC0}); namely, we need
\ali{
N_{(k)} \geq \hc N_{(k)}^{\fr{6}{L} + 4 \eta} \Xi_{(k)}^{4 \eta} (D_u/D_R)_{(k)}.
}
With $\varep = \fr{6}{L} + 4 \eta$ defined as above and substituting \eqref{eq:NkFormula}, it is enough to check that
\ali{
Z^{-2+2\varep} \Xi_{(k)}^{\varep} (D_u/D_R)_{(k)}^{2-\varep} D_{R,(k)}^{2\de(1-\varep)} &\leq 1 \label{eq:ourGoal1}
}
Since the power of $Z$ is negative, it is clear that $Z$ can be chosen large enough so that this inequality holds at $k = 0$.  Now suppose $k \geq 1$.  In this case one has
\ali{
(D_u / D_R)_{(k)} = (D_{R,(k-1)}/D_{R,(k)}) &= Z^{\fr{1}{1+\de}} D_{R,(k)}^{-\de/(1+\de)} 
}
Taking logs of \eqref{eq:ourGoal1}, we need to check:
\ali{
(-2 + 2 \varep +\fr{2 - \varep}{1+\de}) \log Z + \varep \log \Xi_{(k)} + (2 \de (1 - \varep) - \fr{\de}{1+\de} (2-\varep)) \log D_{R,(k)} &\leq 0
}
We prove this inequality by induction, as we have already considered the case $k = 0$.  Since $Z$ is a constant, we need only check that
\ali{
\varep \de_{(k)} \log \Xi_{(k)} + \de \left(2 (1 - \varep) - \fr{1}{1+\de} (2-\varep)\right) \de_{(k)} \log D_{R,(k)} \label{eq:deKbound}
}
is negative, where $\de_{(k)} f = f_{(k+1)} - f_{(k)}$ denotes the discrete difference operator.  Note that since $\de_{(k)} \log \Xi_{(k)} = \log(\Xi_{(k+1)}/\Xi_{(k)}) = \log(\hc N_{(k)})$ and $(D_u/D_R)_{(k)}^{-1} \leq 1 \leq Z$, we have
\ali{
\de_{(k)} \log \Xi_{(k)} = \log(\hc^3 Z^2 (D_u/D_R)_{(k)}^{-1} D_{R(k)}^{-2\de}) \leq \log(\hc^3 Z^3 D_{R,(k)}^{-2\de})\label{eq:XiChangeIneq}
 }
We now use that $\varep = O(\de^3)$, $D_{R,(k)} \leq 1$ for all $k \geq 1$, and we recall $D_{R,(k+1)} = (D_{R,(k)}^{\de}/Z)D_{R,(k)}$ to obtain
\ALI{
\eqref{eq:deKbound} &\leq \varep \de_{(k)} \log \Xi_{(k)} + (2\de^2 + O(\de^3) + O(\varep)) \de_{(k)} \log D_{R,(k)} \\
&\stackrel{\eqref{eq:XiChangeIneq}}{\leq} 3 \varep (\log \hc + \log Z) - 2 \de \varep \log D_{R,(k)} + (2\de^2 + O(\de^3))( \de \log D_{R,(k)} - \log Z ) \\
&\leq O(\de^3) \log \hc + (-2\de^2 + O(\de^3)) \log Z \qquad \qquad \mbox{ since } \log D_{R,(k)} \leq 0.
}
We now choose $Z$ depending on $\de$ and $\hc$ so that the right hand side is negative as desired.  Thus our choice of $N_{(k)}$ is admissible for all $k$.

Applying \eqref{eq:geometricDecay}, our solution $\th$ obeys
\ali{
\co{ |\nb|^{-1/2} (\th - \th_0)} &\ls \sum_{k=0}^\infty \co{ |\nb|^{-1/2} W_{(k)}} \\
&\ls \sum_{k=0}^\infty D_{R,(k)}^{1/2} \leq C_\de D_{R,(0)}^{1/2}.
}
Note that the convergence of this series combined with boundedness of the nonlinearity in $L_t^2 \dot{H}^{-1/2}$ also shows that $\th$ is a weak solution to SQG.  A similar geometric series bounds the size of the increase in time support of $\th_{(k)} - \th_0$ by
\ali{
\sum_k (\Xi_{(k)} e_{u,(k)}^{1/2})^{-1} &\ls (\Xi_0 e_{u,0}^{1/2})^{-1}
}
hence the time support of $\th - \th_0$ is bounded as claimed.

To check the regularity of the solution, we follow the method of \cite{isett}.  It will be convenient to phrase the parameter evolution rules in terms of logs
\ali{
\psi_{(k)} &\equiv [\log D_{R(k)}, \log(D_u/D_R)_{(k)}, \log \Xi_{(k)}]^t \\
\psi_{(k+1)} &= \begin{bmatrix}
    - \log Z   \\
     \log Z   \\ 
     \log (\hc^3 Z^2)
\end{bmatrix} + \begin{pmatrix}
1 + \de   & 0 & 0 \\
-\de & 0 & 0 \\
-2\de & -1 & 1
\end{pmatrix} \psi_{(k)} \label{eq:freqEnEvolveMatrix}
}
We call the  $3 \times 3$ {\bf  parameter evolution matrix} appearing here $T_\de$.

We seek a vector in the null space of $T_\de - (1 + \de)$ with a negative first coordinate, since the largest component of $\psi_{(k)}$ will be in that direction.  An example is given by
\ali{
\psi_+ = \begin{bmatrix}
-(1+\de) \\
\de \\
1 + 2 \de
\end{bmatrix} \label{eq:psiPlus}
}
In terms of eigenvectors $(\psi_+, \psi_0, \psi_1)$ for the $(1+\de, 0, 1)$ eigenspaces respectively, we can decompose
\ali{
\psi_{(k)} &= c_{+,(k)} \psi_+ + c_{0,(k)} \psi_0 + c_{1,(k)} \psi_1 \label{eq:diagonalize}
}
where the coefficients evolve by diagonalizing \eqref{eq:freqEnEvolveMatrix}
\ALI{
c_{+(k+1)} &= (1+\de) c_{+(k)} + \ell_+ \\
c_{1(k+1)} &= 1 c_{1,(k)} + \ell_1 \\
c_{0(k+1)} &= 0 c_{0(k)} + \ell_0
}
for some constants $\ell_+, \ell_1$, and $\ell_0$.  

As we are essentially iteratively applying a fixed matrix, one expects the contribution from the $(1+\de)$ eigenspace to be the largest, and this is indeed the case.  It is easy to check by induction that $c_{1(k)}$ and $c_{0(k)}$ are both bounded by $O(k)$.  Meanwhile $c_{+(k)}$ yields the main term.  In particular,
\ali{c_{+(k)} \geq c (1 +\de)^k \label{eq:mainCoeff}}
for some $c > 0$.  To see this it suffices to note that $\ell_+ > 0$, which can be checked by applying the $(1+\de)$ row eigenvector $[1,0,0]$ (which annihilates $\psi_0$ and $\psi_1$) to \eqref{eq:freqEnEvolveMatrix} and substituting \eqref{eq:diagonalize}.

We now compute the regularity of our solution.  Using the interpolation inequality $\|f \|_{C^\a} \ls \| f\|_{C^0}^{1-\a} \| \nb f \|_{C^0}^\a$, the estimate \eqref{eq:LaWBd} on $W$, the formula \eqref{eq:NkFormula} for $N_{(k)}$ and $\a < 1$, one has
\ali{
\log \| |\nb|^{-1/2} W \|_{C^\a} &\leq \log \hc + \a \log (N_{(k)} \Xi_{(k)}) + \fr{1}{2} \log D_{R,(k)} \\
&\leq \log(\hc^4 Z^2) + \left[ \fr{1}{2} - 2 \a \de, -\a, \a\right] \psi_{(k)},
}
where the last line refers to the linear pairing of the row vector with the column vector $\psi_{(k)}$.  From \eqref{eq:diagonalize} and \eqref{eq:mainCoeff}, we see that the right hand side goes to $-\infty$ exponentially fast whenever the same row vector applied to $\psi_+$ in \eqref{eq:psiPlus} gives a negative value.  In conclusion, $|\nb|^{-1/2} \th \in L_t^\infty C^\a$ whenever 
\ALI{
\a < \fr{1}{2} \left( \fr{1 + \de}{1 + 3 \de + 2 \de^2} \right).
}
Using linearization, one sees that $\a = 1/2 - 2 \de$ satisfies this inequality for $\de$ sufficiently small, hence Theorem~\ref{thm:approxThem} is proven.
\end{proof}

\subsection{Nonuniqueness for the initial value problem} \label{sec:nonuniquenessInitValue}

We now prove Theorem~\ref{thm:cpctSupp} following the argument in \cite{isett}.  Let $c > 0$ be given.  Let $\th_0 : \T^2 \to \R$ be $C^\infty$ with zero mean and choose $T > 0$ such that the smooth classical solution to SQG, denoted by $f$, is defined on an open interval that compactly contains $[-T,T]$.  Taking $c$ smaller if needed, assume $[-(T + 2c), T+2c]$ is also compactly contained in the existence interval for the smooth solution $f$.

Now consider the scalar function
\ali{
\th_0 = \chi(t) f
}
for some cutoff function $\chi(t)$ equal to $1$ on $[-T - c/5,T + c/5]$ and equal to $0$ on $|t| \geq T + c/2$.  We can realize $\th_0$ as the scalar function part of a smooth SQG Reynolds flow
\ali{
\pr_t \th_0 + \nb_j[T^j[\th_0] \th_0] &= \nb_j \nb_\ell R^{j\ell}_{0}. \label{eq:SQGReynoldsNow}
}
Indeed, it suffices to take
\ali{
R_0^{j\ell} &\equiv \RR^{j\ell}[ \pr_t \th_0 + \nb_a[T^a[\th_0] \th_0] = \RR^{j\ell}[\pr_t \chi f + (\chi^2 - \chi)\nb_a[T^af f]] \label{eq:takingR0as}
}
Note that $R_0^{j\ell}$ solves \eqref{eq:SQGReynoldsNow} if the right hand side of \eqref{eq:takingR0as} has mean zero for all $t$, which does hold by the conservation of the integral for SQG solutions, and the fact that a divergence always has integral $0$.  Moreover, if we set $D_R = \| R_0\|_{C^0}$, we have that $(\th_0, R_0)$ has frequency energy levels to order $L$ in $C^0$ below $(\Xi, D_R, D_R)$ for all $\Xi$ sufficiently large. 

The existence of our desired SQG flow, which differs from $\th_0$ only in an arbitrarily small neighborhood of the set $[-T-c/2, -T-c/5] \cup [T+c/5, T+c/2]$ that contains the support of the right hand side of \eqref{eq:takingR0as}, now follows immediately from the approximation theorem, Theorem~\ref{thm:approxThem}, by taking $\Xi$ sufficiently large so that $\hc C_\de (\Xi e_u^{1/2})^{-1} = \hc C_\de (\Xi^{3/2} D_R^{1/2})^{-1}$ is smaller than $c / 7$.

\subsection{h-Principle} \label{sec:hPrinc}

Let $\de > 0$ be given and let $L$ and $C_\de$ be as in Theorem~\ref{thm:approxThem}.

Let $f : (0,T) \times \T^2 \to \R$ be a smooth compactly supported function that conserves the integral.  That is,
\ALI{
\int_{\T^2} f(t,x) dx = 0, \qquad \mbox{ for all } t.
}
We approximate $f$ by the sequence $f_n = P_{\leq n} f$, which satisfy 
\ali{
\sup_n \co{\nb_{\va} \pr_t^r f_n} \ls \co{\nb_{\va} \pr_t^r f}, \quad \mbox{ for } 0 \leq |\va|, r \label{eq:uniformBounds} \\ 
\lim_{n \to \infty} \co{ |\nb|^{-1/2}(f_n - f) } = 0. \label{eq:approxfn}
}
Using the order $-2$ operator $\RR^{j\ell}$, define
\ali{
R_n^{j\ell} &= \RR^{j\ell}[\pr_t f_n +  \nb_a [ f_n T^a f_n] ]
}
so that $(f_n, R_n)$ define an SQG-Reynolds flow with compact frequency support.  (It is important at this point that the right hand side has mean zero at every time.) Choose $D_{R,-1}>0$ large enough such that
\ALI{
\sup_n \co{R_n} &\leq D_{R,-1}.
}
By \eqref{eq:uniformBounds} we can choose a fixed $\Xi_{-1}\geq2$ and a sequence $\Xi_{-1,n}\geq\Xi_{-1}$ tending to $+\infty$ so that $(f_n, R_n)$ is an SQG-Reynolds flow with frequency energy levels to order $L$ bounded by $(\Xi_{-1, n}, D_{R,-1}, D_{R,-1})$ that has compact frequency support in frequencies below $\Xi_{-1,n}$.  Set
\[
e_{R,-1,n}^{1/2}:=\Xi_{-1,n}^{1/2}D_{R,-1}^{1/2},
\qquad e_{R,-1}^{1/2}:=\Xi_{-1}^{1/2}D_{R,-1}^{1/2}.
\]

 We use \cite[Lemma~3.1, with frequency--energy levels as in Definition~2.2]{isett2021direct}, which we now state in the present notation.  Let $\Xi \geq 2$ and $\mathring{D}_u \geq \mathring{D}_R > 0$ and let $N \geq \mathring{D}_u/\mathring{D}_R$.  Let $(f_{-1}, R_{-1})$ be an SQG Reynolds flow with frequency energy levels to order $L$ below $(\Xi, \mathring{D}_u, \mathring{D}_R)$ and such that $\supp \hat{f}_{-1} \subseteq \{ |\xi| \leq \Xi \}$ and such that $\suppt (f_{-1}, R_{-1}) \subseteq J$ for some nonempty set $J$.  Then there is a second SQG Reynolds flow $(\ost{f}, \ost{R})$, $\ost{f} = f_{-1} + W$ with frequency energy levels below
\[
(\Xi^*,\mathring D_u^*,\mathring D_R^*)
=\left(C_bN\Xi,\mathring D_R,G\mathring D_R\right),
\qquad G=\left(\frac{\mathring D_u}{\mathring D_R}\right)^{1/4}N^{-3/4},
\]
to order $L$ in $C^0$ that satisfies the following bounds:
\ali{
\supp_t (\ost{f}, \ost{R}) &\subseteq \{ t + t' ~:~ t \in J, |t'| \leq 3(\Xi e_u^{1/2})^{-1} \} \\
\|\nabla^{\va}|\nb|^{-1/2} W\|_{C^0} &\leq C_L(N\Xi)^{|\va|}D_R^{1/2},\qquad |\va|\leq1,\\
|\nb|^{-1/2} W &= \nb_i Y^i, \qquad \co{Y^i} \leq C_L(N\Xi)^{-1} D_R^{1/2}.
}
where $e_u^{1/2} = \Xi^{1/2} \mathring{D}_u^{1/2}$.  The estimates in this quoted lemma are not as strong as those of the present paper, but it is mainly this last estimate on the anti-divergence that we will require for our specific application to the h-principle.

Now let $N_{-1,n}$ be a sequence tending to $+\infty$ such that $N_{-1,n} \geq C_L$.
Apply the above lemma with $(\Xi_{-1,n},D_{R,-1}, D_{R,-1})$, the flow $(f_n,R_n)$, and the parameter $N_{-1,n}$. For each $n$ this gives a second SQG--Reynolds flow $(\th_{0,n},R_{0,n})$, with $\th_{0,n}=f_n+W_{-1,n}$, such that
\ali{
\supp_t (\th_{0,n}, R_{0,n}) &\subseteq \{ t + t' ~:~ t \in \suppt (f_n, R_n), |t'| \leq 3(\Xi_{-1,n} e_{R,-1,n}^{1/2})^{-1} \} \\
\co{|\nb|^{-1/2} W_{-1,n}} &\leq C_L D_{R,-1}^{1/2} \label{eq:CminOneHalfW} \\
|\nb|^{-1/2} W_{-1,n} &= \nb_i Y_n^i, \qquad 
\co{Y_n^i} \leq \Xi_{-1,n}^{-1} D_{R,-1}^{1/2} \label{eq:antiDivMin1half} 
}
and so that the frequency energy levels of $(\th_{0,n}, R_{0,n})$ are bounded to order $L$ by
\ALI{
(\Xi_{n,(0)}, D_{u,(0)}, D_{R,(0)}) &= \left( C_L N_{-1,n} \Xi_{-1,n}, D_{R,(-1)}, \fr{D_{R,-1}}{N_{-1,n}^{3/4}} \right).
}
We write $e_{u,n,(0)}^{1/2}:=\Xi_{n,(0)}^{1/2}D_{u,(0)}^{1/2}=\Xi_{n,(0)}^{1/2}D_{R,-1}^{1/2}$.
Now apply our approximation theorem, Theorem~\ref{thm:approxThem}, to get an SQG solution $\th_n$ of class $|\nb|^{-1/2}\th_n \in L_t^\infty C^{1/2- 2 \de}$ with
\ali{
\| |\nb|^{-1/2} (\th_n - \th_{0,n}) \|_{C^0} \leq C_\de\fr{D_{R,-1}^{1/2}}{N_{-1,n}^{3/8}}  \label{eq:CminOneHalfW2}
}
and with time support contained in
\ali{
\supp_t \th_n &\subseteq \{ t + t' ~:~ t \in \supp_t (\th_{0,n}, R_{0,n}), \quad |t'| \leq C_\de (\Xi_{n,(0)} e_{u,n,(0)}^{1/2})^{-1} \}.
}

We now claim that $|\nb|^{-1/2} (\th_n - f) \to 0 $ in $L^\infty$ weak-*.  To see this claim, let $g \in L^1((0,T)\times \T^2)$ and let $\ep > 0$ be given.  We will choose a small parameter $\eta$ depending on $\ep$.  Choose a $g_\eta \in C_c^\infty ((0,T) \times \T^2)$ within $\eta$ of $g$ in $L^1$, and write
\ali{
\int g |\nb|^{-1/2}(f - \th_n) dx &= I + II + III \\
I &= \int (g - g_\eta) |\nb|^{-1/2}(f - \th_n) dx \\
II &= \int g_\eta |\nb|^{-1/2}(f - \th_{0,n}) dx \\
III &= \int g_\eta |\nb|^{-1/2}( \th_{0,n} - \th_n ) dx
}
We bound
\ALI{
|I| &\leq \eta ( \co{ |\nb|^{-1/2} f } + \sup_n \co{|\nb|^{-1/2} \th_n} )
}
Note that the $\sup$ exists due to \eqref{eq:CminOneHalfW} and \eqref{eq:CminOneHalfW2}.  Now fix the choice of $\eta$ so that this term is bounded by $\ep/ 3$.

Since $\theta_{0,n}=f_n+W_{-1,n}$ and $|\nabla|^{-1/2}W_{-1,n}=\nabla_iY_n^i$, integration by parts gives
\ALI{
|II|
&\leq \|g_\eta\|_{L^1}\,\big\||\nb|^{-1/2}(f-f_n)\big\|_{C^0}
   +\|\nb g_\eta\|_{L^1}\,\|Y_n\|_{C^0} \\
&\leq \|g_\eta\|_{L^1}\,\big\||\nb|^{-1/2}(f-f_n)\big\|_{C^0}
   +\|\nb g_\eta\|_{L^1}\frac{D_{R,-1}^{1/2}}{\Xi_{-1,n}}.
}
Both terms tend to zero by \eqref{eq:approxfn} and the choice $\Xi_{-1,n}\to\infty$.

Finally we have
\ALI{
|III| &\leq \left( \int |g_\eta| \,dx\right) \co{ |\nb|^{-1/2}(\th_{0,n} - \th_n)} \\
&\leq \left( \int |g_\eta| \,dx\right) C_\de D_{R,-1}^{1/2} N_{-1,n}^{-3/8}
}
As long as we take $N_{-1,n}$ to go to infinity, this term is also arbitrarily small.  From this estimate we conclude that $|\nb|^{-1/2} \th_n \to |\nb|^{-1/2} f$ in $L^\infty$ weak-*.  Furthermore, we have 
\ali{
\suppt \th_n &\subseteq \{ t + t' ~:~ t \in \suppt f, |t'| \leq C_\de (\Xi_{-1} e_{u,-1}^{1/2})^{-1} \},
}
uniformly in $n$, which can be made arbitrarily close to $\suppt f$ by taking $\Xi_{-1}$ large.

\appendix
\section{Appendix}
\subsection{Existence of solutions to equation \eqref{eq:NewtonEquation}} \label{sec:ExistenceNewton}

Let $\Phi w$ be a solution to the equation:
\begin{equation}
\overline{D_t}[\Phi w]+T^{\ell}w\nabla_{\ell}\theta_{\ep}=f \text{ with } (\Phi w)[0]=w_0.
\end{equation}
Subtracting the equation for $\Phi \tilde{w}$ from $\Phi w$, we get:
\begin{equation}
\overline{D_t}(\Phi w-\Phi \tilde{w})+T^{\ell}(w-\tilde{w})\nabla_{\ell}\theta_{\ep}=0 \text{ with } (\Phi w-\Phi \ti w)[0]=0.
\end{equation}
Let $s \geq 0 $ be given.   Using the notation $[\nabla_{\vec{a}},u_{\ep}\cdot\nabla]=\ti \sum\mathbf{1}_{|\vec{a}_2|\leq s-1}(\nabla_{\vec{a}_1}u^j_{\ep}\nabla_{\vec{a}_2}\nabla_j)$,\footnote{We only want to take up to $s$ derivatives of $\Phi w-\Phi\ti w$.} we differentiate the equation with $\nabla_{\vec{a}}$ to get:   \begin{multline}
\overline{D_t}\nabla_{\vec{a}}(\Phi w-\Phi \tilde{w})+[\nabla_{\vec{a}},u_{\ep}\cdot\nabla](\Phi w-\Phi \tilde{w}) \
+\underset{\va_1, \va_2}{\overset{\sim}{\sum}}
\nabla_{\vec{a}_1}T^{\ell}(w-\tilde{w})\nabla_{\vec{a}_2}\nabla_{\ell}\theta_{\ep}=0.
\end{multline}
Multiplying this equation by $\nabla_{\vec{a}}(\Phi w-\Phi \tilde{w})$ and integrating by parts,
$$\frac12\partial_t\|\nb_{\va}(\Phi w)-\nb_{\va}(\Phi\ti w)\|_2^2 $$
$$+ \int \left( [\nabla_{\vec{a}},u_{\ep}\cdot\nabla](\Phi w-\Phi \tilde{w}) \
+\underset{\va_1, \va_2}{\overset{\sim}{\sum}}
\nabla_{\vec{a}_1}T^{\ell}(w-\tilde{w})\nabla_{\vec{a}_2}\nabla_{\ell}\theta_{\ep} \right)\nb_{\va}(\Phi w-\Phi\ti w)\,dx=0,$$
where 
$$\int u_\ep^j\nb_j (\nb_{\va}\Phi w-\nb_{\va}\Phi\ti w)^2 \,dx=0$$
due to the divergence-free property of $u_\ep.$ 

Integrating on $[0,t]$, and using H\"older's inequality, we have
\begin{align*}
\|\nabla_{\vec{a}}(\Phi w) - \nabla_{\vec{a}}(\Phi \tilde{w})\|_2(t)
&\lesssim \int_0^t \left( \|[\nabla_{\vec{a}}, u_\epsilon \cdot \nabla](\Phi w - \Phi \tilde{w})\|_2 + \|\nabla_{\vec{a}_1} T^\ell (w - \tilde{w}) \nabla_{\vec{a}_2} \nabla_\ell \theta_\epsilon\|_2 \right) d\tau \\
&\lesssim t \|T^\ell (w - \tilde{w})\|_{L^\infty H^s} + t \| u_\epsilon\|_{L^\infty C^{s}} \|\Phi w - \Phi \tilde{w}\|_{L^\infty H^s}.
\end{align*}
In the first line, we used Cauchy-Schwarz. 
Here, we used that \(\|\nabla_{\vec{a}_2} \nabla_\ell \theta_\epsilon\|_{L^\infty} \lesssim 1\), which is true since $\th_\ep$ is smooth.  Next, note that \(T^\ell\) is bounded on \(H^s\). Thus we have \(\|T^\ell (w - \tilde{w})\|_{H^s} \lesssim \|w - \tilde{w}\|_{H^s}\). Thus
\ALI{
\sum_{|\va|\le s}\|\nb_{\va}(\Phi w)-\nb_{\va}(\Phi \ti w)\|_2(t) \ls_{\th_\ep, u_\ep} t\|w-\ti w\|_{L^\infty H^s} + t \|\Phi w-\Phi \ti w\|_{L^\infty H^s}
}
and indeed, taking a sup over $t \in [-T,T]$
$$\sum_{|\va|\le s}\|\nb_{\va}(\Phi w)-\nb_{\va}(\Phi \ti w)\|_{L^\infty L^2} \ls_{\th_\ep, u_\ep} T\|w-\ti w\|_{L^\infty H^s} + T\|\Phi w-\Phi \ti w\|_{L^\infty H^s}.
$$
By taking \(T\) sufficiently small, we can absorb the last term into the left-hand side. Taking $T$ smaller if necessary, we obtain 
$$\|\Phi w-\Phi \ti w\|_{L^\infty H^s} \le C \|w-\ti w\|_{L^\infty H^s}, \quad C \in (0,1).$$
We apply the contraction mapping theorem and conclude that there exists a unique fixed point $w \in L^\infty_t H^s_x$ of $\Phi$, which solves equation \eqref{eq:NewtonEquation}.

Inspecting the proof, the timescale of existence is bounded from below by 
\[ C^{-1} (\max \{ \| \th \|_{L_t^\infty C^{s+1}}, \| u_\ep \|_{L_t^\infty C^{s}} \} )^{-1}. \]
Consequently, if $\th$ and $u_\ep$ are smooth, the solution is global in time and smooth in the spatial variables.

\subsection{The Divergence Form Principle}  \label{sec:divFormPrinciple}

Let $\lambda>0$, and let $P_{\lambda,1}$ and $P_{\lambda,2}$ be frequency-localizing operators with smooth multipliers $\chi_{\lambda,1}$ and $\chi_{\lambda,2}$ supported where $c\lambda\leq|\xi|\leq C\lambda$.  Thus
\[
\widehat{P_{\lambda,i}f}(\xi)=\chi_{\lambda,i}(\xi)\widehat f(\xi),
\qquad \chi_{\lambda,i}(\xi)=\chi_{1,i}(\xi/\lambda).
\]
Choose an even function $\chi\in C_c^\infty(\mathbb R^d\setminus\{0\})$ that equals one on an open neighborhood of $\supp\chi_{1,1}\cup\supp\chi_{1,2}$, and define
\[
m_\lambda(\xi):=\chi(\xi/\lambda)m(\xi).
\]
When $m$ is the SQG symbol, we choose $\chi$ radial.  The construction below follows \cite[Section~4.5.2, equations~(15)--(21)]{isett2021direct} and \cite[Section~5, equations~(40)--(45)]{isett2024conservation}, both building on a cancellation observed in \cite{buckShkVicSQG}.  
\begin{thm} \label{thm:divFormPrinciple} Let $\OO$ be an operator with odd symbol $m$ that is degree $\b$ homogeneous and smooth away from $0$.  Then for smooth $f,g$ one can write
\ali{
P_{\la,1} f \OO P_{\la,2} g  &+ P_{\la,2}g \OO P_{\la,1} f= \nb_j [ K_\la^{j} \ast[f,g] ] \label{eq:plusSign} \\
K_\la^{j} \ast[f,g] &= \int_{\R^d \times \R^d} f(x - h_1) g(x - h_2) K_\la^j(h_1, h_2) dh_1 dh_2 \label{eq:bilinConvForm} \\
K_\la^j(h_1, h_2) &= \la^{2 d + \b - 1} K_0^j(\la h_1, \la h_2) \notag 
}
where $K_0^j$ is Schwartz.  In the specific case in which $\OO = T^\ell$ is the SQG multiplier, the tensor $K_\la^{j\ell}$ is trace free and satisfies
\ali{
\widehat{K}_\la^{j\ell}(p, -p) &= -i\nb^j m^\ell(p)
}
for all $p$ such that $\chi_{\la,1}(p) = \chi_{\la,2}(-p) = 1$.
More precisely, if $Q^\ell$ denotes the left hand side of \eqref{eq:plusSign} for $\mathcal O=T^\ell$, then
\ALI{
Q^\ell=\nabla_j(K_\lambda^{j\ell}\ast[f,g]),
\qquad
\nabla_\ell Q^\ell
=\nabla_j\nabla_\ell(K_{\lambda,\mathrm{sym}}^{j\ell}\ast[f,g]),
}
where $K_{\lambda,\mathrm{sym}}^{j\ell}:=(K_\lambda^{j\ell}+K_\lambda^{\ell j})/2$.  Thus symmetrization preserves the double-divergence identity used in our application, although it need not preserve the vector identity for $Q^\ell$.  Finally, when $p$ is such that $\chi = 1$ in a neighborhood of $p/\la$, we have the formula
\ali{
\widehat{K}_{\lambda,\mathrm{sym}}^{j\ell}(p,-p)
&=-\frac{i}{2}\big(\nb^jm^\ell(p)+\nb^\ell m^j(p)\big). \label{eq:divFormDiagonal}
}
\end{thm}
\begin{proof}
By the approximation in \cite[Section~5, before equation~(41)]{isett2024conservation}, it suffices to obtain the divergence form on $\R^d$ for Schwartz functions $f,g$.  Let $Q$ denote the left hand side of \eqref{eq:plusSign}.  Since $m_\lambda=m$ on the supports of both localizers, convolution in Fourier space gives
\ALI{
\hat{Q}(\xi)
&= \int_{\widehat{\R^d}} [m(\xi - \eta) + m(\eta)] \widehat{P_{\la,1}f}(\xi - \eta) \widehat{P_{\la,2}g}(\eta) \,d\eta\\
&= \int_{\widehat{\R^d}} [m_\la(\xi - \eta) + m_\la(\eta)] \widehat{P_{\la,1}f}(\xi - \eta) \widehat{P_{\la,2}g}(\eta) \,d\eta.
}
The evenness of $\chi$ and oddness of $m$ imply that $m_\lambda$ is odd.  Hence  \ALI{
m_\lambda(\xi - \eta )+m_\lambda(\eta)
&=m_\lambda(\xi - \eta )-m_\lambda(-\eta) =\xi_j\int_0^1\nabla^jm_\lambda
   \big(\sigma\xi -\eta\big)\,d\sigma.
}
Hence, we obtain a divergence form
\ALI{
\hat{Q}(\xi) &= \xi_j \int_{\widehat{\R}^d} \int_0^1 \nb^j m_\la(\si \xi - \eta) d\si \widehat{P_{\la,1}f}(\xi - \eta) \widehat{P_{\la,2}g}(\eta) \,d\eta \\
\hat{Q}(\xi) &= i \xi_j \widehat{U^j}(\xi)}
Taking the inverse Fourier transform of $\widehat{U^j}(\xi)$ as in \cite[equations~(41)--(44)]{isett2024conservation} using $\zeta \equiv \xi - \eta$ as an integration variable gives \eqref{eq:bilinConvForm}, with kernel defined by
\ALI{
\widehat{K_\lambda^j}(\zeta,\eta)
&=-i\,\chi_{\lambda,1}(\zeta)\chi_{\lambda,2}(\eta)
\int_0^1\nabla^jm_\lambda
   \big(\sigma\zeta-(1-\sigma)\eta\big)\,d\sigma.
}

The Fourier-side kernel is smooth and compactly supported because $m_\lambda$ is smooth on all of $\R^d$.  It therefore has a Schwartz inverse Fourier transform.  Homogeneity of $m$ and the change of variables $(\zeta,\eta)=\lambda(\zeta',\eta')$ give
\[
K_\lambda^j(h_1,h_2)
=\lambda^{2d+\beta-1}K_0^j(\lambda h_1,\lambda h_2).
\]
At $(\zeta,\eta)=(p,-p)$, the segment in the kernel formula is the single point $p$.  The plateau property that $\chi = 1$ where $\chi_{\la,1}(p)=\chi_{\la,2}(-p)$ are supported therefore gives $\nabla m_\lambda(p)=\nabla m(p)$ in this range and hence the stated diagonal identity.  In the SQG case, radiality of $\chi$ gives $\nabla_jm_\lambda^j=0$, so $K_\lambda^{j\ell}$ is trace free.  Finally, contraction of two derivatives with the antisymmetric part of $K_\lambda^{j\ell}$ vanishes, which proves the double-divergence statement and \eqref{eq:divFormDiagonal}.
\end{proof}

We also use a version of this principle for even multipliers.
\begin{thm}  Let $\calE$ be an operator with even symbol $m$ that is degree $\b$ homogeneous and smooth away from $0$.  Then for smooth $f,g$ one can write
\ali{
    P_{\la,1} f \calE P_{\la,2} g  &- P_{\la,2}g \calE P_{\la,1} f = \nb_j [ K_\la^{j} \ast[f,g] ]\label{eq:minusSign}\\
K_\la^{j} \ast[f,g] &= \int_{\R^d \times \R^d} f(x - h_1) g(x - h_2) K_\la^j(h_1, h_2) dh_1 dh_2 \notag \\
K_\la^j(h_1, h_2) &= \la^{2 d + \b - 1} K_0^j(\la h_1, \la h_2) \notag
}
where the $K_0^j$ are Schwartz functions.
\end{thm}
What is crucial here is the minus sign in \eqref{eq:minusSign} instead of the plus sign in \eqref{eq:plusSign}.  With the corrected order of the localizers, the Fourier multiplier is $m(\eta)-m(\xi-\eta)=m(-\eta)-m(\xi-\eta)$, which is factored by $\xi$ along the same segment as above.

\subsection{Glossary} \label{sec:Glossary}

\begingroup
\footnotesize
\setlength{\LTleft}{0pt}
\setlength{\LTright}{0pt}
\setlength{\LTpre}{0.5em}
\setlength{\LTpost}{0.5em}
\setlength{\tabcolsep}{4pt}
\renewcommand{\arraystretch}{1.16}
\newcommand{\glossrow}[3]{#1 & #2 & #3 \\}

This glossary is alphabetized by the printed symbol. Hats, bars, checks, tildes, and stars are sorted under the underlying letter. We include symbols that have a continuing role in the proof, and we omit dummy tensor indices.

\begin{longtable}{@{}>{\raggedright\arraybackslash}p{0.21\textwidth}>{\raggedright\arraybackslash}p{0.27\textwidth}>{\raggedright\arraybackslash}p{0.46\textwidth}@{}}
\hline
\textbf{Symbol} & \textbf{First appearance} & \textbf{Definition} \\
\hline
\endfirsthead
\hline
\textbf{Symbol} & \textbf{First appearance} & \textbf{Definition} \\
\hline
\endhead
\hline
\endfoot
\hline
\endlastfoot

\glossrow{$\alpha$}{Line \eqref{eq:defhh}}{Some H\"older exponent in $(0,1)$.}
\glossrow{$A$}{Introduction}{The angular momentum $A=\int_{\mathbb R^2}|x|^2\theta(t,x)dx$ in the nonperiodic conservation laws.}
\glossrow{$A_{(f,J)}^{j\ell}$}{Equation~\eqref{eq:NewtonEquation}}{The one dimensional stress component used as forcing in the Newton equation with $J=(k,n)$.}
\glossrow{$A_L$}{Equation~\eqref{eq:ALappears}}{The implicit constant in $\co{R_\ep} \leq A_L \co{R}$ and $\co{\nb u_\ep} \leq A_L \co{\nb u}$.
}
\glossrow{$b$}{Equation~\eqref{eq:natTimescale}}{A small dimensionless constant in the Newton time scale $\tau=b(\log\widehat\Xi)^{-1}(\Xi e_u^{1/2})^{-1}$.}
\glossrow{$B^{j\ell}(p)$}{Line \eqref{eq:Bjldef}}{The trace free tensor symbol $B^{j\ell}(p)=-i(\partial_{p_j}m^\ell(p)+\partial_{p_\ell}m^j(p))$.}
\glossrow{$\BB_a^{ij}[z^a,\theta]$}{Equation~\eqref{eq:plainForm2}}{A bilinear anti-divergence with $\nabla_j\BB_a^{ij}[z^a,\theta]=\nabla_aT^i[z^a]\theta-z^a\nabla_aT^i[\theta]$.}
\glossrow{$\widehat C$}{Main Lemma}{A large constant depending on $L$ and $\eta$.}
\glossrow{$\chi_k$}{Section~\ref{sec:preludeConvex}}{A partition of unity cutoff at the time scale $\tau$.}
\glossrow{$\tilde\chi_k$}{Section ``Setting up the Newton iteration''}{A wider time cutoff used for the Newton correction.}

\glossrow{$d$}{Section~\ref{sec:notation}}{The spatial dimension. In the paper $d=2$.}
\glossrow{$D$}{Proposition ``Nonstationary phase''}{The number of terms retained in the nonstationary phase parametrix.}
\glossrow{$D_t$}{Definition~\ref{defn:freqEnLevels}}{The advective derivative $D_t=\partial_t+T^\ell\theta\nabla_\ell$.}
\glossrow{$\Ddt$}{Section ``Regularizing the scalar field and error tensor''}{The coarse scale advective derivative $\Ddt=\partial_t+u_\epsilon\cdot\nabla$.}
\glossrow{$\Dtdt$}{Section ``Convex integration''}{The post Newton coarse scale advective derivative $\Dtdt=\partial_t+\tilde u_\epsilon\cdot\nabla$.}
\glossrow{$\ost D_t$}{Section ``Estimating the corrections''}{The final advective derivative $\ost D_t=\partial_t+\ost u\cdot\nabla$.}
\glossrow{$\D_R$}{Definition~\ref{defn:freqEnLevels}}{The stress level in the frequency energy bounds.}
\glossrow{$D_{R,n}$}{Equation~\eqref{eq:DRn}}{The target size of the $n$th Newton gluing error, $D_{R,n}=(N^{-\eta}\Xi^{-\eta})^nD_{R,0}$.}
\glossrow{$\D_u$}{Definition~\ref{defn:freqEnLevels}}{The scalar or velocity level in the frequency energy bounds.}
\glossrow{$\delta$}{Section~\ref{sec:mainLemImpThm}}{A small loss parameter used in the proof of Theorem~\ref{thm:approxThem}.}
\glossrow{$e_0(t)$}{Equation~\eqref{eq:step0Prepared}}{A time envelope chosen so that $e_0(t)M^{j\ell}$ dominates $R_\epsilon^{j\ell}$.}
\glossrow{$e_n(t)$}{Before equation~\eqref{eq:enConstruct}}{The time envelope for the $n$th Newton stage, with size comparable to $D_{R,n}$.}
\glossrow{$e_R$}{Definition~\ref{defn:freqEnLevels}}{The stress energy level, $e_R^{1/2}=\Xi^{1/2}\D_R^{1/2}$.}
\glossrow{$e_u$}{Definition~\ref{defn:freqEnLevels}}{The velocity energy level, $e_u^{1/2}=\Xi^{1/2}\D_u^{1/2}$.}
\glossrow{$\widetilde e_u$}{Equation~\eqref{eq:effectiveEu}}{The auxiliary energy level $\widetilde e_u^{1/2}=(\log\widehat\Xi)e_u^{1/2}$ used when differentiating time cutoffs and envelopes; $\Xi\widetilde e_u^{1/2}=b\tau^{-1}\sim_{L,\eta}\tau^{-1}$.}
\glossrow{$\epsilon^{\ell a}$}{Introduction}{The two dimensional antisymmetric symbol used in $T^\ell\theta=\epsilon^{\ell a}\nabla_a|\nabla|^{-1}\theta$.}
\glossrow{$\epsilon$}{Section~\ref{sec:lengthAndTimeScales}}{The spatial regularization length $\epsilon=N^{-1/L}\Xi^{-1}=\widehat N^{-1}\Xi^{-1}$.}
\glossrow{$\epsilon_t$}{Equation~\eqref{eq:ept}}{The time mollification scale along the $\Dtdt$ flow, $\epsilon_t=(\Xi e_u^{1/2})^{-1}(D_u/D_R)^{-1/2}N^{-1/2}$.}
\glossrow{$\epsilon_x$}{Section ``Regularizing the scalar field and error tensor''}{The spatial mollification scale for the stress, equal to $N^{-1/L}\Xi^{-1}$.}
\glossrow{$\eta$}{Main Lemma}{A small parameter controlling the number of Newton steps and iterations of nonstationary phase in the iteration.}
\glossrow{$\varepsilon$}{Section~\ref{sec:mainLemImpThm}}{The parameter $\varepsilon=6/L+4\eta$, chosen below the target loss.}
\glossrow{$F$}{Introduction and Section ``The Newton Step''}{In the Introduction, $F$ is a smooth function in $\int F(\theta(t,x))dx$. In the construction, $F=\{\pm(1,2),\pm(2,1)\}$ is the set of wave directions.}
\glossrow{$\F$}{Section ``The Newton Step''}{The reduced direction set $\F=\{(1,2),(2,1)\}$.}
\glossrow{$f$}{Theorem~\ref{thm:hPrinciple}}{The compactly supported mean zero target function in the h-principle statement.}
\glossrow{$f_n$}{Section ``h-Principle''}{The restricted frequency approximation $f_n=P_{\le n}f$.}
\glossrow{$g_{(f,[k],n)}$}{Section~\ref{sec:settingUp}}{A family of smooth $1$ periodic cutoffs with disjoint support and normalization properties.}
\glossrow{$g_{[f,J]}$}{Before equation~\eqref{eq:NewtonEquation}}{The shorthand $g_{[f,J]}=g_{(f,[k],n)}$ for $J=(k,n)$.}
\glossrow{$g_{[I]}$}{Equation~\eqref{eq:ThI}}{The cutoff $g_{(f,[k],n)}$ attached to the wave index $I=(f,k,n)$.}
\glossrow{$g_{(k)}$}{Equation~\eqref{eq:gkplus1}}{The scalar amplitude produced at the $k$th step of the nonstationary phase parametrix.}
\glossrow{$\Gamma$}{Section ``The Newton Step''}{The number of Newton iterations, $\Gamma=\lceil\eta^{-1}\rceil$.}
\glossrow{$\gamma_f$}{Equation~\eqref{eq:implicitFunct}}{A smooth coefficient function obtained by solving the algebraic stress equation.}
\glossrow{$\gamma_I$}{Equation~\eqref{eq:gaIdefn}}{The slowly varying amplitude coefficient for the wave $\Theta_I$.}
\glossrow{$H_{\zeta,M}^{(R)}[F]$}{Before Lemma~\ref{lem:mixed derivatives}}{A weighted norm controlling up to $R$ coarse advective derivatives and up to $M$ total derivatives.}
\glossrow{$\mrg H_\zeta[F]$}{Lemma~\ref{lem:interchangeDerivs}}{A weighted norm built from one of $\partial_t$, $\Ddt$, or $\Dtdt$.}
\glossrow{$\mrg H'_\zeta[F]$}{After Proposition~\ref{prop:chainRule}}{The prime version of $\mrg H_\zeta$, with the spatial loss shifted from $\underline L$ to $\underline L-1$.}
\glossrow{$\mrg H^*[F]$}{Equation~\eqref{eq:finalWeighted} and Definition~\ref{defn:finalWeight}}{The final weighted norm at frequency scale $N\Xi$ and time scale $\ost\tau$.}
\glossrow{$\tilde H''_M[F]$}{Section ``Nonstationary phase''}{The weighted norm with spatial scale $N^{1/2}\Xi$ used in the nonstationary phase estimate.}
\glossrow{$\calH^\ell_a$}{Section~\ref{sec:notation}}{The Helmholtz projection symbol $\calH^\ell_a=\delta^\ell_a-\Delta^{-1}\nabla^\ell\nabla_a$.}
\glossrow{$\hh(t)$}{Before Proposition~\ref{prop:weightedNorm}}{The weighted sum of $C^0$ and $\dot C^\alpha$ norms used in the Newton step Gronwall estimate.}

\glossrow{$I$}{Section ``Convex integration''}{A wave index $I=(f,k,n)\in\II$.}
\glossrow{$\bar I$}{Section ``The Newton Step''}{The conjugate index, with $\Theta_{\bar I}=\overline{\Theta_I}$ and $\xi_{\bar I}=-\xi_I$.}
\glossrow{$\II$}{Section ``Convex integration''}{The full wave set is $\mathcal I=\mathcal I_+\cup\{\bar I:I\in\mathcal I_+\}$, where $\mathcal I_+=\F\times\mathbb Z\times\{0,\ldots,\Gamma-1\}$.}
\glossrow{$\vec I$}{Introduction}{The impulse $\vec I=\int_{\mathbb R^2}x\theta(t,x)dx$ in the nonperiodic conservation laws.}
\glossrow{$J$}{Section ``Setting up the Newton iteration''}{A Newton index $J=(k,n)$.}
\glossrow{$J_0$}{Main Lemma}{A time set containing $\operatorname{supp}_t R$.}
\glossrow{$J_n$}{Before equation~\eqref{eq:enConstruct}}{The enlarged time support used to define $e_n(t)$.}
\glossrow{$K_\lambda^j$}{Theorem~\ref{thm:divFormPrinciple}}{The bilinear convolution kernel giving a divergence form for high frequency products.}
\glossrow{$K_\lambda^{j\ell}$}{Theorem~\ref{thm:divFormPrinciple}}{The trace free bilinear kernel used for SQG products.}
\glossrow{$K_q^{j\ell}$}{Proof of the $\BB$ term}{The dyadic version of the bilinear convolution kernel.}
\glossrow{$L$}{Definition~\ref{defn:freqEnLevels}}{The number of spatial derivatives controlled by the frequency energy levels. In the Main Lemma, $L\geq 7$.}
\glossrow{$L'$}{Before Proposition~\ref{prop:weightedNorm}}{A derivative count used in the Newton step bootstrap norm.}
\glossrow{$\underline L$}{Before Proposition~\ref{prop:newtonBound}}{The threshold for borrowed derivatives after the Newton step.  We put $\underline L=L-3$.}
\glossrow{$\lambda$}{Section ``Convex integration''}{The oscillation frequency, set here as $\lambda=\lceil N\Xi\rceil$.}
\glossrow{$M$}{Introduction}{The mean $M=\int\theta(t,x)dx$.}
\glossrow{$M^{j\ell}$}{Equation~\eqref{eq:stressEqn}}{The fixed trace free tensor $M^{j\ell}=B^{j\ell}((1,2))+B^{j\ell}((2,1))$.}
\glossrow{$M_e$}{After equation~\eqref{eq:neighborhood}}{A large constant used to make $e_0(t)M^{j\ell}$ dominate $R_\epsilon^{j\ell}$.}
\glossrow{$m^\ell(p)$}{Section ``Overall gameplan''}{The SQG multiplier $m^\ell(p)=i\epsilon^{\ell a}p_a|p|^{-1}$.}
\glossrow{$\mu$}{Equation~\eqref{eq:muChoice}}{The inverse time scale $\mu=N^{1/2}\Xi e_R^{1/2}=N^{1/2}\Xi^{3/2}\D_R^{1/2}$.}
\glossrow{$N$}{Main Lemma}{The large frequency growth parameter satisfying the lower bound~\eqref{eq:NlowerBd}.}
\glossrow{$N(J_0)$}{Equation~\eqref{eq:supportEnlargement}}{The time enlargement $N(J_0)=\{t+h:t\in J_0, |h|\le 5(\Xi e_u^{1/2})^{-1}\}$.}
\glossrow{$N_{(k)}$}{Proof of Theorem~\ref{thm:approxThem}}{The value of $N$ used at the $k$th application of the Main Lemma.}
\glossrow{$\widehat N$}{After the Main Lemma}{We put $\widehat N=N^{1/L}$.}
\glossrow{$\nabla_{\vec a}$}{Section~\ref{sec:notation}}{The spatial derivative associated to the multi-index $\vec a$.}
\glossrow{$O_J^{j\ell}$}{Equation~\eqref{eq:NewtonEquation}}{The Newton forcing tensor, defined by $\nabla_j\nabla_\ell O_J^{j\ell}=\sum_{f\in\F}(1-g_{[f,J]}^2(\mu t))\nabla_j\nabla_\ell A_{(f,J)}^{j\ell}$.}
\glossrow{$\tilde O_J$, $\tilde O_J^i$, $\tilde O_J^{ij}$}{Equations~\eqref{eq:barwJ}, \eqref{eq:barzJ}, and \eqref{eq:barrJ}}{The lower order forcing terms produced after integrating the time oscillation by parts.}
\glossrow{$\OO$}{Theorem~\ref{thm:divFormPrinciple}}{A Fourier multiplier with odd symbol in the divergence form principle.}
\glossrow{$\calE$}{After Theorem~\ref{thm:divFormPrinciple}}{A Fourier multiplier with even symbol in the minus sign version of the divergence form principle.}
\glossrow{$p_I$}{Equation~\eqref{eq:pI}}{The convex integration amplitude input is $p_I=(\tilde R_{(n)}/(M_eD_{R,n}),\nabla\xi_k)$.}
\glossrow{$\check p_I$}{Equation~\eqref{eq:pchkI}}{The Newton step amplitude input is $\check p_I=(R_{(n)}/(M_eD_{R,n}),\nabla\check\xi_k)$.}
\glossrow{$P_{\le q}$}{Section~\ref{sec:notation}}{The low frequency Littlewood Paley projection.}
\glossrow{$P_q$}{Section~\ref{sec:notation}}{The dyadic Littlewood Paley projection $P_q=P_{\le q+1}-P_{\le q}$.}
\glossrow{$P_{\approx q}$}{Section~\ref{sec:notation}}{A Fourier multiplier adapted to frequencies $|\xi|\sim2^q$.}
\glossrow{$P_I$}{Equation~\eqref{eq:ThI}}{The Fourier projection adapted to the wave packet indexed by $I$.}
\glossrow{$P_{(n)}^{j\ell}$}{Equation~\eqref{eq:afternNewton}}{The acceptable stress error after $n$ Newton steps.}
\glossrow{$\Phi_s$}{Before Proposition~\ref{prop:stress}}{The flow map of $\partial_t+\tilde u_\epsilon\cdot\nabla$ used to define $\tilde R_{(n)}$.}
\glossrow{$\Phi$}{Appendix subsection on equation~\eqref{eq:NewtonEquation}}{The contraction map used to construct solutions to the Newton equation.}
\glossrow{$\psi_{(k)}$}{Proof of Theorem~\ref{thm:approxThem}}{The logarithmic parameter vector $\psi_{(k)}=[\log D_{R,(k)},\log(D_u/D_R)_{(k)},\log\Xi_{(k)}]^t$.}
\glossrow{$\psi_+$, $\psi_0$, $\psi_1$}{Proof of Theorem~\ref{thm:approxThem}}{Eigenvectors for the parameter evolution matrix; $\psi_+$ is displayed in equation~\eqref{eq:psiPlus}.}

\glossrow{$\widehat q$}{Before Lemma~\ref{lem:cdaCharacterization}}{A dyadic index satisfying $2^{\widehat q}\sim\widehat\Xi$.}
\glossrow{$q_\epsilon$}{Section ``Regularizing the scalar field and error tensor''}{An integer with $2^{q_\epsilon}\sim\epsilon^{-1}=\widehat N\Xi$.}
\glossrow{$q^\lambda$}{Section~\ref{sec:cvxMollErr}}{A dyadic index satisfying $2^{q^\lambda}\sim\lambda$.}
\glossrow{$\bar q$}{Proof of equation~\eqref{eq:sizeOfNewtCorrect}}{A dyadic cutoff with $2^{\bar q}\sim\Xi$.}
\glossrow{$Q(\phi,\theta)$}{Introduction}{The quadratic form used to define the SQG nonlinearity at low regularity.}
\glossrow{$Q^{j\ell}$}{Proposition ``Nonstationary phase''}{A trace free double anti-divergence of $e^{i\lambda\xi_I}g$.}
\glossrow{$Q_{(D)}^{j\ell}$}{Equation~\eqref{eq:parametrix}}{The finite nonstationary phase parametrix.}
\glossrow{$\tilde Q_{(D)}^{j\ell}$}{Equation~\eqref{eq:parametrix}}{The elliptic correction to the nonstationary phase parametrix.}
\glossrow{$q^{j\ell}(p)$}{Proof of Proposition ``Nonstationary phase''}{The trace free symbol satisfying $p_jp_\ell q^{j\ell}(p)=1$.}
\glossrow{$q_{(k)}^{j\ell}$}{Equation~\eqref{eq:gkplus1}}{The $k$th tensor amplitude in the nonstationary phase parametrix.}
\glossrow{$R^{j\ell}$}{Definition of the SQG Reynolds equations}{The symmetric trace free stress tensor; the SQG equation is recovered when $R=0$.}
\glossrow{$R_\epsilon^{j\ell}$}{Section ``Regularizing the scalar field and error tensor''}{The spatial mollification of the stress $R^{j\ell}$.}
\glossrow{$R_{(n)}^{j\ell}$}{Equation~\eqref{eq:afternNewton}}{The gluing error after $n$ Newton steps.}
\glossrow{$R_H^{j\ell}$}{Equation~\eqref{eq:errorList}}{The high frequency self interaction error in the convex integration step.}
\glossrow{$R_M^{j\ell}$}{Equation~\eqref{eq:errorList}}{The mollification error in the convex integration step.}
\glossrow{$R_{M,n}^{j\ell}$}{Before Proposition~\ref{prop:RmRqnewt}}{The Newton step mollification error caused by $\theta-\theta_\epsilon$.}
\glossrow{$R_{MHq}^{j\ell}$, $R_{MHL}^{j\ell}$, $R_{MLH}^{j\ell}$}{Section~\ref{sec:cvxMollErr}}{The high high, high low, and low high pieces of the convex integration mollification error.}
\glossrow{$R_{Q,n}^{j\ell}$}{Before Proposition~\ref{prop:RmRqnewt}}{The Newton step quadratic error from products of Newton corrections.}
\glossrow{$R_S^{j\ell}$}{Equation~\eqref{eq:errorList}}{The stress and flow error in the convex integration step.}
\glossrow{$R_{SI}^{j\ell}$}{Section~\ref{sec:estimateRS}}{The $I$th contribution to the stress or flow error $R_S$.}
\glossrow{$R_T^{j\ell}$}{Equation~\eqref{eq:errorList}}{The transport error in the convex integration step.}
\glossrow{$\RR_a^{j\ell}$}{Section~\ref{sec:notation}}{The order $-1$ trace free anti-divergence for mean zero vector fields, so that $\nabla_j\RR_a^{j\ell}[f^a]=f^\ell$.}
\glossrow{$\RR^{j\ell}$}{Section~\ref{sec:notation}}{The order $-2$ trace free anti-divergence for mean zero scalar fields, so that $\nabla_j\nabla_\ell\RR^{j\ell}[f]=f$.}
\glossrow{$\rho_J^{j\ell}$}{Before Proposition~\ref{prop:newtonBound}}{The trace free double anti-divergence $\rho_J^{j\ell}=\RR^{j\ell}w_J$.}
\glossrow{$r_J^{ij}$}{Equation~\eqref{eq:transElliptic}}{A second order anti-divergence for $w_J$, transported by the coarse flow.}
\glossrow{$\bar r_J^{ij}$, $\tilde r_J^{ij}$}{Before equation~\eqref{eq:barrJ}}{The slowly varying and explicitly oscillatory parts of $r_J^{ij}$.}
\glossrow{$S_F$}{Table~\eqref{eq:sizeOfTerms}}{The assigned size of a Newton step field $F$.}
\glossrow{$S_{(n)}^{j\ell}$}{Equation~\eqref{eq:afternNewton}}{The part of the stress after $n$ Newton steps designed for cancellation by one dimensional waves.}
\glossrow{$S_w$, $S_z$, $S_r$, $S_u$, $S_\rho$}{Table~\eqref{eq:sizeOfTerms}}{The assigned sizes of $w_J$, $z_J$, $r_J$, $u_J$, and $\rho_J$.}
\glossrow{$T$}{SQG equation in the Introduction}{The operator $T^\ell\theta=\epsilon^{\ell a}\nabla_a|\nabla|^{-1}\theta$ mapping the scalar to the velocity.}
\glossrow{$T_\delta$}{Proof of Theorem~\ref{thm:approxThem}}{The parameter evolution matrix for the logarithmic variables.}
\glossrow{$\theta$}{Introduction}{The scalar unknown in SQG.  Later the scalar field in the SQG Reynolds equations.}
\glossrow{$\theta_\epsilon$}{Section ``Regularizing the scalar field and error tensor''}{The coarse scalar field $\theta_\epsilon=P_{\le q_\epsilon}\theta$.}
\glossrow{$\theta_I$}{Equation~\eqref{eq:gaIdefn}}{The leading amplitude of the wave $\Theta_I$. It is $\theta_I=(2\lambda)^{1/2}\gamma_I$.}
\glossrow{$\theta_n$}{Theorem~\ref{thm:hPrinciple}, then the Newton step}{In the theorem, $\theta_n$ is an approximating SQG solution. In the Newton step, it denotes the scalar after $n$ Newton corrections.}
\glossrow{$\Theta$}{Section ``Overall gameplan''}{The oscillatory perturbation added after the Newton step.}
\glossrow{$\Theta_I$}{Equation~\eqref{eq:ThI}}{The $I$th oscillatory wave, $\Theta_I=g_{[I]}(\mu t)P_I[e^{i\lambda\xi_I}\theta_I]$.}
\glossrow{$\delta\theta_I$}{After Proposition~\ref{prop:amplitudeBounds}}{The lower order scalar amplitude produced by the microlocal expansion of $\Theta_I$.}
\glossrow{$\tau$}{Equation~\eqref{eq:natTimescale}}{The Newton time scale $\tau=b(\log\widehat\Xi)^{-1}(\Xi e_u^{1/2})^{-1}$.}
\glossrow{$\ost\tau$}{After equation \eqref{eq:finalWeighted}}{The final time scale, defined by $\ost\tau^{-1}=(N\Xi)^{3/2}D_R^{1/2}$.}
\glossrow{$u$}{Introduction}{The velocity field $u^\ell=T^\ell\theta$.}
\glossrow{$u_\epsilon$}{Section ``Regularizing the scalar field and error tensor''}{The coarse velocity field $u_\epsilon^\ell=T^\ell\theta_\epsilon$.}
\glossrow{$\tilde u_\epsilon$}{Section ``Convex integration''}{The coarse velocity after the Newton step, $\tilde u_\epsilon^\ell=T^\ell\tilde\theta_\epsilon$.}
\glossrow{$u_I^\ell$}{After Proposition~\ref{prop:amplitudeBounds}}{The leading velocity amplitude $u_I^\ell=m^\ell(\nabla\xi_I)\theta_I$.}
\glossrow{$\delta u_I^\ell$}{After Proposition~\ref{prop:amplitudeBounds}}{The lower order velocity amplitude produced by the microlocal expansion of $T^\ell\Theta_I$.}
\glossrow{$u_J^\ell$}{Before Proposition~\ref{prop:newtonBound}}{The velocity generated by the Newton correction, $u_J^\ell=T^\ell w_J$.}
\glossrow{$\vec a$}{Section~\ref{sec:notation}}{A multi-index for spatial derivatives.}
\glossrow{$W$}{Main Lemma}{The total correction in the lemma, $\ost\theta=\theta+W$.}
\glossrow{$w$}{Section ``Overall gameplan''}{The Newton perturbation in $\ost\theta=\theta+w+\Theta$.}
\glossrow{$w_J$}{Equation~\eqref{eq:NewtonEquation}}{The Newton correction with index $J=(k,n)$.}
\glossrow{$w_n$}{Before equation~\eqref{eq:gluingError2}}{The sum of Newton corrections at the $n$th Newton stage.}
\glossrow{$\bar w_J$, $\tilde w_J$}{Before equation~\eqref{eq:barwJ}}{The transported part and explicit time oscillatory part of $w_J$.}
\glossrow{$w_{0,(k,n)}$}{Equation~\eqref{eq:NewtonEquation}}{The initial data for $w_{(k,n)}$ at time $k\tau$.}
\glossrow{$\Xi$}{Definition~\ref{defn:freqEnLevels}}{The frequency level in the frequency energy bounds.}
\glossrow{$\widehat\Xi$}{Section ``Length scales and time scales''}{We put $\widehat\Xi=N^{1/L}\Xi=\widehat N\Xi$.}
\glossrow{$\ost\Xi$}{Main Lemma}{The updated frequency level $\ost\Xi=\widehat C N\Xi$.}
\glossrow{$\xi_I$}{Equation~\eqref{eq:transportPhase}}{The phase function transported by $\tilde u_\epsilon$ for the wave $\Theta_I$.}
\glossrow{$\check\xi_{(k,f)}$}{Equation~\eqref{eq:transportOfPhase}}{The phase function transported by $u_\epsilon$ during the Newton setup.}
\glossrow{$Y_n^i$}{Equation~\eqref{eq:antiDivMin1half}}{An anti-divergence satisfying $|\nabla|^{-1/2}W_{-1,n}=\nabla_iY_n^i$ in the h-principle proof.}
\glossrow{$Z$}{Proof of Theorem~\ref{thm:approxThem}}{A large constant used in the recursive definition $D_{R,(k+1)}=D_{R,(k)}^{1+\delta}/Z$.}
\glossrow{$z_J^i$}{Equation~\eqref{eq:plainForm}}{A first order anti-divergence for $w_J$, so that $\nabla_i z_J^i=w_J$.}
\glossrow{$\bar z_J^i$, $\tilde z_J^i$}{Before equation~\eqref{eq:barzJ}}{The transported part and explicit time oscillatory part of $z_J^i$.}
\glossrow{$|\nabla|$}{SQG equation in the Introduction}{The Fourier multiplier $|\nabla|=\sqrt{-\Delta}$.}
\glossrow{$|\vec a|$}{Section~\ref{sec:notation}}{The order of the multi-index $\vec a$.}
\glossrow{$(x)_+$}{Section~\ref{sec:notation}}{The positive part, $(x)_+=\max\{x,0\}$.}
\glossrow{$\tisum$}{Section~\ref{sec:notation}}{A sum with omitted combinatorial coefficients.}

\end{longtable}

\noindent\textbf{Parameter inequalities.}
The inequalities used repeatedly in the iteration are
\begin{align*}
\D_R &\leq \D_u, & e_R &\leq e_u, \\
\Xi e_u^{1/2} &\leq \tau^{-1} \leq \mu \leq \epsilon_t^{-1} \leq \ost\tau^{-1}, & \Xi &\leq \widehat\Xi \leq \lambda, \\
1 &\leq (D_u/D_R)\leq \widehat C(D_u/D_R)(N\Xi)^{4\eta}N^{6/L} \leq N.
\end{align*}

\endgroup

\subsection{Estimates cheat sheet} \label{sec:estimatesCheatSheet}

This table is meant as a reading guide for the estimates used in the proof.  We put $\unl=L-3$, $\hn=N^{1/L}$, $\hxi=\hn\Xi$, and $\la \in 2\pi \Z, \la \sim N\Xi$.  Unless a row says otherwise, the implicit constants may depend on $L$, $\eta$, the derivative order, and the fixed cutoffs, but not on $N$, $\Xi$, $D_u$, or $D_R$.  The entries are alphabetized by the left column heading in the alphabetical order for the symbol; for example, $\mathcal B$ appears before $\gamma$. (Only a few exceptions to this alphabetical ordering occur, for example in ``Difference of phase gradients'', for pedagogical reasons.)

\begingroup
\footnotesize
\setlength{\LTpre}{4pt}
\setlength{\LTpost}{4pt}
\setlength{\tabcolsep}{3pt}
\renewcommand{\arraystretch}{1.18}
\begin{longtable}{@{}>{\raggedright\arraybackslash}p{0.27\textwidth}>{\raggedright\arraybackslash}p{0.69\textwidth}@{}}
\textbf{Quantity} & \textbf{Estimate and location}\\
\hline
\endfirsthead
\textbf{Quantity} & \textbf{Estimate and location}\\
\hline
\endhead
\hline
\multicolumn{2}{r@{}}{\footnotesize continued on next page}\\
\endfoot
\hline
\endlastfoot

Bilinear anti-divergence $\BB_a^{ij}$ & In the proof of Proposition~\ref{prop:boiledDownBound} for $\bar r_J$, the high high, high low, and low high pieces of $\BB_a^{ij}[z_J^a,\theta_\epsilon]$ are bounded at size $\hn^{(|\vec a|-\unl)_+}\Xi^{|\vec a|}\tau^{-1}S_r(1+\hh)$.  The same proof gives the high frequency version with the additional factor $2^{-\alpha q}\hxi^\alpha$ after applying $P_q$.  See the paragraph ``The $\BB$ term.''\\

Lower order wave terms $\delta\theta_I,\delta u_I$ & For $0\leq r\leq1$, the microlocal remainders satisfy $\|\nabla_{\vec a}\Dtdt^r\delta\theta_I\|_{C^0}+\|\nabla_{\vec a}\Dtdt^r\delta u_I\|_{C^0}\lesssim \lambda^{1/2}N^{-1}\hn^{(|\vec a|+2-\unl)_+}\Xi^{|\vec a|}\tau^{-r}D_{R,n}^{1/2}$.  For two material derivatives, replace $\tau^{-r}D_{R,n}^{1/2}$ by $\epsilon_t^{-1}(\Xi e_u^{1/2})D_{R,n}^{1/2}$.  See \eqref{eq:deBounds1} and \eqref{eq:deBounds2}.\\

Phase exponential $e^{i\lambda\xi_I}$ & The final weighted norm satisfies $\ost H^*[e^{i\lambda\xi_I}]\lesssim 1$.  This estimate is used in the nonstationary phase parametrix.  See \eqref{eq:sizeOfPhase}.\\

Time envelopes $e_n$ & The envelope satisfies $\|d^r e_n^{1/2}/dt^r\|_{C^0}\lesssim \tau^{-r}D_{R,n}^{1/2}$.  The Newton stress scale is $D_{R,n}=(N^{-\eta}\Xi^{-\eta})^nD_{R,0}$.  See \eqref{eq:enConstruct} and \eqref{eq:DRn}.\\

Amplitudes $\gamma_I$ & For $0\leq r\leq1$, $\|\nabla_{\vec a}\Dtdt^r\gamma_I\|_{C^0}\lesssim \hn^{(|\vec a|+1-\unl)_+}\Xi^{|\vec a|}\tau^{-r}D_{R,n}^{1/2}$.  Also $\|\nabla_{\vec a}\Dtdt^2\gamma_I\|_{C^0}\lesssim \hn^{(|\vec a|+2-\unl)_+}\Xi^{|\vec a|}\epsilon_t^{-1}(\Xi e_u^{1/2})D_{R,n}^{1/2}$.  See Proposition~\ref{prop:amplitudeBounds}, especially \eqref{eq:firstAdvecAmplitudes}.\\

Final weighted norms $H^*$ & The norms associated to $\partial_t$, $\Ddt$, $\Dtdt$, and $\ost D_t$ are comparable: $\ost H^*[F]\lesssim H^*[F]\lesssim \bar H^*[F]\lesssim \tilde H^*[F]\lesssim \ost H^*[F]$.  They also satisfy $\mrg H^*[FG]\lesssim \mrg H^*[F]\mrg H^*[G]$.  See Proposition~\ref{prop:finalComp} and \eqref{eq:comparable}.\\

Mixed derivatives $H_{\zeta,M}^{(R)}$ & If $\zeta\geq \Xi e_u^{1/2}$, then $\|\nabla_{\vec a}\Ddt\nabla_{\vec b}F\|_{C^0}\lesssim \hn^{(|\vec a|+|\vec b|+1-L)_+}\Xi^{|\vec a|+|\vec b|}\zeta H_{\zeta,M}^{(1)}[F]$.  See Lemma~\ref{lem:mixed derivatives}.\\

Weighted norms $\mrg H_\zeta$ & For $\mrg D_t\in\{\partial_t,\Ddt,\Dtdt\}$, the weighted norms are comparable, and the product and chain rules hold: $\mrg H_\zeta[FG]\lesssim \mrg H_\zeta[F]\mrg H_\zeta[G]$ and $\mrg H_\zeta[G(F)]\lesssim (1+\mrg H_\zeta[F])^{L'}$.  See Lemma~\ref{lem:interchangeDerivs} and Proposition~\ref{prop:chainRule}.\\

Parameter hierarchy & We use $\mu=N^{1/2}\Xi e_R^{1/2}$, $\tau=b(\log\hxi)^{-1}(\Xi e_u^{1/2})^{-1}$, $\widetilde e_u^{1/2}=(\log\widehat\Xi)e_u^{1/2}$ with $\Xi\widetilde e_u^{1/2}=b\tau^{-1}$, $\epsilon_t=(\Xi e_u^{1/2})^{-1}(D_u/D_R)^{-1/2}N^{-1/2}$, and $\ost\tau^{-1}=(N\Xi)^{3/2}D_R^{1/2}$.  The useful chain is $\Xi e_u^{1/2}\leq \tau^{-1}\leq \mu\leq \epsilon_t^{-1}\leq \ost\tau^{-1}$.  See \eqref{eq:natTimescale}--\eqref{eq:ept}, Proposition~\ref{prop:RmRqnewt}, and the parameter list in the glossary.\\

Frequency growth parameter $N$ & The Main Lemma assumes $N\geq \hc N^{6/L}N^{4\eta}\Xi^{4\eta}(D_u/D_R)$.  This lower bound is used when absorbing the Newton gluing error and later when comparing the final time scale.  See \eqref{eq:NlowerBd}.\\

Difference of phase gradients $\nabla\check\xi-\nabla\xi$ & The Newton and convex integration phase gradients satisfy $\|\nabla\check\xi(t)-\nabla\xi_I(t)\|_{C^0}\lesssim N^{-1/2}(D_u/D_R)^{-1/2}$ on the support of the wave.  See \eqref{eq:flowErrorFirst}.\\

Nonstationary phase $Q^{j\ell}$ & If $G=e^{i\lambda\xi_I}g$ has mean zero, then there is a traceless symmetric $Q^{j\ell}$ with $\nabla_j\nabla_\ell Q^{j\ell}=G$ and $\tilde H^*[Q]\leq C_D((N\Xi)^{-2}+N^{-D/2})\tilde H''_{2D+L}[g]$.  See Proposition ``Nonstationary phase'' and \eqref{eq:statPhaseBd}.\\

Zero order multipliers and commutators & A degree zero multiplier $\QQ$ obeys $\|\nabla_{\vec a}P_q\QQ F\|_{C^0}\lesssim 2^{-\alpha q}\|\nabla_{\vec a}F\|_{\dot C^\alpha}$ and $\|\nabla_{\vec a}\QQ F\|_{C^0}\lesssim (\log\hxi)\|\nabla_{\vec a}F\|_{C^0}+\hxi^{-\alpha}\|\nabla_{\vec a}F\|_{\dot C^\alpha}$.  The commutator with $\Ddt$ is estimated in Lemma~\ref{lem:zeroOrder}.  See Proposition~\ref{prop:standardDegZero} and Lemma~\ref{lem:zeroOrder}.\\

Coarse stress $R_\epsilon$ & The spatially mollified stress satisfies $\|R-R_\epsilon\|_{C^0}\lesssim N^{-1}D_R$, $\|\nabla_{\vec a}R_\epsilon\|_{C^0}\lesssim_{\vec a}\hn^{(|\vec a|-L)_+}\Xi^{|\vec a|}D_R$, and $\|\nabla_{\vec a}\Ddt R_\epsilon\|_{C^0}\lesssim_{\vec a}(\Xi e_u^{1/2})\hn^{(|\vec a|+1-L)_+}\Xi^{|\vec a|}D_R$.  See the end of the subsection ``Regularizing the scalar field and error tensor.''\\

Gluing stress $R_{(n+1)}$ & The Newton gluing stress satisfies $\|R_{(n+1)}\|_{C^0}\leq D_{R,n+1}$, $\|\nabla_{\vec a}R_{(n+1)}\|_{C^0}\lesssim_{\vec a}\hn^{(|\vec a|-L)_+}\Xi^{|\vec a|}D_{R,n+1}$, and $\|\nabla_{\vec a}\Ddt R_{(n+1)}\|_{C^0}\lesssim_{\vec a}\hn^{(|\vec a|+1-L)_+}\Xi^{|\vec a|}(\Xi e_u^{1/2})D_{R,n+1}$.  See Proposition~\ref{prop:newtonBound} and \eqref{eq:RnC0}--\eqref{eq:Rnplus1Ddt}.\\

High frequency error $R_H$ & After nonstationary phase, $\tilde H^*[R_H]\lesssim (\lambda^{-2}+N^{-D/2})(\lambda\Xi D_R)$.  With $D$ chosen depending on $\eta$, this gives $\tilde H^*[R_H]\lesssim D_R/N$.  See Section~\ref{sec:highFreqErrors}.\\

Newton quadratic and mollification errors $R_{Q,n},R_{M,n}$ & The Newton errors satisfy $\|\nabla_{\vec a}R_{M,n}\|_{C^0}+\|\nabla_{\vec a}R_{Q,n}\|_{C^0}\lesssim (N\Xi)^{|\vec a|}N^{-1}D_R$ and the corresponding pure time derivative bounds cost $\ost\tau^{-1}$.  See Proposition~\ref{prop:RmRqnewt}.\\

Stress error $R_S$ & The flow and stress error obeys $\ost H^*[R_S]\lesssim (D_u/D_R)^{-1/2}N^{-1/2}D_R$.  The microlocal remainder $\delta B_I$ contributes $N^{-1}D_R$.  See Section~\ref{sec:estimateRS}.\\

Transport error $R_T$ & The transport error satisfies $\tilde H^*[R_T]\lesssim (N\Xi)^{-2}\mu\lambda^{1/2}D_R^{1/2}\sim N^{-1}D_R$.  See Section~\ref{sec:highFreqErrors}.\\

Mollified stress $\tilde R_n$ & Flow mollification gives $\|R_{(n)}-\tilde R_n\|_{C^0}\lesssim N^{-1/2}(D_u/D_R)^{-1/2}D_{R,n}$, $\|\nabla_{\vec a}\tilde R_n\|_{C^0}\lesssim_{\vec a}\hn^{(|\vec a|-\unl)_+}\Xi^{|\vec a|}D_{R,n}$, $\|\nabla_{\vec a}\Dtdt\tilde R_n\|_{C^0}\lesssim_{\vec a}(\Xi e_u^{1/2})\hn^{(|\vec a|+1-\unl)_+}\Xi^{|\vec a|}D_{R,n}$, and the second material derivative costs $\epsilon_t^{-1}(\Xi e_u^{1/2})\hn^{(|\vec a|+1-\unl)_+}\Xi^{|\vec a|}D_{R,n}$.  See Proposition~\ref{prop:stress} and \eqref{eq:RerrorMollPhi}.\\

Error list for $\ost R$ & After adding $\Theta$, the new stress is decomposed as $\ost R=R_T+R_H+R_M+R_S+P_{(\Gamma)}+R_{(\Gamma)}$.  The defining equations for $R_T$, $R_H$, $R_M$, and $R_S$ appear in \eqref{eq:errorList}.  The Summary Section before the Overall gameplan lists where the estimates for these pieces are checked.\\

Trace free Newton anti-divergence $\rho_J$ & On $\operatorname{supp}\tilde\chi_k'(t)$, $\|\nabla_{\vec a}\Ddt^r\rho_J\|_{C^0}\lesssim_{\vec a}\hn^{(|\vec a|+r-\unl)_+}\Xi^{|\vec a|}\tau^{-r}S_\rho$ for $0\leq r\leq 1$.  See \eqref{eq:newtTrfreeBd}.\\

Size table $S_F$ & The assigned sizes are $S_w=\Xi^2\mu^{-1}D_{R,n}$, $S_z=\Xi\mu^{-1}D_{R,n}$, $S_r=\mu^{-1}D_{R,n}$, $S_u=(\log\hxi)\Xi^2\mu^{-1}D_{R,n}$, and $S_\rho=(\log\hxi)\mu^{-1}D_{R,n}$.  See \eqref{eq:sizeOfTerms}.\\

Microlocal wave $\Theta_I$ & For each wave, $\|\nabla_{\vec a}\Theta_I\|_{C^0}+\|\nabla_{\vec a}T^\ell\Theta_I\|_{C^0}\lesssim_{\vec a}\lambda^{|\vec a|+1/2}D_R^{1/2}$ and $\|\nabla_{\vec a}|\nabla|^{-1/2}\Theta_I\|_{C^0}\lesssim_{\vec a}\lambda^{|\vec a|}D_R^{1/2}$.  See \eqref{eq:bdFornbaThTTh} and \eqref{eq:nbminhalfTh}.\\

Coarse field $\theta_\epsilon,u_\epsilon$ & For the spatial mollification scale $\epsilon=N^{-1/L}\Xi^{-1}$, $\|\nabla_{\vec a}\theta_\epsilon\|_{C^0}+\|\nabla_{\vec a}u_\epsilon\|_{C^0}\lesssim_{\vec a}\hn^{(|\vec a|-L)_+}\Xi^{|\vec a|}e_u^{1/2}$, and $\|\nabla_{\vec a}\Ddt\theta_\epsilon\|_{C^0}+\|\nabla_{\vec a}\Ddt u_\epsilon\|_{C^0}\lesssim_{\vec a}\hn^{(|\vec a|+1-L)_+}\Xi^{|\vec a|+1}e_u$.  See \eqref{eq:prelimScalVel1} and \eqref{eq:prelimScalVel2}.\\

Difference $\theta-\theta_\epsilon,u-u_\epsilon$ & For $0\leq |\vec a|+r\leq L$ and $0\leq r\leq 1$, $\|\nabla_{\vec a}\partial_t^r(\theta-\theta_\epsilon)\|_{C^0}+\|\nabla_{\vec a}\partial_t^r(u-u_\epsilon)\|_{C^0}\lesssim (N\Xi)^{|\vec a|}(\ost\Xi\ost e_u^{1/2})^rN^{-1}e_u^{1/2}$.  See Lemma~\ref{lem:new}.\\

Final fields $\ost\theta,\ost u$ & The final weighted norm obeys $\ost H^*[\ost u]\lesssim (N\Xi)^{1/2}D_R^{1/2}$ and $\ost H^*[\ost\theta]\lesssim (N\Xi)^{1/2}D_R^{1/2}$.  See Proposition~\ref{prop:newVelocBound}.\\

Regularity iteration & The limiting construction gives $\||\nabla|^{-1/2}(\theta-\theta_0)\|_{C^0}\leq C_\delta D_{R,0}^{1/2}$ and $|\nabla|^{-1/2}\theta\in C^{1/2-2\delta}$.  The parameter recursion uses \eqref{eq:NkFormula}.  See Theorem~\ref{thm:approxThem}.\\

Newton velocity $u_J,u_n$ & The velocity produced by a Newton correction satisfies $\|\nabla_{\vec a}u_n\|_{C^0}\lesssim_{\vec a}\hn^{(|\vec a|-\unl)_+}\Xi^{|\vec a|}S_u$.  Also $\|\nabla_{\vec a}\Ddt^r u_n\|_{C^0}\lesssim_{\vec a}\hn^{(r+|\vec a|-\unl)_+}\Xi^{|\vec a|}(\Xi e_u^{1/2})^r e_u^{1/2}$ for $0\leq r\leq 1$.  See \eqref{eq:newtVelocBd} and \eqref{eq:newtVelocBd2}.\\

Main Lemma correction $W$ & The correction satisfies $\|\nabla_{\vec a}|\nabla|^{-1/2}W\|_{C^0}\leq \hc(N\Xi)^{|\vec a|}D_R^{1/2}$ for $|\vec a|=0,1$.  Its time support obeys $\operatorname{supp}_t W\cup\operatorname{supp}_t\ost R\subseteq N(J_0)$.  See the Main Lemma, \eqref{eq:LaWBd}, and \eqref{eq:supportEnlargement}.\\

Newton fields $w_J,z_J,r_J$ & For $F\in {\bf F}_J\cup\{w_J,z_J,r_J\}$, $\|\nabla_{\vec a}F\|_{C^0}\lesssim_{\vec a}\hn^{(|\vec a|-\unl)_+}\Xi^{|\vec a|}S_F$.  See Proposition~\ref{prop:newtonBound} and \eqref{eq:newtonStepBoundnbaF}.\\

Newton perturbation $w_n$ & The Newton correction satisfies $\|\nabla_{\vec a}|\nabla|^{-1/2}w_n\|_{C^0}\lesssim \Xi^{|\vec a|}D_{R,n}^{1/2}$ for $0\leq |\vec a|\leq 1$.  The rough pure time bounds are $\|\nabla_{\vec a}\partial_t^r w_n\|_{C^0}+\Xi\|\nabla_{\vec a}\partial_t^r z_n\|_{C^0}\lesssim (N\Xi)^{|\vec a|}\mu^rS_w$.  See \eqref{eq:sizeOfNewtCorrect} and \eqref{eq:lossyBdwn}.\\

Phase function $\check\xi_J$ & The Newton phase gradients satisfy $\|\nabla_{\vec a}\nabla\check\xi_J\|_{C^0}\lesssim_{\vec a}\hn^{(|\vec a|+1-L)_+}\Xi^{|\vec a|}$ and $\|\nabla_{\vec a}\Ddt\nabla\check\xi_J\|_{C^0}\lesssim_{\vec a}\hn^{(|\vec a|+1-L)_+}\Xi^{|\vec a|+1}e_u^{1/2}$.  See Proposition~\ref{prop:phasegradients}.\\

Phase function $\xi_I$ & The convex integration phases satisfy $\|\nabla_{\vec a}\nabla\xi_I\|_{C^0}\lesssim_{\vec a}\hn^{(|\vec a|+1-\unl)_+}\Xi^{|\vec a|}$, $\|\nabla_{\vec a}\Dtdt\nabla\xi_I\|_{C^0}\lesssim_{\vec a}\hn^{(|\vec a|+1-\unl)_+}\Xi^{|\vec a|}(\Xi e_u^{1/2})$, and a second advective derivative bound with $\hn^{(|\vec a|+2-\unl)_+}(\Xi e_u^{1/2})^2$.  See Proposition~\ref{prop:phase}.\\

\end{longtable}
\endgroup

\bibliographystyle{abbrv}
\bibliography{eulerOnRn}

\end{document}